\documentclass{amsart}

\usepackage[utf8]{inputenc}

\usepackage{geometry}
\usepackage{epsdice}

\usepackage{enumitem}

\usepackage{color}

\usepackage[OT2,T1]{fontenc}

\usepackage[final]{microtype}

\usepackage[normalem]{ulem}

\usepackage{tikz}
\usepackage{tikz-cd}
\usepackage{xypic}

\usepackage[draft=false,linktocpage=true]{hyperref}
\usepackage[capitalise]{cleveref}

\usepackage{relsize}
\usepackage[bbgreekl]{mathbbol}
\usepackage{amsfonts}
\DeclareSymbolFontAlphabet{\mathbb}{AMSb} 
\DeclareSymbolFontAlphabet{\mathbbl}{bbold}
\newcommand{\Prism}{{\mathlarger{\mathbbl{\Delta}}}} 

\makeatletter
\newsavebox{\@brx}
\newcommand{\llangle}[1][]{\savebox{\@brx}{\(\m@th{#1\langle}\)}%
  \mathopen{\copy\@brx\kern-0.5\wd\@brx\usebox{\@brx}}}
\newcommand{\rrangle}[1][]{\savebox{\@brx}{\(\m@th{#1\rangle}\)}%
  \mathclose{\copy\@brx\kern-0.5\wd\@brx\usebox{\@brx}}}
\makeatother

\usepackage{mathtools}
\usepackage{amsmath}
\usepackage{mathrsfs}

\usepackage{amsthm,amssymb}
\numberwithin{equation}{section}

\theoremstyle{plain}
\newtheorem{theorem}[equation]{Theorem}

\newtheorem{proposition}[equation]{Proposition}
\newtheorem{lemma}[equation]{Lemma}
\newtheorem{claim}[equation]{Claim}
\newtheorem{corollary}[equation]{Corollary}

\theoremstyle{definition}
\newtheorem{definition}[equation]{Definition}
\newtheorem{construction}[equation]{Construction}

\newtheorem{example}[equation]{Example}

\newtheorem{remark}[equation]{Remark}

\newtheorem{warning}[equation]{Warning}

\makeatletter

\newcommand{\Rmnum}[1]{\expandafter\@slowromancap\romannumeral #1@}
\makeatother

\DeclareMathOperator{\gr}{gr}
\DeclareMathOperator{\id}{id}

\DeclareMathOperator{\Sym}{Sym}

\DeclareMathOperator{\Gal}{Gal}

\DeclareMathOperator{\Spec}{Spec}
\DeclareMathOperator{\Spf}{Spf}

\DeclareMathOperator{\coker}{coker}
\DeclareMathOperator{\et}{\acute{e}t}

\DeclareMathOperator{\dR}{dR}

\DeclareMathOperator{\cris}{crys}

\DeclareMathOperator{\Mod}{Mod}

\DeclareMathOperator{\Fil}{Fil}

\DeclareMathOperator{\st}{st}

\DeclareMathOperator{\tor}{tor}

\DeclareMathOperator{\Tor}{Tor}

\DeclareMathOperator{\rank}{rank}
\DeclareMathOperator{\BK}{BK}
\DeclareMathOperator{\qsyn}{{qSyn}}
\DeclareMathOperator{\RG}{{R\Gamma}}
\DeclareMathOperator{\Len}{{length}}

\newcommand{\FM}{\mathfrak{M}}

\newcommand{\LL}{{\mathbb{L}}}

\newcommand{\NN}{{\mathbb{N}}}

\newcommand*{\Z}{\ensuremath{\mathbb{Z}}}

\newcommand*{\Kbar}{\overline{K}}

\newcommand*{\m}{\mathfrak{M}}
\newcommand*{\n}{\mathfrak{N}}

\newcommand*{\s}{\mathfrak{S}}

\newcommand*{\C}{\mathbf{C}}

\newcommand*{\calM}{\mathcal{M}}
\renewcommand*{\O}{\mathcal{O}}

\newcommand{\rH}{{\rm H}}

\renewcommand*{\int}{\ensuremath{\mathrm{int}}}

\renewcommand*{\u}[1]{\underline{#1}}
\renewcommand*{\o}[1]{\overline{#1}}
\newcommand*{\wh}[1]{\widehat{#1}}

\newcommand*{\inj}{\hookrightarrow}
\newcommand*{\onto}{\twoheadrightarrow}

\renewcommand{\tilde}{\widetilde}
\renewcommand{\bar}{\overline}
\newcommand{\BL}{{\mathbb L}}

\newcommand{\ku}{{k[\![u]\!]}}
\newcommand{\gt}{{\mathfrak t}}

\newcommand{\cI}{\mathcal I}

\DeclareMathOperator{\fib}{fib}

\newcommand{\ue}{\underline{\varepsilon}}
\newcommand{\upi}{\underline{\pi}}

\date{First version December 8, 2020; Revised version \today.}

\begin{document}

\title{Comparison of prismatic cohomology and derived de Rham cohomology}

\author{Shizhang Li}

\author{Tong Liu}




\begin{abstract}
We establish a comparison isomorphism between prismatic cohomology and derived de Rham cohomology respecting various structures, 
such as their Frobenius actions and filtrations.
As an application, when $X$ is a proper smooth formal scheme over $\O_K$ with $K$ being a $p$-adic field, 
we improve Breuil--Caruso's theory on comparison between torsion crystalline cohomology and torsion \'etale cohomology. 
\end{abstract}


\maketitle

\tableofcontents


\section{Introduction}\label{Intro}
Let $k$ be a perfect field of characteristic $p>0$ and $K$ a totally
ramified degree $e$ field extension of $W(k)[1/p]$. 
Fix an algebraic closure $\o{K}$ of $K$,
denote its $p$-adic completion by $\C$, and use $\O_{\C}$ to denote its ring of integers.
Let $X$  be a smooth proper formal scheme over $\O_K$
with (rigid analytic) geometric generic fiber $X_{\bar \eta}$. Write $X_n : = X \times_{\Z} \Z/ p ^n \Z$. 
Starting from \cite{FontaineMessing} and \cite{Katovanishingcycles}, lots of efforts have been made in investigating the relationship between the crystalline cohomology  (and other variants) and the \'etale cohomology attached to $X$.  

When $e=1 $,  it is proved by Fontaine--Messing (\cite{FontaineMessing}) and Kato (\cite{Katovanishingcycles})
that if  $X$ is a proper\footnote{Projective in Kato's paper.} smooth scheme over   $\O_K = W(k)$,
then $ \rH^i  _{\cris}(X_n / W_n (k) )$ admits a Fontaine--Laffaille module structure 
when $i \leq p -1$ and the functor $T_{\cris}$ on the category of Fontaine--Laffaille modules (from Fontaine--Laffaille theory)
satisfies $T_{\cris}(\rH ^i_{\cris} (X_n /W_n(k)))\simeq \rH ^i _{\et} (X_{\bar \eta}, \Z/ p ^n \Z)$  as $G_K$-modules when $i \leq p-2$.  

When $e> 1$, more complicated base ring has to be introduced. Fix a uniformizer $\pi $ of $K$ and  $ E= E(u)\in W(k)[u]$ the Eisenstein polynomial of $\pi$.  
Let $S$ be the $p$-adic completion of the PD envelope of $W(k)[u]$ for the ideal $(E) $. 
Note that $S$ admits:
\begin{itemize}
    \item a Frobenius action $\varphi: S\to S$ which extends the Frobenius $\varphi$ on $W(k)$ and satisfies $\varphi (u)= u ^p$;
    \item a filtration $\Fil ^i S $ which is the $p$-complete $i$-th PD ideal; and
    \item a monodromy operator $N :  S \to S$ via $N(f(u)) = \frac{df}{du} (-u)$.
\end{itemize}
In \cite{Bre98}, Breuil introduced the notion of a \emph{Breuil module} to describe the structure of $\rH ^i _{\cris} (X_n / S_n)$, 
and  constructed a functor $T_{\st, \star}$  from the category of Breuil modules to the category of $\Z_p$-representations of $G_K$.  
Here, a Breuil module is a datum consisting of a finite $S$-module $\calM$ together with a one-step filtration $\Fil^h\calM\subset \calM$, 
a ``divided Frobenius'' $\varphi_h: \Fil^h\calM \to \calM$, and a monodromy operator $N : \calM \to \calM$ which satisfies some conditions given in \S \ref{subsec-Breuilmodules}.   

Following ideas of Breuil, Caruso proved the following.

\begin{theorem}[\cite{CarusoInvent}]\label{thm-caruso} 
Let $X$ be a proper semi-stable scheme over $\O _K$. Then its log-crystalline cohomology $\rH ^i_{\rm log\text{-}crys} (X_n /S_n)$ has a Breuil module structure
and $T_{\st, \star  }(\rH ^i_{\rm log\text{-}crys} (X_n /S_n))\simeq \rH^i _{\et} (X_{\bar \eta} , \Z/ p ^n \Z) (i)$ 
as $G_K$-modules for $ e(i+1) < p -1$ if $n > 1$ and $ei <p-1$ for $n =1$. 
\end{theorem}

As new cohomology theories have been introduced in \cite{BMS1}, \cite{BMS2} and \cite{BS19},
it is natural to ask whether these new cohomology theories can recover the aforementioned results due to 
Fontaine--Messing, Breuil, and Caruso, and hopefully even improve these results. 
In this paper, we use these new cohomology theories, in particular, prismatic cohomology and derived de Rham cohomology, 
to study torsion crystalline cohomology, torsion \'etale cohomology, and their relationship. We obtain the following result: 
\begin{theorem}\label{Thm-intro-1}  
Let  $X$  be a smooth proper formal scheme over $\O_K$
with geometric generic fiber $X_{\bar \eta}$, and let $i$ be an integer satisfying $ei < p-1$. 
Then $\rH ^i _{\cris} (X_n / S_n)$ has a structure of Breuil modules and 
$T_{\st, \star}\left(\rH ^i _{\cris} (X_n /S _n) \right)\simeq \rH ^i _{\et}(X_{\bar \eta}, \Z/ p ^n \Z)(i)$  as $\Z_p [G_K ]$-modules. 
\end{theorem}

Here the additional data of the Breuil module structure is roughly given by the following:
\begin{itemize}
\item the filtration is given by the cohomology of the PD powers of a natural PD ideal sheaf $\cI_{\cris}$ 
on the crystalline site $\rH^i_{\cris}(X_n/S_n, \cI ^{[h]}_{\cris})$;
\item the $N$ is a disguise of the connection given by the crystal nature of crystalline cohomology; and
\item the divided Frobenius is induced by a natural map of (quasi-)syntomic sheaves.
\end{itemize}
From now on, when we talk about $\rH ^i _{\cris} (X_n / S_n)$, we always implicitly think of it carrying these additional data.

\begin{remark}
\label{rem-better-than-caruso} 
\leavevmode
\begin{enumerate}
    \item 
Let us highlight the difference between Caruso's results and our theorem above. 
\begin{enumerate} 
\item The $X$ in our theorem is a \emph{smooth} proper \emph{formal} scheme over $\O_K$,
whereas the $X$ in \cite{CarusoInvent} is a semi-stable $\O_K$-model of a smooth proper $K$-variety.
\item Our restriction on $e$ and $i$ is $ei < p -1$ for any $n$ while the restriction in \cite{CarusoInvent} is $ei < p-1$ for $n=1$ and $e(i +1) < p-1$ for $n > 1$. 
\end{enumerate}
\item We actually use another functor $T_S$ relating torsion crystalline and \'{e}tale cohomology in the above theorem. 
But $T_S$ and $T_{\st, \star}$ are essentially the same. See \S \ref{subsec-two-T}. 
\end{enumerate}
\end{remark}

Now let us discuss the strategy of this paper to see how prismatic cohomology and (derived)  de Rham cohomology come into the picture.
Let $\s= W(k)[\![u]\!]$ equipped with the Frobenius morphism $\varphi$ extending (arithmetic) Frobenius $\varphi$ on $W(k)$ and $\varphi (u) = u ^p$. 
Then $(\s , (E))$ is the so-called Breuil--Kisin prism. 
Classically, an \emph{(\'etale) Kisin module} of height $h$ is a finite $u$-torsion free $\s$-module $\m$, together with a semi-linear map 
$\varphi_\m \colon \m \to \m$ so that the cokernel of  $1\otimes\varphi_\m \colon \s \otimes_{\varphi, \s}\m \to \m$  is killed by $E^h$. 
By definition, $\varphi ^* \m \coloneqq \s \otimes_{\varphi, \s} \m$ admits a \emph{Breuil--Kisin (BK) filtration} 
$\Fil^h \varphi ^* \m \coloneqq (1\otimes \varphi_\m) ^{-1} (E^h \m)$, which plays an important technical role later. 
It is well-known that Kisin module theory is a powerful tool in \emph{abstract} integral $p$-adic Hodge theory: 
the study of   $\Z_p$-lattices in crystalline  (semi-stable) representations and their modulo $p^n$-representations, 
which can been seen as the arithmetic counterpart  of 
$\rH ^n _{\et} (X_{\bar \eta} , \Z_p)$ and $\rH ^i _{\et} (X_{\bar \eta}, \Z/ p ^n \Z)$. 
Also the relationship between Kisin modules, Galois representations and Breuil modules
are known in the abstract theory. 
In particular, the functor $\underline \calM \colon \m \mapsto  \underline \calM (\m) \coloneqq S \otimes_{\varphi, \s}\m$
sends a Kisin module $\m$ of height $h \leq p-1$ to a Breuil module (without $N$-structures) where 
\[\Fil^h \calM (\m) \coloneqq \{x \in \calM(\m)| (1\otimes \varphi_{\m}) (x) \in \Fil^h S \otimes_{\s}\m \}\subset \underline \calM(\m)\]
and $\varphi_h \colon \Fil^h \calM (\m) \overset {1\otimes \varphi_{\m} }{\longrightarrow} \Fil^h S \otimes _{\s} \m \overset{\varphi_h \otimes 1}\longrightarrow S \otimes_{\varphi, \s}\m = \underline \calM (\m)$ 
where $\varphi_h \colon \Fil ^h S \to S$ is defined by $\varphi_h (x) = \dfrac{\varphi (x) }{p ^h}$. 
See \S  \ref{subsec-Breuilmodules} for more details.

It turns out that prismatic cohomology $\rH ^i _{\Prism} (X /\s)$ gives geometric realizations of Kisin modules, 
in the sense that $\rH^i_{\Prism} (X/\s)$ modulo its $u^\infty$-torsion submodule is an \'etale Kisin module of height $i$
(see \S \ref{subsec-Galos rep and Kisin},  \S \ref{subsec-prism-is-kisin} and the discussion below for more details). 
Suggested by the functor $\underline\calM$ in the abstract theory, 
one naturally expects the following comparison between Breuil--Kisin prismatic cohomology and crystalline cohomology: 
\begin{equation}
\label{eqn-1}
\RG _{\Prism} (X/\s) \otimes^\BL_{\s, \varphi} S \simeq \RG_{\cris} (X/ S).    
\end{equation}
This comparison follows from \cite[Theorem 5.2]{BS19} and base change of prismatic cohomology.
This is pointed out to us by Koshikawa.

Inspired by the above discussion, we show in this paper the following comparison result:
\begin{theorem}[{see \Cref{comparing pris and crys} and \Cref{global comparison}}]
\label{comparison in introduction}
Let $(A,I)$ be a bounded prism, and let $X$ be a smooth proper ($p$-adic) formal scheme over $\Spf(A/I)$.
Then we have a functorial isomorphism
\[
\mathrm{R\Gamma}_{\Prism}(X/A) \otimes^\BL_{A, \varphi_A} A \otimes^\BL_A \dR_{(A/I)/A}^\wedge 
\cong \mathrm{R\Gamma}(X, \dR_{-/A}^\wedge),
\]
which is compatible with base change in the prism $(A,I)$.
\end{theorem}
Here $\dR_{-/A}^\wedge$ denotes the (relative to $A$) $p$-adic derived de Rham complex introduced by Illusie in \cite[Chapter VIII]{Ill72}
and studied extensively by Bhatt in \cite{Bha12}.
In fact, when $A/I$ is $p$-torsion free, this is known due to \cite[Theorem 5.2]{BS19}.
Our proof follows closely the proof of crystalline comparison in \cite{BS19}.

As a consequence, the above gives several comparison results, all of which were known due to work of 
Bhatt, Morrow, and Scholze \cite{BMS1}, \cite{BMS2}, \cite{BS19}.
\begin{example}
By \cite[Theorem 3.27]{Bha12}, we know that when $A/I$ is $p$-torsion free, the derived de Rham
complex appearing above is given by certain crystalline cohomology.
With this being said, we can explain what the above comparison gives in concrete situations.
\begin{enumerate}
\item BMS2/Breuil--Kisin prism: when $(A,I) = (\s, (E))$, then the above comparison becomes \Cref{eqn-1} which,
as mentioned above, was obtained in \cite{BS19}.
As a consequence, we see that Breuil's crystalline cohomology groups $\rH^i_{\cris}(X/S)$ are finitely presented $S$-module,
see \Cref{fp proposition}.
To the best of our knowledge, coherence of $S$ is unknown, and we are unaware of any other means showing that these cohomology groups
are finitely presented.
We thank Bhatt for pointing out this application to us.
\item BMS1: when $(A,I) = (A_{\inf}, \ker(\theta))$ is the perfect prism associated with $\mathcal{O}_{\C}$, then the above comparison
says
\[
\mathrm{R\Gamma}_{\Prism}(X/A_{\inf}) \otimes^\BL_{A_{\inf}, \varphi} A_{\inf} \otimes^\BL_{A_{\inf}} A_{\cris}
\cong \mathrm{R\Gamma_{\cris}}(X/A_{\inf}).
\]
Recall \cite[Theorem 17.2]{BS19} states that the first base change of the left hand side gives the $A_{\inf}$-cohomology theory constructed
in \cite{BMS1}. Then our comparison here becomes the one established by \cite[Theorem 1.8.(iii)]{BMS1} (see also \cite{Yao19}).
\item PD prism: suppose $I \subset A$ admits a PD structure $\gamma$. Then our comparison implies
\[
\mathrm{R\Gamma}_{\Prism}(X/A) \otimes^\BL_{A, \varphi_A} A
\cong \mathrm{R\Gamma_{\cris}}(X/(A, I, \gamma)).
\]
When $I = (p)$, then the above is nothing but the crystalline comparison established in \cite[Theorem 1.8.(1)]{BS19}.
Notice here the left hand side does not depend on the choice of $\gamma$, consequently neither does the right hand side.
Another class of potentially interesting PD prisms consists of $(W(S), V(1))$ for any bounded $p$-complete ring $S$.
\item de Rham comparison: there is a natural map $\gr^0 \colon \dR_{R/A}^\wedge \to R^\wedge$ given by ``quotient out'' first Hodge filtration.
Our comparison result above, after composing with this further base change, gives
\[
\mathrm{R\Gamma}_{\Prism}(X/A) \otimes^\BL_{A, \varphi_A} A \otimes^\BL_A A/I
\cong \mathrm{R\Gamma_{dR}}(X/(A/I))^\wedge;
\]
here, we have used \cite[Proposition 3.11]{GL20} to identify the result of right hand side under this base change.
This is the de Rham comparison given by \cite[Theorem 1.8.(3)]{BS19}.
\end{enumerate}
\end{example}

In (1)-(3) above, the crystalline comparison \cite[Theorem 5.2]{BS19} also yields comparison isomorphisms.
Note that there are at least two comparison isomorphisms in above discussion,
and we just claimed that they give rise to commutative diagrams, which might worry some readers.
To assure these readers, we establish the following rigidity of $p$-adic derived de Rham cohomology theory.
\begin{theorem}[{see \Cref{functorial endomorphism theorem} and \Cref{functorial endo remark}}]
Let $(A,I)$ be a prism such that $A/I$ is $p$-torsion free.
Then the functor $R \mapsto \dR_{R/A}^\wedge$ from the category of smooth $(A/I)$-algebras to $\mathrm{CAlg}(D(\dR_{(A/I)/A}^\wedge))$
has no automorphism.
Similar statement holds for the functor $R \mapsto \dR_{R/(A/I)}^\wedge$.
\end{theorem}
Therefore, whenever one has a diagram of functorial comparisons between various cohomology theories and the $p$-adic derived de Rham
cohomology, the diagram is always forced to be commutative.
Our method of proving such rigidity is largely inspired by \cite[Sections 10.3 and 10.4]{BLM18} and \cite[Section 18]{BS19}.
In view of rigidity aspects of $p$-adic derived de Rham complexes, 
we would like to mention a recent result of Mondal \cite{Mon20}: 
roughly speaking, there is a \emph{unique} deformation of de Rham cohomology
from characteristic $p$ to Artinian local rings given by crystalline cohomology (c.f.~\cite[Theorem 10.1.2]{BLM18} for the case of deformation over $\mathbb{Z}_p$).
Let us mention that a recent collaboration between Mondal and the first named author \cite{LM21} computes
endomorphisms of $p$-adic derived de Rham cohomology in various $p$-adic settings.

Next, we discuss compatibility of additional structures on both sides being compared in \Cref{comparison in introduction}, most notably the Frobenius action
and filtration.
In \S \ref{Frobenii} we define a natural Frobenius action on $p$-adic derived de Rham complex assuming the base ring $A$ is a $p$-torsion free
$\delta$-ring. 
Therefore the right hand side is equipped with a Frobenius action.
The left hand side admits a Frobenius action as well,
by extending the Frobenius action on prismatic cohomology, as $A \to \dR_{(A/I)/A}^\wedge$ is compatible with Frobenii on them.
The two Frobenii on two sides in \Cref{comparison in introduction} agree when $A$ is $p$-torsion free, see \Cref{compatible with Frobenii}.
Let us remark that these $p$-torsion free conditions most likely can be relaxed, with extra work in developing
the theory of ``derived $\delta$-rings''. We  expect the above theoretical results to hold verbatim.

The story of comparing filtrations is our main new contribution to this theory and it is quite involved. Let us rewrite the comparison:
\[
\varphi^*\mathrm{R\Gamma}_{\Prism}(X/A) \otimes^\BL_A \dR_{(A/I)/A}^\wedge 
\cong \mathrm{R\Gamma}(X, \dR_{-/A}^\wedge).
\]
There are $3$ natural filtrations here: 
\begin{itemize}
    \item the Nygaard filtration $\Fil^{\bullet}_{\mathrm{N}}(\Prism^{(1)}_{-/A})$ on $\varphi^*\mathrm{R\Gamma}_{\Prism}(X/A)$, see~\cite[Section 15]{BS19};
    \item the $I$-adic filtration on $A$; and
    \item the Hodge filtration $\Fil^{\bullet}_{\mathrm{H}}(\dR_{-/A}^\wedge)$ on $\dR_{(A/I)/A}^\wedge$ and $\mathrm{R\Gamma}(X, \dR_{-/A}^\wedge)$.
\end{itemize}
They are related in the following fashion.
\begin{theorem}[{see \Cref{smooth H and N filtration}}]
Let $(A,I)$ be a prism such that $A/I$ is $p$-torsion free, and let $X$ be a smooth proper ($p$-adic) formal scheme over $\Spf(A/I)$.
\begin{enumerate}
\item The isomorphism in \Cref{comparison in introduction} refines to a filtered isomorphism:
\[
\bigg(\mathrm{R\Gamma}(X, \Fil^{\bullet}_{\mathrm{N}}(\Prism^{(1)}_{-/A}))\bigg) 
\widehat{\otimes}^\BL_{(A, I^{\bullet})} (\mathcal{A}, \mathcal{I}^{[\bullet]})
\cong \mathrm{R\Gamma}(X, \Fil^{\bullet}_{\mathrm{H}}(\dR_{-/A}^\wedge)),
\]
where the left hand side denotes the $p$-complete derived tensor of filtered objects over the filtered ring $(A, I^{\bullet})$
provided by the lax symmetric monoidal structure on the filtered derived infinity category.

In particular, we obtain a graded isomorphism between graded algebras:
\[
\bigg(\gr^*_{\mathrm{N}}\mathrm{R\Gamma}(X, \Prism^{(1)}_{-/A}) \widehat{\otimes}^\BL_{\Sym^*_{A/I}(I/I^2)} \Gamma^*_{A/I}(I/I^2)\bigg)
\cong \bigg(\gr^*_{\mathrm{H}}\mathrm{R\Gamma}(X, \dR_{-/A}^\wedge)\bigg).
\]
\item The isomorphism in \Cref{comparison in introduction} induces natural isomorphisms:
\[
\mathrm{R\Gamma}(X, \Prism^{(1)}_{-/A}/\Fil^i_{\mathrm{N}}) \cong
\mathrm{R\Gamma}(X, \dR_{-/A}^\wedge/\Fil^i_{\mathrm{H}})
\]
for all $i \leq p$.
\end{enumerate}
Moreover these isomorphisms are functorial in $X$ and $A$.
\end{theorem}
Here $\mathcal{A}$ denotes the $p$-adic PD envelope of $A \twoheadrightarrow A/I$, and $\mathcal{I}^{\bullet}$ denotes the
filtration of PD powers of the ideal $\ker(\mathcal{A} \twoheadrightarrow A/I)$.
For a (somewhat) concrete description of the filtration on $\varphi^*\mathrm{R\Gamma}_{\Prism}(X/A) \otimes^\BL_A \dR_{(A/I)/A}^\wedge$
appearing in the above Theorem, we refer readers to \cite[Construction 3.9]{GL20}.
Note that by combining aforesaid result of Bhatt \cite[Theorem 3.27]{Bha12} and a classical result of Illusie \cite[Corollaire VIII.2.2.8]{Ill72},
there is a natural filtered isomorphism $(\dR_{(A/I)/A}^\wedge, \Fil^{\bullet}_{\mathrm{H}}) \cong (\mathcal{A}, \mathcal{I}^{\bullet})$.

In \cite[Section 2]{Ill20}, a comparison of Nygaard and Hodge filtration
is established for crystalline base prism $(A,I) = (W(k), (p))$, in particular
his $A/I$ is entirely $p$-torsion.
It seems reasonable to expect the comparison of filtration holds for general
base prisms.
Both Bhatt and Illusie have sketched to us an approach
of resolving base prism by prisms $(A,I)$ such that $A/I$ is $p$-torsion free,
to reduce the general comparison to our theorem above.
We choose to not pursue that direction further in this paper, as our final
application only uses the comparison when the base is the Breuil--Kisin prism.

With the above general preparation, we are ready to show $\underline{\calM} (\rH ^i _{\Prism} (X/\s))\simeq \rH^i _{\cris} (X/S) $ (when $\rH^i_{\Prism} (X/\s)$ is $u$-torsion free). 
In order to treat $p ^n$-torsion cohomologies in Theorem \ref{Thm-intro-1}, 
we consider the derived mod $p^n$ variants of the aforementioned cohomology theories.
For example, we denote the $p ^n $-torsion prismatic cohomology as $\RG_{\Prism} (X_n/ A_n) : = \RG_{\Prism} (X/ A)\otimes ^\BL_ {\Z} \Z/ p ^n \Z$.
As pointed out by Warning \ref{Warning-not-intrinsic}, 
such  $p ^n$-torsion prismatic cohomology \emph{does not only} depend on $X_n = X \times_{\Z_p} \Z/ p ^n \Z$. 
But it is enough for our purpose to understand the $p^n$-torsion crystalline cohomology 
$\rH^i _{\cris} (X_n / S_n)$ and its relation with \'etale cohomology $\rH ^i _{\et} (X_{\bar \eta}, \Z/ p ^n \Z)$. 

Note that the cohomology groups of $\RG _{\Prism} (X_n /\s_n)$ do fit in our setting of \emph{generalized} Kisin module $\m $ of height $h$ 
(discussed in \S \ref{subsec-generalized-Kisin}),  
i.e.~a  finitely generated $\s$-module $\m$  together with a $\varphi_\s$-semi-linear map $\varphi_\m : \m \to \m$ 
and an $\s$-linear map $\psi: \m \to \varphi ^* \m$ so that $ \psi\circ (1 \otimes \varphi _\m) = E^h \id_{\varphi^* \m} $ and $(1 \otimes \varphi _\m) \circ \psi = E^h \id_{\m}$. 
The generalized Kisin module is a natural extension of classical (\'etale) Kisin module discussed above allowing $u$-torsions. 
In particular,  an \'etale  Kisin module  $\m$ of height  $h $  is a generalized Kisin module of height $h$ without $u$-torsion, where $\psi $  is just defined by   $\m \simeq E ^h \m \overset {(1\otimes \varphi_\m)^{-1}} {\simeq }\Fil^{h}_{\BK} \varphi ^*\m\subset \varphi ^*\m $,   
and similarly the BK filtration can be extended to generalized Kisin module by defining 
$\Fil^h_{\BK} \varphi^*\m \coloneqq \mathrm{Im}(\psi: \m \to \varphi ^* \m )$.   
Most importantly, $\rH ^i _{\Prism} (X_n / \s _n)$ is a generalized Kisin module of height $i$,   
and  the BK filtration on $ \varphi ^* \rH ^i_{\Prism} (X_n/ \s_n)$ exactly matches with the image of 
the Nygaard filtration  $\rH ^ i _{\qsyn} (X , \Fil^i_{\rm N} \Prism^{(1)}_n)  \to  \rH^i_{\qsyn} (X,  \Prism^{(1)}_n)$  where 
$\Fil^i_{\rm N} \Prism^{(1)}_n = \Fil^i_{\rm N} \Prism^{(1)}_{-/ \s} \otimes ^{\BL}_{\Z }\Z/ p ^n \Z $ and $ \Prism^{(1)}_n =  \Prism^{(1)}_{-/ \s} \otimes ^{\BL}_{\Z }\Z/ p ^n \Z $, 
see Proposition \ref{prop-height} and Corollary \ref{cor-BK-Nygaard}. 
One can apply many methods in the study of \'etale Kisin modules to treat $\rH^i _{\Prism} (X_n / \s_n)$ as well. 
As a consequence, we prove the following: 
\begin{theorem}\label{thm-intro-3} 
Let $A = (\s, E)$ be the Breuil--Kisin prism and write $\m ^i _n : = \rH ^i _\Prism (X_n / A_n)$.
Let $i \leq p-2$ be an integer.
Then
$\rH^{i}_{\cris} (X_n /S_n)$ has a Breuil module structure if and only if $\m^j_n$ has no $u$-torsion for $j = i , i +1$. In this scenario, we have  $\underline{\calM} (\m ^ i _  n)\simeq \rH^{i}_{\cris} (X_n /S_n)$ and $T_S (\rH^{i}_{\cris} (X_n /S_n))\simeq \rH ^i _{\et} (X_{\bar \eta}, \Z/ p ^n \Z) (i)$ as $G_K$-modules.  
\end{theorem}
Finally, by using  Caruso's \Cref{thm-caruso} for $n =1$, we can show that $\m_n  ^{i +1}$ has no $u$-torsion if $ei < p -1$,
hence deducing \Cref{Thm-intro-1}. 

In the end of this introduction, let us report what we now,
a year since writing this paper,
know slightly beyond the case of $e \cdot i < p-1$.
Recall Breuil asked \cite[Question 4.1]{BreuilIntegral} whether,
assuming $i < p-1$,
it is true that $\rH^i_{\cris}(X_n/S_n)$
always supports a Breuil module structure
with associated Galois representation given by
$\rH^i_{\et}(X_{\bar{\eta}}, \mathbb{Z}/p^n\Z)$.
In view of our \Cref{thm-intro-3}, what Breuil asked for is really
some module theoretic structure of prismatic cohomology of $X$:
whether the prismatic cohomology always has no $u$-torsion
when $i < p$.
In a sequel to this paper, among other things,
we study $u^\infty$-torsions in prismatic cohomology.
We obtain some results in the boundary case: that is $e \cdot i = p-1$.
Let us restrict ourselves further to the two extrema of the boundary case,
and our relevant findings is summarized below:
\begin{itemize}
\item When $i = 1$, Breuil's question amounts to vanishing of $u$-torsion
in the first and second prismatic cohomology.
The first prismatic cohomology is always $u$-torsion free.
The second prismatic cohomology having $u$-torsion is showed to be equivalent to
the failure of having Albanese abelian (formal) scheme of $X$,
by which we mean a map $X \to A$ with both central and 
generic fibres being the Albanese map.
This was studied by Raynaud \cite{Ray79} and we extend some of his results
using this prismatic perspective.
We generalize a construction in \cite[Subsection 2.1]{BMS1} to produce
counterexamples to Breuil's question with $e = p-1 > 1$.
In fact the module structure of $\rH^1_{\cris}(X_n/S_n)$ of this example
is too pathological that it cannot possibly
support a Breuil module structure.
In particular, in hindsight, our \Cref{Thm-intro-1} is sharp.
\item When $e=1$, this is what Fontaine--Messing \cite{FontaineMessing}
and Kato \cite{Katovanishingcycles} studied.
In the boundary case $i = p-1$, we show that the Galois representation $\rho^{p-1}$
attached to the Fontaine--Laffaille structure on the $(p-1)$-st crystalline cohomology 
is not far from the $(p-1)$-st \'{e}tale cohomology of the geometric generic fibre.
\end{itemize}
The case of $e \cdot i > p-1$ remains mysterious to us.
We believe the first step of investigation would be a better understanding of
$u^\infty$-torsions in prismatic cohomology, extending our results so far.

We arrange our paper as follows: after collecting rudiments on prismatic cohomology and derived de Rham cohomology in \S 2,  
we  establish our comparison isomorphism between the two cohomologies in \S 3 together with Frobenius structures.  
We devote efforts to discuss various filtrations  in \S 4 and establish a filtered comparison.
We remark that the theory in \S 2--\S 4 accommodates quite general classes of prisms, which opens the possibilities to develop, 
for example, Breuil--Caruso theory,  for more general base rings. 
We hope to report the generalization in this direction in future work. 
Starting from \S 5, 
we restrict ourselves to the Breuil--Kisin prism $(\s, (E))$ and focus on structures of torsion prismatic cohomology and torsion crystalline cohomology for proper smooth formal scheme $X$ over $\O_K$. 
In \S 5 we construct a connection $\nabla$ on derived de Rham cohomology and hence on crystalline cohomology. 
The \S6 recalls classical theory of Kisin modules, Breuil modules, functors to Galois representations and 
the functor $\underline \calM$ connecting Kisin modules and Breuil modules. 
Finally \S 7 assembles all previous preparations to prove Theorem \ref{thm-intro-3}. 

\subsection*{Acknowledgement} We would like to thank Bhargav Bhatt, Bryden Cais, Luc Illusie,
Teruhisa Koshikawa, Shubhodip Mondal, Deepam Patel, and Emanuel Reinecke,
for very useful discussion and communication during the course of preparing this paper.
The first named author especially would like to thank Bhargav Bhatt for so many helpful discussions
and suggestions concerning this project.
The influence of Bhatt's work and comments on the first half of this paper should be obvious to readers.
We are also grateful to Illusie for stimulating discussion and his encouragement.

\section{Preliminaries}

Starting with this section through \Cref{section filtrations}, unless stated otherwise, all completions and (completed) tensor products are derived.

\subsection{Transversal prisms}

\begin{lemma}
\label{oriented transversal prism lemma}
Let $(A,I)$ be an oriented prism with $I = (d)$. The following are equivalent:
\begin{enumerate}
\item the sequence $(p, d)$ is Koszul regular;
\item the sequence $(p, d)$ is regular;
\item the morphism $\mathbb{Z}_p[\![T]\!] \to A$ sending $T$ to $d$
is flat.
\end{enumerate}
\end{lemma}

\begin{proof}
(3) implies (1): as (3) implies that 
$A \otimes_{\mathbb{Z}_p[\![T]\!]} \mathbb{Z}_p[\![T]\!]/(p, T)$ is discrete.

(1) implies (2): (1) implies that the $p$-torsions in $A$ is uniquely $d$-divisible,
and $A/p$ has no $d$-torsion.
On the other hand, we know the $p$-torsions in $A$ is derived $d$-complete,
hence must vanish. Therefore $(p,d)$ is a regular sequence.

(2) implies (3): it suffices to show that for any prime ideal 
$\mathfrak{p} \subset \mathbb{Z}_p[\![T]\!]$ the derived tensor
$A \otimes_{\mathbb{Z}_p[\![T]\!]} \mathbb{Z}_p[\![T]\!]/\mathfrak{p}$ is discrete.
When $\mathfrak{p}$ is the unique maximal ideal, this follows immediately from (2).
So we only have to deal with height $1$ primes which are always generated by
a polynomial of the form
\[
f = T^n + p \cdot (\text{lower order terms}),
\] 
and we need to show that $A$ is $f$-torsionfree.
Suppose $a \in A$ is an $f$-torsion, modulo $p$ we see that $\bar{a} \in A/p$
is a $d^n$-torsion, now (2) implies that $\bar{a} = 0 \in A/p$.
Therefore we see that $f$-torsions in $A$ is divisible by $p$.
As (2) also implies that $A$ is $p$-torsionfree, we see that $f$-torsions in $A$
is uniquely $p$-divisible.
Since $A$ is derived $p$-complete, we see that $A$ must in fact be $f$-torsionfree.
\end{proof}

We can globalize to non-oriented prisms $(A,I)$.
The following easily follows from~\Cref{oriented transversal prism lemma}.

\begin{lemma}
\label{transversal prism lemma}
Let $(A, I)$ be a prism. The following are equivalent:
\begin{enumerate}
\item there is a $(p, I)$-completely faithfully flat cover by an oriented prism
$(A', I A')$, which satisfies the equivalent conditions 
in~\Cref{oriented transversal prism lemma};
\item the ideal $I$ is $p$-completely regular;
\item Zariski locally $(p, I)$ is a regular sequence;
\item the natural morphism 
$\Spf(A) \to [\Spf(\mathbb{Z}_p[\![T]\!])/(\mathbb{G}_m)_{\mathbb{Z}_p}]$
classified by $I$ is flat.
\end{enumerate}
\end{lemma}

Let us explain the morphism in (4) above: Zariski locally $I$ is generated
by a nonzerodivisor $d$, hence Zariski locally we get a map
$\Spf(A) \to \Spf(\mathbb{Z}_p[\![T]\!])$, 
and on overlap these generators differ by a unit in $A$, hence globally
we have a morphism to the quotient stack.
Alternatively, we can understand this map as the composition of
the universal map $\Spf(A) \to \Sigma$
introduced by Drinfeld~\cite[Section 1.2]{Dr20},
and $\Sigma \to [\Spf(\mathbb{Z}_p[\![T]\!])/(\mathbb{G}_m)_{\mathbb{Z}_p}]$
induced by $W_{\mathrm{prim}} \to \Spf(\mathbb{Z}_p[\![T]\!])$
sending a Witt vector $(x_0, x_1, \ldots)$ to $x_0$.

\begin{definition}
A prism $(A, I)$ is said to be \emph{transversal}
if it satisfies the equivalent conditions in~\Cref{transversal prism lemma}.
\end{definition}

For the remaining of this subsection, let us assume $(A,I)$ to be a transversal prism.
Denote the $p$-completed PD envelope of $A \twoheadrightarrow A/I$ by $\mathcal{A}$, and
denote the kernel of $\mathcal{A} \twoheadrightarrow A/I$ by $\mathcal{I}$.

\begin{example}
Let us list some examples of transversal prisms.
\begin{enumerate}
\item The universal oriented prism is transversal.
\item The Breuil--Kisin prism~\cite[Example 1.3.(3)]{BS19} is transversal. 
We have $A = \mathfrak{S}$ and $\mathcal{A}$ is classically denoted by $S$
in classical literature concerning Breuil modules.
\item Let $\C$ be an algebraically closed complete non-Archimedean field extension of
 $\mathbb{Q}_p$. Then the perfect prism associated with $\mathcal{O}_\C$ is transversal.
We have $A = A_{\mathrm{inf}}$ and $\mathcal{A} = A_{\mathrm{crys}}$.
\end{enumerate}
\end{example}

Although $\mathcal{A}$ is usually not flat over $A$,
it has $p$-completely finite Tor dimension.
In the next subsection we shall see that this is a general phenomenon
about derived de Rham complex and regularity of $I$.

\begin{lemma}
\label{calA Tor dim 1}
Let $(A,I)$ be a transversal prism. 
Then $A \to \mathcal{A}$ has $p$-complete amplitude
in $[-1,0]$, in particular $p$-completely base changing along $A \to \mathcal{A}$
commutes with taking totalizations in $D^{\geq 0}(A)$.
\end{lemma}

\begin{proof}
It suffices to check the statement Zariski locally on $\Spf(A)$,
hence we may assume the prism to be oriented, say $I = (d)$.
Then we may base change to $A/p$.
So we need to check that given an $\mathbb{F}_p$-algebra $R$,
and a nonzerodivisor $d \in R$, the divided power algebra
$S = D_R(d)$ has Tor amplitude in $[-1,0]$ over $R$.
This follows from the fact that $d^p = 0$ in $S$ and
$S$ is a free $R/(d^p)$-module.
The commutation of tensor and totalization now follows from~\cite[Lemma 4.20]{BS19}.
\end{proof}

\subsection{Envelopes and derived de Rham cohomology}

Let $(A,I)$ be a bounded prism.
In this subsection we review derived de Rham complex of simplicial $A$-algebras relative to $A$.

First we want to spell out explicitly the process of freely adjoining divided powers
or delta powers of elements mentioned in \cite[subsections 2.5-2.6 and section 3]{BS19}.

\begin{construction}
\label{definition of envelopes}
(0) Recall $I$ is locally generated by a nonzerodivisor in $A$.
Let $A_i$ be an affine open cover of $\Spf(A)$ such that $I \cdot A_i = (d_i)$ where $d_i \in I$. 
There is an $A$-algebra $A[I \cdot x]$ by glueing $A_i[x_i]$'s via
$x_i = \frac{d_i}{d_j}x_j$, it has a surjection $A[I \cdot x] \twoheadrightarrow A$
by glueing maps $x_i \mapsto d_i$.
Alternatively one may directly define
\[
A[I \cdot x] \coloneqq \bigoplus_{n \geq 0} I^n
\]
with the evident surjection being the natural inclusion on each factor.
It can also be seen as the ring of functions on the total space of the line bundle $I^{-1}$
on $\Spec(A)$.

Similarly there is a $\delta$-$A$-algebra $A\{ I \cdot x\}$
by glueing $A_i\{x_i\}$'s via
$A_i\{x_i\} \otimes_{A_i} A_{ij} \xrightarrow{x_i \mapsto \frac{d_i}{d_j}x_j} 
A_j\{x_j\} \otimes_{A_j} A_{ij}$
with a surjection $A\{I \cdot x\} \twoheadrightarrow A$ by glueing maps
$x_i \mapsto d_i$.
Alternatively one may directly define
\[
A\{ I \cdot x\} \coloneqq \bigotimes_A^{m \geq 1} \big( \bigoplus_{n \geq 0} (\delta^m(I))^n \big)
\]
with the evident surjection being the natural map on each tensor factor.
This can also be seen as the ring of functions on the total space of an infinite rank vector bundle
on $\Spec(A)$.

Note that the above construction can be generalized to the case where $I$ is replaced
by a line bundle $\mathcal{L}$ on $\Spec(A)$.
In particular one can make sense of $A\{I^{-1} \cdot x\}$ and $A\{\varphi(I) \cdot x\}$.
We remark that there is a natural map $A\{x\} \to A\{I^{-1} \cdot y\}$ by glueing
the maps $x \mapsto d_i y_i$ which we short hand as $x \mapsto y$.

(1) Let $B$ be an $A$-algebra, let $f_1, \ldots, f_r$ be a finite set of elements in $B$.
The \emph{simplicial $B$-algebra obtained by freely adjoining divided powers of $f_i$}
is denoted
by $B \llangle f_i \rrangle$ and
defined to be the derived tensor of the following:
\[
\xymatrix{
\mathbb{Z}[x_1, \ldots, x_r] \ar[r]^-{x_i \mapsto f_i} \ar[d] & B \\
D_{\mathbb{Z}[x_1, \ldots, x_r]} (x_1, \ldots x_r). &
}
\]
The \emph{simplicial $A$-algebra obtained by freely adjoining divided powers of $I$}
is denoted by $A \llangle I \rrangle$ and defined to be the derived tensor of the following:
\[
\xymatrix{
A[I \cdot x] \ar[r]^-{x_i \mapsto d_i} \ar[d] & A \\
D_{A[I \cdot x]}(\ker(A[I \cdot x] \twoheadrightarrow A)), &
}
\]
alternatively one may define it as the glueing of the simplicial $A$-algebras
$A_i \otimes_{x \mapsto d_i, A[x]} D_{A[x]}(x)$.

The \emph{simplicial $B$-algebra obtained by freely adjoining divided powers of $I, f_i$},
denoted by $B \llangle I, f_i \rrangle$,
is defined as the derived tensor of the above two algebras over $A$.

(2) Let $B$ be a $\delta$-$A$-algebra, let $f_1, \ldots, f_r$ be a finite set of elements in $B$.
We define $B\{\frac{f_i}{p}\}$ as derived pushout of the following diagram of simplicial algebras:
\[
\xymatrix{
A\{x_1, \ldots, x_r\} \ar[r]^-{x_i \mapsto f_i} \ar[d]^-{x_i \mapsto p \cdot y_i} & B \\
A\{y_1, \ldots, y_r\}. &
}
\]
We define $A\{\varphi(I)/p\}$ as derived pushout of the following:
\[
\xymatrix{
A\{I \cdot x\} \ar[r]^-{\varphi} \ar[d]^{x \mapsto p \cdot y} & A \\
A\{I \cdot y\}, &
}
\]
alternatively one may define it as the glueing of the simplicial $\delta$-$A$-algebras
$A_i\{\frac{\varphi_A(d_i)}{p}\}$.

Analogously $B\{\frac{\varphi(I)}{p}, \frac{f_i}{p}\}$ is defined as derived tensoring the above
two algebras over $A$.

(3) Given a sequence $(f_1, \ldots, f_r)$ of elements inside a ring $B$, 
we use notation
\[
\dR_{B}(f_1, \ldots, f_r)^\wedge \coloneqq \dR_{\mathrm{Kos}(B; f_1, \ldots, f_r)/B}^\wedge
\]
to denote the derived $p$-completed
derived de Rham complex of $\mathrm{Kos}(B; f_1, \ldots, f_r)$,
viewed as a simplicial $B$-algebra, over $B$.

Similarly when $B$ is an $A$-algebra, we denote
\[
\dR_{B}(I)^\wedge \coloneqq \dR_{(B \otimes_A (A/I))/B}^\wedge,
\]
and 
\[
\dR_{B}(I, f_i)^\wedge \coloneqq \dR_{\big(\mathrm{Kos}(B; f_1, \ldots, f_r) \otimes_A (A/I)\big)/B}^\wedge.
\]
Let $J$ be an ideal inside $B$, then we denote
\[
\dR_B(J)^\wedge \coloneqq \dR_{(B/J)/B}^\wedge.
\]
Here all the completion are derived $p$-completion.
\end{construction}

\begin{remark}\label{rem-2.7}
(1) Let $B = A\{x\}^\wedge$, note that $x$ is $(p,I)$-completely regular relative to $A$.
Using \cite[Proposition 3.13]{BS19}, we can get a $B$-algebra $C \coloneqq B\{\frac{x}{I}\}^\wedge$
which is locally (on $\Spf(A)$ as one needs to trivialize the line bundle $I$) given by
$C = A\{y\}^\wedge$ together with $B$-algebra structure $x \mapsto d \cdot y$
where $d$ is the local generator of $I$.
One checks immediately that, in our notation here, we have $C \cong A\{I^{-1} \cdot y\}^\wedge$
with $B$-structure given by ($(p,I)$-completion of) $x \mapsto y$.

In fact, after examining the proof of \cite[Proposition 3.13]{BS19}, one finds that
in the situation described in loc.~cit.~, the algebra $B\{\frac{J}{I}\}^\wedge$
is the derived $(p,I)$-complete pushout of the following diagram:
\[
\xymatrix{
A\{x_1, \ldots, x_r\} \ar[r] \ar[d]^-{x_i \mapsto y_i} & B \\
A\{I^{-1} \cdot y_i\}. &
}
\]

(2) We warn readers that when $J = (f_1, \ldots, f_r)$ is an ideal inside $B$, 
the two simplicial $B$-algebras $\dR_B(J)^\wedge$ and $\dR_B(f_1, \ldots, f_r)^\wedge$
are usually different.
These two agree when $(f_i)$ is a $p$-completely Koszul regular sequence.
\end{remark}

Below we shall see the relation between derived de Rham complex, divided power envelopes,
and prismatic envelopes which directly follows from \cite[Subsection 2.5]{BS19}.

\begin{lemma}
\label{three envelopes}
(1) Let $B$ be an $A$-algebra, let $\{f_1, \ldots, f_r\}$ be a finite set of elements of $B$.
Then we have the following identification of derived $p$-complete simplicial $B$-algebras:
\[
\dR_B(f_1, \ldots, f_r)^\wedge \cong
B \llangle f_i \rrangle^\wedge.
\]
Similarly we have an identification:
\[
\dR_B(I)^\wedge \cong
B \llangle I \rrangle^\wedge.
\]

(2) Let $B$ be a $\delta$-$A$-algebra, let $\{f_1, \ldots, f_r\}$ be a finite set of elements of $B$.
Then we have the following identification of derived $p$-complete simplicial $B$-algebras:
\[
B \llangle f_i \rrangle^\wedge \cong
B\{\frac{\varphi(f_i)}{p}\}^\wedge.
\]
Similarly we have an identification:
\[
B \llangle I \rrangle^\wedge \cong
B\{\frac{\varphi(I)}{p}\}^\wedge.
\]
\end{lemma}

\begin{proof}
By base change property of their constructions, we reduce ourselves to
the case where $B = A\{x_1, \ldots, x_r\}$ with $f_i = x_i$.
Again by base change we may assume $A$ is the initial oriented prism,
in particular it is flat over $\mathbb{Z}_p$ and $I = (d)$ is generated by a nonzerodivisor.
So we can focus on the case concerning finite set of elements of $B$,
we may further reduce to the case where the set is a singleton.

Now the identification in (1) follows from (the limit version of) \cite[Theorem 3.27]{Bha12}
and \cite[Th\'{e}or\`{e}me V.2.3.2]{crystal2}.
The identification in (2) follows from \cite[Lemma 2.36]{BS19}.
\end{proof}

We deduce a consequence concerning the Tor amplitude of $\dR_A(I)^\wedge$ over $A$,
generalizing \Cref{calA Tor dim 1}.

\begin{lemma}
\label{dR Tor dim 1}
Let $(A,I)$ be a prism. 
Then $A \to \dR_A(I)^\wedge$ has $p$-complete amplitude
in $[-1,0]$, in particular $p$-completely base changing along $A \to \dR_A(I)^\wedge$
commutes with taking totalizations in $D^{\geq 0}(A)$.
\end{lemma}

\begin{proof}
We may check this statement locally on $\Spf(A)$, hence we may assume $I = (d)$.
Next, by base change, we may assume $A$ to be the initial oriented prism,
in particular we may assume it to be transversal.
Using \Cref{three envelopes} (1), we see now
$\dR_A(I)^\wedge$ is the $p$-completion of the divided power envelope $D_A(I)^\wedge$.
This reduces the Lemma to \Cref{calA Tor dim 1}.
\end{proof}

We also have a prototype base change formula which will be used in the next section
to establish a general comparison.

\begin{lemma}
\label{naive base change}
Let $(A,I)$ be a prism, denote the composition 
$A\{x\} \xrightarrow{\varphi_{A}, x \mapsto \varphi(z)} A\{z\} \to \dR_{A\{z\}}(I)^\wedge$ by $f$.
Then we have a base change formula:
\[
A\{I^{-1} \cdot x\} \widehat{\otimes}_{A\{x\}, f} \dR_{A\{z\}}(I)^\wedge
\cong \dR_{A\{z\}}(I,z)^\wedge,
\]
\end{lemma}

Here the completion on the left hand side is derived $p$-completion. 
As $\varphi(I) = (p)$ inside $\pi_0(\dR_{A\{z\}}(I)^\wedge)$, it is the same
as derived $(p,I)$-completion when viewed as an $A$-complex
via $\varphi_A \colon A \to \dR_A(I)^\wedge$.

\begin{proof}
Note that by \Cref{three envelopes} we have identifications
$\dR_{A\{z\}}(I)^\wedge \cong A\{z\}\{\frac{\varphi(I)}{p}\}^\wedge$
as $p$-complete simplicial $A\{z\}$-algebras.
Similarly we can identify $\dR_{A\{z\}}(I,z)^\wedge$ with 
$A\{z\}\{\frac{\varphi(z)}{p}, \frac{\varphi(I)}{p}\}^\wedge$.

Now we look at the following diagram
\[
\xymatrix{
A\{x\} \ar[r] \ar[d] & A\{z\} \ar[r] \ar[d] & A\{z\}\{\frac{\varphi(I)}{p}\} \ar[d] \\
A\{I^{-1} \cdot x\} \ar[r] & A\{z\}\{\frac{\varphi(z)}{\varphi(I)}\} \ar[r] & A\{z\}\{\frac{\varphi(z)}{p}, \frac{\varphi(I)}{p}\}.
}
\]
The left square is a pushout diagram by definition.
Hence it suffices to show the right square, after derived $p$-completion,
is also a pushout diagram of $p$-complete simplicial $A\{z\}$-algebras.

To that end, we may work Zariski locally on $A$, so we can assume $I = (d)$
is generated by one element.
This square is the base change of the same diagram when $A$ is the
initial oriented prism, so we have reduced ourselves to that case.
Now every ring in sight is discrete, and the $p$-completed square is a pushout diagram
because $\varphi(d)$ and $p$ differ by a unit inside 
$A\{\frac{\varphi(d)}{p}\}^\wedge \cong D_A(d)^\wedge$.
\end{proof}

In \cite[Proposition 3.25]{Bha12}, for any $p$-complete $A$-algebra $B$,
Bhatt constructed a natural map
\[
\mathcal{C}\mathrm{omp}_{B/A} \colon \dR_{B/A}^\wedge \to \mathrm{R\Gamma_{crys}}(B/A).
\]
Here the right hand side denotes the $p$-complete crystalline cohomology
defined using PD thickenings of $B$ relative to $(A, (p), \gamma)$ 
with $\gamma_i(p) = p^i/i!$.
This natural map is functorial in $A \to B$ and agrees with Berthelot's 
de Rham--crystalline comparison \cite[Th\'{e}or\`{e}me IV.2.3.2]{crystal2}
when it is formally smooth (viewed as $p$-adic algebras).
It is shown that when both $A$ and $B$ are flat over $\mathbb{Z}_p$ and 
$A \to B$ is $p$-completely locally complete intersection, then the natural map
above is an isomorphism \cite[Theorem 3.27]{Bha12}.

For our purpose we shall be interested in the situation where $B$ is formally smooth
over $A/I$, we cannot summon the above Theorem in loc.~cit.~to say that the natural
map in this situation is an isomorphism.
In fact, when $B = A/I$ the left hand side is $\dR_A(I)^\wedge$ and the right hand side
is the classical $p$-adic completion of the PD envelope of $A$ along $I$ (compatible with the natural PD structure on $(A, (p))$),
denoted as $\mathcal{A}$  \footnote{This notation agrees with the previous subsection as we assumed $(A,I)$ to be a transversal prism there.}.
These two need not be the same, e.g.~take $A = \mathbb{Z}_p$ and $I = (p)$,
then $\dR_A(I)^\wedge = \mathbb{Z}_p[T^i/i!]^\wedge/(T-p)$ but $\mathcal{A} = \mathbb{Z}_p$.
However this turns out to be the only problem.

\begin{proposition}
\label{dR and crys}
Let $B$ be an $A/I$-algebra.

(1) If $B$ is formally smooth over $A/I$, then we have a natural identification
\[
\mathrm{R\Gamma_{crys}}(B/A) \cong \mathrm{R\Gamma_{crys}}(B/\mathcal{A}),
\]
where the right hand side is the usual crystalline cohomology of $\Spf(B)$ over
the PD base $\mathcal{A}$.

(2) There is a natural map
\[
\mathrm{Comp}_{B/A} \colon \dR_{B/A}^\wedge \widehat\otimes_{\dR_A(I)^\wedge} \mathcal{A}
\to \mathrm{R\Gamma_{crys}}(B/\mathcal{A}),
\]
which is functorial in $A/I \to B$.

(3) If $B$ is formally smooth over $A/I$, then the above is an isomorphism.
\end{proposition}

\begin{proof}
(1) is an easy consequence of the fact that $B$ is an $A/I$-algebra.
In fact we only need $A/I \to B$ to be a local complete intersection.
Indeed we use \v{C}ech--Alexander complex to compute both crystalline cohomology,
and one reduces to the following:
Let $P$ be a polynomial $A$-algebra with a surjection $P \twoheadrightarrow B$ of $A$-algebras, 
then there is a naturally induced surjection $P \otimes_A \mathcal{A} \twoheadrightarrow B$ of $\mathcal{A}$-algebras,
and we have an identification of PD envelopes
\[
D_{(A,(p), \gamma)}(P \twoheadrightarrow B) = D_{(\mathcal{A},\mathcal{I}, \gamma)}(P \otimes_A \mathcal{A} \twoheadrightarrow B).
\]

(2) The functoriality of Bhatt's $\mathcal{C}\mathrm{omp}_{B/A}$ asserts that the map
is compatible with the natural map $\dR_A(I)^\wedge \to \mathcal{A}$, 
hence we get our natural map $\mathrm{Comp}_{B/A}$.

(3) Choose a formal lift $\tilde{B}$ over $A$ (note that $A$ is $(p,I)$-complete).
By the functoriality of Bhatt's
$\mathcal{C}\mathrm{omp}_{B/A}$, we get the following commutative diagram:
\[
\xymatrix{
\dR_{\tilde{B}/A}^\wedge \widehat\otimes_A \mathcal{A} \ar[r] \ar[d] &
\mathrm{R\Gamma_{crys}}(\tilde{B}/A) \widehat\otimes_A \mathcal{A} \ar[d] \\
\dR_{B/A}^\wedge \widehat\otimes_{\dR_A(I)^\wedge} \mathcal{A} \ar[r] &
\mathrm{R\Gamma_{crys}}(B/\mathcal{A}).
}
\]
The top horizontal arrow is an isomorphism by Berthelot's 
de Rham-crystalline comparison.
The left vertical arrow is an isomorphism by the K\"{u}nneth formula of derived de Rham
complex: $\dR_{\tilde{B}/A}^\wedge \widehat\otimes_A \dR_A(I)^\wedge \cong \dR_{B/A}^\wedge$.
The right vertical arrow is an isomorphism by base change formula of crystalline cohomology.
Therefore we conclude that the bottom horizontal arrow, which is our $\mathrm{Comp}_{B/A}$,
must also be an isomorphism.
\end{proof}

The above proposition and Bhatt's results discussed before suggest that derived de Rham
complex is a substitute of crystalline cohomology.
Inspired by this philosophy, below let us show that derived de Rham complex only
``depends on the reduction mod $p$ of the input algebra''.
We need to introduce some notations first: denote the $p$-adic derived de Rham complex
$\dR_{\mathbb{F}_p/\mathbb{Z}_p}^\wedge$ by $D$.
Bhatt's result implies that the natural map $\mathbb{Z}_p \to D$ admits a retraction $D \to \mathbb{Z}_p$.
In \Cref{example computing Frobenius on dR} (1) below,
one finds a detailed description of $D$.

\begin{remark}
In fact, one can show that $D$ is the $p$-complete PD envelope of $\mathbb{Z}_p$ along the ideal $(p)$ by $D$.
Moreover under this identification one can easily see that the retraction above is unique, 
and is given by the fact that there is a unique PD structure on $(\mathbb{Z}_p, (p))$ (as $\mathbb{Z}_p$ has no $p$-torsion).
Notice that when taking PD envelope, one has to fix a PD base ring,
and we always take it to be the trivial PD ring $(\mathbb{Z}_p, (0), \gamma_{\mathrm{triv}})$
when we say PD envelope without mentioning a PD base ring.
\end{remark}

\begin{proposition}
\label{dependence on special fiber}
Let $R$ be a ring with its derived $p$-completion $R^\wedge$, 
let $B$ be a simplicial $R$-algebra. Then there is a natural isomorphism:
\[
\dR_{\mathrm{Kos}(B; p)/R}^\wedge \widehat{\otimes}_{D} \mathbb{Z}_p
\cong \dR_{B/R}^\wedge
\]
which is functorial in $R \to B$.
\end{proposition}

Here the map $D \to \dR_{\mathrm{Kos}(B; p)/R}^\wedge$ is induced by the following natural diagram
\[
\xymatrix{
\mathbb{F}_p \ar[r] & \mathrm{Kos}(B;p) = B \otimes_{\mathbb{Z}} \mathbb{F}_p \\
\mathbb{Z} \ar[r] \ar[u] & R. \ar[u] 
}
\]

\begin{proof}
This follows from the K\"{u}nneth formula of derived de Rham complex:
\[
\dR_{\mathrm{Kos}(B; p)/R}^\wedge \cong \dR_{B/R}^\wedge \widehat{\otimes}_R \dR_R(p)^\wedge,
\]
and the base change formula $\dR_R(p)^\wedge \cong D \widehat{\otimes}_{\mathbb{Z}_p} R^\wedge$ 
as $\mathrm{Kos}(R; p) = R \otimes_\mathbb{Z} \mathbb{F}_p$.
\end{proof}

\subsection{Frobenii}
\label{Frobenii}

Let $A$ be a $p$-torsionfree $\delta$-ring.
Using \Cref{dependence on special fiber} we can define a Frobenius action on $\dR_{B/A}^\wedge$
which is functorial in $(A, \varphi_A)$ and the $A$-algebra $B$.

\begin{construction}
\label{Frobenius construction}
Let $A$ be a $p$-torsionfree $\delta$-ring and $B$ a simplicial $A$-algebra.
Recall there is a functorial endomorphism on simplicial $\mathbb{F}_p$-algebras given by left Kan extending the usual Frobenius
on polynomial $\mathbb{F}_p$-algebras, see \cite[Construction 2.2.6]{DAG13}. 
For discrete $\mathbb{F}_p$-algebras, it is just the usual Frobenius.
We may view $B/p = B \otimes_{A} A/p$, using the fact that $\varphi_A$ on $A$ is a lift of Frobenius on $A/p$ we get the following commutative diagram:
\[
\xymatrix{
B/p \ar[r]^-{\varphi_{B/p}} & B/p \\
A \ar[r]^-{\varphi_A} \ar[u] & A \ar[u],
}
\]
it induces a Frobenius map $\tilde{\varphi} \colon \dR_{\mathrm{Kos}(B;p)/A}^\wedge \to \dR_{\mathrm{Kos}(B;p)/A}^\wedge$ which is functorial in $(A \to B, \varphi_A)$.

Similar diagram for $\mathbb{Z} \to \mathbb{F}_p$ (where $A = B = \mathbb{Z}_p$) induces identity on $D$, hence we have a commutative diagram:
\[
\xymatrix{
\dR_{\mathrm{Kos}(B;p)/A}^\wedge \ar[rr]^-{\tilde{\varphi}} & & \dR_{\mathrm{Kos}(B;p)/A}^\wedge \\
& D. \ar[lu] \ar[ru] &
}
\]
Finally we define a Frobenius map $\varphi_{B/A} \colon \dR_{\mathrm{Kos}(B; p)/A}^\wedge \widehat{\otimes}_{D} \mathbb{Z}_p \cong \dR_{B/A}^\wedge 
\xrightarrow{\tilde{\varphi} \widehat{\otimes}_{\mathrm{id}_D} {\mathrm{id}_{\mathbb{Z}_p}}} \dR_{B/A}^\wedge$
which is functorial in $(A \to B, \varphi_A)$.
\end{construction}

\begin{remark}

(1)
It is conceivable that the above works for general $\delta$-rings.
In private communication we learned from Bhatt that a $\delta$-structure on a ring $A$ is equivalent to specifying a commutative diagram as follows:
\[
\xymatrix{
A/p \ar[r]^-{\varphi_{A/p}} & A/p \\
A \ar[r]^-{\varphi_A} \ar[u] & A \ar[u],
}
\]
note that here $A/p$ is a simplicial $\mathbb{F}_p$-algebra that has nontrivial $\pi_1$ when $A$ is not $p$-torsionfree.
Hence for any simplicial $A$-algebra $B$, 
one can also define a Frobenius on $\dR_{B/A}^\wedge$ as above.
However we do not work out the full story here as we do not need this great generality
for our intended applications later.

(2) By letting $n \to \infty$ in
\cite[Proposition 3.47]{Bha12}, one gets another construction of Frobenius on 
$\dR_{A/\mathbb{Z}_p}^\wedge$ for any $\mathbb{Z}_p$-algebra $A$.
However later on we shall see in \Cref{functorial endo remark} that there is only one 
Frobenius that is functorial enough (in a suitable sense) on $p$-completed 
derived de Rham complexes when the base algebra is a 
$p$-torsionfree $\delta$-algebra.
In particular, our construction above agrees with Bhatt's whenever both are defined
(i.e.~when the base is $\mathbb{Z}_p$).
\end{remark}

Let us work out some examples.

\begin{example}
\label{example computing Frobenius on dR}
(1) As an illustrative example, let us contemplate with $A = \mathbb{Z}_p$ and $B = \mathbb{F}_p$.
We have a derived pushout square of rings:
\[
\xymatrix{
\mathbb{Z}_p \ar[r] & B \\
\mathbb{Z}_p[T] \ar[u]^-{T \mapsto p} \ar[r]^-{T \mapsto 0} & A \ar[u],
}
\]
moreover the bottom map is a map of $\delta$-rings if we give
$\mathbb{Z}_p[T]$ a $\delta$-structure with $\varphi(T) = T^p$.
Then we get a pushout diagram of derived de Rham complex which says
$D \cong \dR_{\mathbb{Z}_p/\mathbb{Z}_p[T]}^\wedge \widehat{\otimes}_{\mathbb{Z}_p[T]} \mathbb{Z}_p$.
The latter is the same as 
$\mathbb{Z}_p\llangle T \rrangle^\wedge/(T)$ where we have used the fact that
$p$ has divided powers in $\mathbb{Z}_p$ (hence adjoining divided powers of
$T - p$ is the same as adjoining divided powers of $T$).
It is easy to see that the Frobenius defined on 
$\dR_{\mathbb{Z}_p/\mathbb{Z}_p[T]}^\wedge \cong 
\mathbb{Z}_p\llangle T \rrangle$ is induced by $T \mapsto T^p$
because it has to be compatible with the Frobenius on $\mathbb{Z}_p[T]$.
Therefore the induced Frobenius on $\dR_{B/A}^\wedge$
is \emph{not} the identity.
This might be surprising as one would na\"{i}vely think that the Frobenius
on the pair $(\mathbb{Z}_p, \mathbb{F}_p)$ is identity, hence must induce
identity on the derived de Rham complex.
However the Frobenius on
$\mathbb{F}_p \otimes_{\mathbb{Z}_p} \mathbb{F}_p$ is \emph{not}
the identity (as Frobenius always kills cohomology classes in negative degrees, 
see \cite[Remark 2.2.7]{DAG13}), 
and it is this Frobenius that induces a map on the derived de Rham complex.
On a related note, Bhatt has pointed out to us that the identity map
is also \emph{not} a lift of Frobenius on
$D \cong  \mathbb{Z}_p\llangle T \rrangle^\wedge/(T)$.

(2) Let $J \subset A$ be an ideal which is Zariski locally on $\Spec(A)$ a
colimit of ideals generated by a $p$-completely regular sequence.
Then by \Cref{three envelopes} (1), we have an identification:
$\dR_{A}(J)^\wedge \cong D_A(J)^\wedge$.
Since the Frobenius map obtained is compatible with $\varphi_A$ and $D_A(J)^\wedge$
is $p$-torsionfree, we see that this pins down the Frobenius on $\dR_A(J)^\wedge$:
any $\gamma_n(f)$ with $f \in J$ must be sent to $\frac{\varphi_A(f)^n}{n!}$.
Note that $f^p$ is divisible by $p$ in $D_A(J)^\wedge$, hence
$\varphi_A(f)$ is divisible by $p$ in $D_A(J)^\wedge$.

(3) Let $A$ be $p$-complete, and let $B = A \langle X^{1/p^{\infty}} \rangle$.
Since $A \to B$ is relatively perfect
modulo $p$, there is a unique lift of Frobenius $\varphi_B$ on $B$ covering the Frobenius
on $A$ and it is given by $\varphi_B(X^i) = X^{i \cdot p}$.
By \cite[Proposition 3.4.(1)]{GL20}, we see the natural map to $0$-th
graded piece of Hodge filtration induces an isomorphism
$\dR_{B/A}^\wedge \cong B$.
Applying the functoriality of the \Cref{Frobenius construction} to the map of
triples: $(A \to B, \varphi_A) \to (B \to B, \varphi_B)$, we see that the Frobenius
on $\dR_{B/A}^\wedge \cong B$ must be $\varphi_B$.
\end{example}

When the map $A \to B$ is a surjection with good regularity properties, 
we see in \Cref{three envelopes} that one can express $\dR_{B/A}^\wedge$
in terms of prismatic envelopes. Since prismatic envelopes are $\delta$-rings, they possess a Frobenius map by design. 
We can use this to give an alternative construction of the Frobenius for derived de Rham cohomology of certain regular $A$-algebras relative to $A$.
To that end, we need to first establish a sheaf property for derived de Rham cohomology.

\begin{proposition}
\label{dR sheaf}
Let $S$ be an $R$-algebra.
Assume:
\begin{itemize}
\item the cotangent complex
$\mathbb{L}_{S/R} \in D(S)$ has $p$-completely Tor amplitude in $[-1,0]$;
\item the (relative to $R/p$) Frobenius twist of $S/p$ is in $D^{\geq -m}(\mathbb{F}_p)$.
\end{itemize}
Consider the category $\mathcal{C}$ consisting of triangles
$R \to P \to S$ with $P$ being an ind-polynomial $R$-algebra,
equipped with indiscrete topology.
Let $\dR_{S/-}^\wedge$ be the sheaf that associates any triangle
$R \to P \to S$ with $\dR_{S/P}^\wedge$.
Then we have:
\begin{enumerate}
\item For any $R \to P \to S$, the $\dR_{S/P}^\wedge$ is in $D^{\geq -m}(R)$.
\item The natural map $\dR_{S/R}^\wedge \to \lim_{\mathcal{C}} \dR_{S/P}^\wedge$
is an isomorphism.
\item For any $R \to P \to S$ with $P \twoheadrightarrow S$ surjective, the natural
map
$\dR_{S/R}^\wedge \to \lim_{\Delta} \dR_{S/P_{\bullet}}^\wedge$
is an isomorphism.
Here $P_n \coloneqq P^{\otimes_R {(n+1)}}$ for any $[n] \in \Delta$,
with induced maps $P_n \twoheadrightarrow S$.
\end{enumerate}
\end{proposition}

\begin{proof}
We shall prove this by reduction modulo $p$.
Hence we may assume $R$ and $S$ are simplicial $\mathbb{F}_p$-algebras.

For (1) we use the conjugate filtrations on the derived de Rham complex.
Since $\mathbb{L}_{S/R}$ has Tor amplitude in $[-1,0]$, so is $\mathbb{L}_{S^{(1,P)}/P}$
where $S^{(1,P)}$ is the (relative to $P$) Frobenius twist of $S$.
The above estimate shows that the graded pieces of the conjugate filtration
has Tor amplitude at least $0$ over $S^{(1,P)}$. 
Since $S^{(1)}$ is assumed to be in $D^{\geq -m}(\mathbb{F}_p)$ and the relative Frobenius
for $P$ is flat,
we see that all the graded pieces of the conjugate filtration lives in
$D^{\geq -m}(R)$.

Note that $P \twoheadrightarrow S$ is surjective if and only if
$R \to P \to S$ is weakly final in $\mathcal{C}$.
Since these $\dR_{S/P}^\wedge$ are cohomologically uniformly bounded below,
\cite[Lecture V, Lemma 4.3]{BhaNotes18}
(see also~\cite[\href{https://stacks.math.columbia.edu/tag/07JM}{Tag 07JM}]{stacks-project}) reduces (2) to (3).

Lastly to show (3) we appeal to the conjugate filtration again.
Since the graded pieces of the conjugate filtration is 
cohomologically uniformly bounded below by our proof of (1) above,
it suffices to show 
$\mathbb{L}_{S^{(1)}/R} \to \lim_{\Delta} \mathbb{L}_{S^{(1,\bullet)}/P_{\bullet}}$
is an isomorphism, where $S^{(1,n)}$ is the (relative to $P_n$)
Frobenius twist of $S$.
This follows easily from the fact that $\lim_{\Delta} \mathbb{L}_{P_{\bullet}/R} \cong 0$.
\end{proof}

The above Proposition gives us a way to describe the Frobenius action of the $p$-completed
derived de Rham complex in more cases than those listed in 
\Cref{example computing Frobenius on dR}.

\begin{proposition}
\label{another Frobenius}
Let $A$ be a $p$-torsionfree $p$-complete $\delta$-algebra,
and let $I \subset A$ be an ideal which is Zariski locally on $\Spec(A)$ generated by
a $p$-completely regular element.
Let $B$ be a $p$-completely smooth $A/I$-algebra.
Then we have:
\begin{enumerate}
\item For any $(p, I)$-completely 
ind-polynomial $A$-algebra $P$ with a surjection $P \twoheadrightarrow B$,
the kernel $J$ is Zariski locally on $\Spf(P)$ colimit of ideals generated by a $p$-completely
regular sequence.
\item For any such $A \to P \to B$ as in (1), the $\dR_{B/P}^\wedge$ is an
ordinary algebra.
\item For any $(p, I)$-completely 
free $\delta$-$A$-algebra $F$ with a surjection $F \twoheadrightarrow B$,
there is a unique $\delta$-algebra structure on $\dR_{B/F}^\wedge$
compatible with that on $F$.
With this $\delta$-structure, we have an identification:
\[
\dR_{B/F}^\wedge \cong F\{\frac{\varphi_F(J)}{p}\}^\wedge.
\]
\item Consider the category $\mathcal{C}$ of all triples
$A \to F \twoheadrightarrow B$ as in (3), we have
\[
\dR_{B/A}^\wedge \cong \lim_{\mathcal{C}} \dR_{B/F}^\wedge.
\]
In fact, it suffices to take limit over the Cech nerve of one such $F \twoheadrightarrow B$.
Together with (3) we get a natural Frobenius action on $\dR_{B/A}^\wedge$.
\item The Frobenius on $\dR_{B/A}^\wedge$ obtained in (4) 
agrees with the one in \Cref{Frobenius construction}.
\end{enumerate}
\end{proposition}

The notation $F\{\frac{\varphi_F(J)}{p}\}^\wedge$ is defined analogously as in 
\cite[Corollary 3.14]{BS19}.
Using $J$ is Zariski locally given by an ind-$p$-completely regular ideal,
we may define $F\{\frac{\varphi_F(J)}{p}\}^\wedge$ as the glueing of the colimit
of $F\{\frac{\varphi_F(f_i)}{p}\}^\wedge$, where $(f_i)$ is the ind-regular sequence
generating $J$ on a Zariski open.

\begin{proof}
(1) follows easily from the fact that $B$ is formally smooth over $A/I$ and $I$ is
Zariski locally generated by a $p$-completely regular element.

(2) follows from the argument of \Cref{dR sheaf} (1).
Indeed we set $R = A$ and $S = B$. The Frobenius twist of $B/p$ is smooth over
$A/(\varphi_A(I),p) = A/(I^p,p)$, and the latter is an ordinary algebra.
Hence in our situation, we have $m=0$ in the condition of \Cref{dR sheaf}.
This shows that the $\dR_{B/P}^\wedge$ is in $D^{\geq 0}$.
Using conjugate filtration again, it is easy to see that the $p$-completed derived de Rham
complex of any surjection must be in $D^{\leq 0}$.
Hence our $\dR_{B/P}^\wedge$ must in fact be an ordinary algebra.

(3) essentially follows from (1)  and \Cref{three envelopes}.
Indeed by description of $J$, we see that $\dR_{B/F}^\wedge \cong D_F(J)^\wedge$.
Since $J$ is Zariski locally an ind-$p$-completely regular ideal, we see that
$D_F(J)^\wedge$ is $p$-torsionfree, hence having a $\delta$-structure is equivalent
to having a lift of Frobenius.
The argument in \Cref{example computing Frobenius on dR} (2) tells us
that there is at most one Frobenius structure on it compatible with that on $F$.
Lastly \Cref{three envelopes} shows that we can put a $\delta$-structure on
it by identifying
\[
\dR_{B/F}^\wedge \cong D_F(J)^\wedge \cong F\{\frac{\varphi_F(J)}{p}\}.
\]

(4) follows from \Cref{dR sheaf} (2)-(3).

As for (5), it suffices to notice that for any of these $A \to F \twoheadrightarrow B$
the two Frobenii defined on $\dR_{B/F}^\wedge$ agree and they are
both functorial in $A \to F \twoheadrightarrow B$.
\end{proof}

The following is similar to \Cref{dR sheaf}, and will be used later in the next section.
\begin{proposition}
\label{lim dR}
Let $(A,I)$ be a bounded prism. Let $R$ be a formally smooth $A/I$-algebra.
Consider $\mathcal{C}$ the category of all triples $A \to P \twoheadrightarrow R$
where $P$ is a $p$-completed polynomial algebra over $A$.
Associated with such a triple is the following diagram:
\[
\xymatrix{
A \ar[r] \ar[d] & P \ar[r] \ar[d] & F \ar[d] \\
A/I \ar[r] & R \ar[r] & S,
}
\]
where $F$ is the $p$-completed free $\delta$-$A$-algebra associated with $P$, and
$S$ is the $p$-completed tensor product $S \coloneqq R \widehat{\otimes}_{P} F$.
Then we have:
\begin{enumerate}
    \item Choose an object $A \to P \twoheadrightarrow R$, consider the $n$-th self-fiber product
    $A \to P^n \coloneqq P^{\hat{\otimes}_A n} \twoheadrightarrow R$ for any positive integer $n$.
    Then the associated $p$-completed free $\delta$-$A$-algebra is $F^n \coloneqq F^{\hat{\otimes}_A n}$,
    and we have
    \[
    R \widehat{\otimes}_{P^n} F^n \cong S^{\hat{\otimes}_R n},
    \]
    which we shall denote by $S^n$ below.
    \item Choose an object $A \to P \twoheadrightarrow R$, then the natural map
    \[
    \dR_{R/A}^\wedge \rightarrow \lim_{[n] \in \Delta} \dR_{S^n/F^n}^\wedge
    \]
    is an isomorphism.
    \item The natural map 
    \[
    \dR_{R/A}^\wedge \rightarrow \lim_{\mathcal{C}} \dR_{S/F}^\wedge
    \]
    is an isomorphism.
\end{enumerate}
\end{proposition}

Notice that we do not need to assume $A$ to be $p$-torsionfree here.

\begin{proof}
For (1): if $P$ is $p$-completely adjoin a set $T$ of variables, then $F$ is $p$-completely adjoin
the set $\coprod_{\mathbb{N}} T$ of variables, where $t$ in the $i$-th component represents $\delta^i(x_t)$.
The statement on fiber product and the associated $F^n$ is clear.
As for the statement about $S^n$, just notice that we have the following pushout diagrams:
\[
\xymatrix{
P^n \ar[r] \ar[d] & R^n \coloneqq R^{\hat{\otimes}_A n} \ar[d] \ar[r] & R \ar[d] \\
F^n \ar[r] & S^{\hat{\otimes}_A n} \ar[r] & S^n \coloneqq S^{\hat{\otimes}_R n}.
}
\]

To prove (2), we may reduce modulo $p$. 
Note that $A \to F$ and $R \to S$ are $p$-completely faithfully flat.
In a similar manner to the proof of \Cref{dR sheaf} (3),
using conjugate filtration, plus the distinguished triangle of cotangent complex,
and fpqc descent of cotangent complex (see \cite[Theorem 3.1]{BMS2}), one can show this natural map
is an isomorphism.

(3) follows from (2) in the same way as how \Cref{dR sheaf} (2) follows from \Cref{dR sheaf} (3).
\end{proof}

\begin{remark}
\label{another Frobenius on limit dR}
Similar to \Cref{another Frobenius}, assume $A$ to be $p$-torsionfree, 
then these $\dR_{S/F}^\wedge$ appeared above are discrete rings, and we can equip
them a natural $\delta$-structure.
By the same proof of \Cref{another Frobenius} the induced Frobenius on $\dR_{R/A}^\wedge$
agrees with the one provided by \Cref{Frobenius construction}.
\end{remark}

Later on we shall see in \Cref{functorial endo remark} (1) that 
if $(A,I)$ is a transversal prism, then there is only one
Frobenius in a strong sense. So all these different constructions must give rise to the same map.

\subsection{Na\"{i}ve comparison}

Consider the composition $f \colon A \xrightarrow{\varphi_A} A \to \mathcal{A}$, it induces a morphism of prisms which we still denote by
$f \colon (A,I) \to (\mathcal{A},(p))$.
Let $\mathcal{X}$ be a $p$-completely smooth affine formal scheme over $\Spf(A/I)$.
Now by base change formula of prismatic cohomology~\cite[Theorem 1.8.(5)]{BS19}, we have
\[
\mathrm{R\Gamma}_{\Prism}(\mathcal{X}/A) \widehat{\otimes}_{A,f} \mathcal{A} \cong
\mathrm{R\Gamma}_{\Prism}(\mathcal{Y}/\mathcal{A}),
\]
where $\mathcal{Y} = \mathcal{X} \times_{\mathrm{Spf}(A/I), f} \mathrm{Spec}(\mathcal{A}/p)$.

Then the crystalline comparison of prismatic cohomology~\cite[Theorem 1.8.(1)]{BS19} gives us
\[
\label{naive comparison}
\tag{\epsdice{1}}
\varphi_{\mathcal{A}}^*(\mathrm{R\Gamma}_{\Prism}(\mathcal{X}/A) 
\widehat{\otimes}_{A,f} \mathcal{A}) \cong
\varphi_{\mathcal{A}}^*(\mathrm{R\Gamma}_{\Prism}(\mathcal{Y}/\mathcal{A})) \cong 
\mathrm{R\Gamma_{crys}}(\mathcal{Y}/\mathcal{A}) \cong
\varphi_{\mathcal{A}}^*(\mathrm{R\Gamma_{crys}}(\mathcal{X}/\mathcal{A})).
\]
Here the last isomorphism comes from the following commutative diagram
\[
\xymatrix{
\mathcal{A} \ar[d]_{\varphi_{\mathcal{A}}} \ar@{->>}[r] & A/(I,p) \ar[d]^{f} \\
\mathcal{A} \ar@{->>}[r] & \mathcal{A}/p.
}
\]
In the following, we aim at getting a Frobenius descent of the isomorphism obtained in~\epsdice{1}, see~\Cref{Frobenius pullback is naive}.




\section{Comparing prismatic and derived de Rham cohomology}

Let $(A,I)$ be a bounded prism. 
Let $X$ be a $p$-adic formal scheme which is formally smooth over $\Spf(A/I)$.
In this section we shall establish a functorial comparison between the prismatic cohomology
$\mathrm{R\Gamma}_{\Prism}(X/A)$ with the derived de Rham cohomology
$\dR_{X/A}^{\wedge}$.

\subsection{The comparison}

In the beginning of this subsection we need to comment on an error in the construction of
\v{C}ech--Alexander complex in \cite[Construction 4.16]{BS19}.
We learned this subtlety from Bhatt who was informed by Koshikawa.
The issue is as follows, with notation as in loc.~cit.: suppose $D \rightarrow D/ID \leftarrow R$
is an object in $(R/A)_{\Prism}$, then one needs to exhibit a morphism $(B\{\frac{J}{I}\}^\wedge \to D)$
in $(R/A)_{\Prism}$.
The argument was along the following line, by universal property it suffices to exhibit a map
$B \to D$ sending $J$ into $ID$, which is amount to filling in the following dotted arrow
(of $\delta$-rings)
\[
\xymatrix{
R \ar[r] & D/ID \\
B \ar[u] \ar@{-->}[r] & D \ar[u]
}
\]
that makes the diagram commutative.
At first sight this seems easy, as $B$ is a free $\delta$-ring in a set of variables, we just lift
images of those variables under $B \to R \to D/ID$ to $D$ to get a map of $\delta$-rings.
But there is no way a general lift will make the above diagram commutative for the $\delta$'s of those variables.

Below we describe a fix that we learned from Bhatt.
Recall that the forgetful functor from $\delta$-$A$-algebras to $A$-algebras admits a left adjoint,
see \cite[Remark 2.7]{BS19}.
One checks the following easily:
\begin{itemize}
\item given a derived $(p,I)$-completed polynomial $A$-algebra $P$ which is freely
generated by a set of variables, then apply this left adjoint we will get a
derived $(p,I)$-completed free $\delta$-$A$-algebra
$F$ generated by the same set of variables.
\item this left adjoint commutes with completed tensor product.
\end{itemize}
In particular the natural map $P \to F$ is $(p,I)$-completely ind-smooth.

\begin{construction}[{\v{C}ech--Alexander complex for prismatic cohomology}]
\label{Cech--Alexander complex}
Let $R$ be a $p$-completely smooth $A/I$-algebra.
Let $P$ be a derived $(p,I)$-completed polynomial $A$-algebra along with a surjection $P \twoheadrightarrow R$,
and let $J$ be the kernel.
Associated with the triple $A \to P \twoheadrightarrow R$ is a $\delta$-$A$-algebra 
$F\{\frac{JF}{I}\}^\wedge$, obtained by applying \cite[Corollary 3.14]{BS19}.
We make three claims about this construction.

\begin{claim}
\label{claim about Cech--Alexander construction}
\leavevmode
\begin{enumerate}
    \item The $\delta$-$A$-algebra $F\{\frac{JF}{I}\}^\wedge$ is naturally an object in $(R/A)_{\Prism}$;
    \item as such, it is weakly initial in $(R/A)_{\Prism}$; and
    \item if there is a set of triples $A \to P_i \twoheadrightarrow R$, then the coproduct of 
    associated $F_i\{\frac{J_iF_i}{I}\}^\wedge$ in $(R/A)_{\Prism}$ is
    given by the $\delta$-$A$-algebra associated with the triple
    $A \to \widehat{\otimes}_A P_i \twoheadrightarrow R$ where the second map is given by the completed
    tensor of those $P_i \twoheadrightarrow R$ maps.
\end{enumerate}
\end{claim}

Let us postpone the verification of these claims and continue with the construction.
At this point we may simply follow the rest of \cite[Construction 4.16]{BS19}.
Form the derived $(p,I)$-completed \v{C}ech nerve $P^{\bullet}$ of $A \to P$,
and let $J^{\bullet} \subset P^{\bullet}$ be the kernel of the augmentation map $P^{\bullet} \to P \to R$.
By the first claim above,
we get a cosimplicial object $\left(F^{\bullet}\{\frac{J^\bullet F^\bullet}{I}\}^\wedge\right)$ in $(R/A)_{\Prism}$.
The third claim above shows that this is the \v{C}ech nerve of $F\{\frac{JF}{I}\}^\wedge$ in $(R/A)_{\Prism}$,
and according to the second claim the object $F\{\frac{JF}{I}\}^\wedge$ covers the final object of the topos
$\mathrm{Shv}((R/A)_{\Prism})$.
Therefore $\Prism_{R/A}$ is computed by $F^{\bullet}\{\frac{J^\bullet F^\bullet}{I}\}^\wedge$.

This construction commutes with base change of the prism $(A,I)$.
When $(A,I)$ is fixed, this construction can be carried out in a way which is strictly functorial in $R$,
by setting $P$ to be the completed polynomial $A$-algebra generated by the underlying set of $R$.
\end{construction}

\begin{proof}[Proof of {\Cref{claim about Cech--Alexander construction}}]
Proof of (1): form the following pushout diagram:
\[
\xymatrix{
R \ar[r] & S \\
P \ar[u] \ar[r] & F \ar[u].
}
\]
Denote $F\{\frac{JF}{I}\}^\wedge$ by $C^0$, by its defining property there is a natural map 
$S \cong F/JF \to C^0/IC^0$. Hence $C^0$ gives rise to a diagram $(C^0 \to C^0/IC^0 \leftarrow S \leftarrow R)$
which is an object in $(R/A)_{\Prism}$.

Proof of (2) and (3): this follows from chasing through universal properties.
Let $(D \to D/ID \leftarrow R)$ be an object in $(R/A)_{\Prism}$, we have the following chain of equivalences:
\begin{align*}
F\{\frac{JF}{I}\}^\wedge \to D \text{ in } (R/A)_{\Prism} \iff 
\text{ a map of $\delta$-$A$-algebras }  F \to D \text{ such that } JF \text{ is mapped into } ID \\
\iff \text{ a map of $A$-algebras } P \to D \text{ such that } J \text{ is mapped into } ID.
\end{align*}
It is easy to see that the last statement is equivalent to filling in the following dotted arrow
below
\[
\xymatrix{
R \ar[r] & D/ID \\
P \ar[u] \ar@{-->}[r] & D \ar[u]
}
\]
as $A$-algebras, making the diagram commutative.
Note that there is no requirement from $\delta$-ring consideration here.
Now one checks the claims (2) and (3) easily.
\end{proof}

With the above preparatory discussion, we are ready to compare prismatic cohomology
and derived de Rham cohomology.
The key computation we need is the following.

\begin{lemma}[Comparing prismatic and PD envelopes for regular sequences]
\label{envelopes for regular sequence}
Let $B$ be a $(p,I)$-completely flat $\delta$-$A$-algebra, let $f_1, \ldots, f_r \in B$
be a $(p,I)$-completely regular sequence.
Write $J = (I, f_1, \ldots, f_r) \subset B$.
Then we have a natural identification of $p$-completely flat $\dR_{A}(I)^\wedge$-algebras:
\[
B\{\frac{J}{I}\}^{\wedge} \widehat{\otimes}_{B,\varphi_B} B \widehat{\otimes}_{A} \dR_{A}(I)^\wedge
\cong \dR_{B}(J)^\wedge.
\]
\end{lemma}

Here the $B\{\frac{J}{I}\}^{\wedge}$ is as in~\cite[Proposition 3.13]{BS19},
which is $(p,I)$-completely flat over $A$.
Let us clarify the various completions involved on the left hand side. 
First we perform derived $(p,I)$-complete tensor,
then we perform derived $(p,\varphi(I))$-complete tensor
which is the same as derived $p$-complete tensor as $\varphi(I) = (p)$ in
$\pi_0(\dR_A(I)^\wedge)$.

\begin{proof}
Recall that in the proof of \cite[Proposition 3.13]{BS19} and also explained in Remark \ref{rem-2.7}, the $B\{\frac{J}{I}\}^{\wedge}$
is constructed as $p$-completely pushout the following diagram:
\[
\xymatrix{
A\{x_1, \ldots, x_r\} \ar[r]^-{x_i \mapsto f_i} \ar[d] & B \\
A\{x_1, \ldots, x_r\}\{\frac{x_i}{I}\}. & 
}
\]
The left hand side in this Lemma is therefore given by pushing-out the above diagram
further along $f_B \colon B \xrightarrow{\varphi_B} B \to B \widehat{\otimes}_{A} \dR_{A}(I)^\wedge$.
The composition
$A\{x_1, \ldots, x_r\} \xrightarrow{x_i \mapsto f_i}  B \xrightarrow{f_B} 
B \widehat{\otimes}_{A} \dR_{A}(I)^\wedge$ can now be factored as
$A\{x_1, \ldots, x_r\} \xrightarrow{\varphi_A, x_i \mapsto \varphi(z_i)} A\{z_1, \ldots, z_r\}
\to \dR_{A\{z_1, \ldots, z_r\}}(I)^\wedge \to B \widehat{\otimes}_{A} \dR_{A}(I)^\wedge$,
where in the last map $z_i$ is sent to $f_i \otimes 1$.
Hence the left hand side becomes the $p$-completely 
outer pushout of the following diagram with solid arrows:
\[
\xymatrix{
A\{x_1, \ldots, x_r\} \ar[r] \ar[d] & \dR_{A\{z_1, \ldots, z_r\}}(I)^\wedge 
\ar[r] \ar@{.>}[d] & B \widehat{\otimes}_{A} \dR_{A}(I)^\wedge \ar@{.>}[d] \\
A\{x_1, \ldots, x_r\}\{\frac{x_i}{I}\} \ar@{.>}[r] & \dR_{A\{z_1, \ldots, z_r\}}(I, z_1, \ldots z_r) ^\wedge
\ar@{.>}[r] & \dR_{B \widehat{\otimes}_{A} \dR_{A}(I)^\wedge}(f_i \otimes 1)^\wedge
\cong \dR_{B}(J)^\wedge.
}
\]
Using (multi-variable version of)~\Cref{naive base change} we see the left square of the above
is a $p$-completely pushout.
Base change property of derived de Rham complex now shows the right square of the above
to be a $p$-completely pushout.
Here the isomorphism of the right bottom corner follows from the fact that
$(I, f_1, \ldots, f_r)$ is a Koszul regular sequence in $B$.
\end{proof}

Just like how~\cite[Proposition 3.13]{BS19} implies~\cite[Corollary 3.14]{BS19},
our~\Cref{envelopes for regular sequence} above gives us the following.

\begin{lemma}
\label{weakly initial comparison}
Let $R$ be a $p$-completely smooth $A/I$-algebra.
Let $P$ be a $p$-completed polynomial algebra over $A$,
and let $P \twoheadrightarrow R$ be a surjection of $A$-algebras
with kernel $J$.
Consider the following diagram:
\[
\xymatrix{
A/I \ar[r] & R \ar[r] & S \\
A \ar[r] \ar[u] & P \ar[r] \ar[u] & F \ar[u] \\
}
\]
where $F$ is the $p$-completed free $\delta$-$A$-algebra associated with $P$, and
$S$ is the $p$-completed tensor product $S \coloneqq R \widehat{\otimes}_{P} F$.
Then we have a natural identification of $p$-completely flat $\dR_{A}(I)^\wedge$-algebras:
\[
F\{\frac{J \cdot F}{I}\} \widehat{\otimes}_{F,\varphi_F} F \widehat{\otimes}_{A} \dR_{A}(I)^\wedge
\cong \dR_{S/F}^\wedge.
\]
\end{lemma}

\begin{proof}
Zariski locally on $\Spf(P)$ and $\Spf(F)$, the kernel $J$ and $J \cdot F$
is a colimit of the form considered in \Cref{envelopes for regular sequence}.
Also note that $F/J \cdot F \cong S$, by definition we have $\dR_F(J \cdot F)^\wedge \cong \dR_{S/F}^\wedge$.

Since formation of $p$-complete derived de Rham complex commutes with taking
$p$-complete colimit (of the algebra over $A$) and descends from $p$-completely flat covers,
we may glue the local isomorphisms obtained in 
\Cref{envelopes for regular sequence} and take colimit to get our identification here.
\end{proof}

Using this comparison of prismatic envelope and derived de Rham complex, we get a comparison
between prismatic and derived de Rham cohomology as follows.

\begin{theorem}
\label{comparing pris and crys}
Let $(A,I)$ be a bounded prism.
For any $p$-completely smooth $A/I$-algebra $R$, there is a natural isomorphism
in $\mathrm{CAlg}(A)$:
\[
\mathrm{R\Gamma_{\Prism}}(R/A) \widehat{\otimes}_{A,\varphi_A} A 
\widehat{\otimes}_A \dR_A(I)^\wedge \cong
\dR_{R/A}^\wedge,
\]
which is functorial in $A/I \to R$ and satisfies base change in $(A,I)$.
\end{theorem}

Let us emphasize again that when $(A, I)$ is transversal, this follows from \cite[Theorem 5.2]{BS19}.



\begin{proof}

Let us first construct the desired natural morphism
\[
\mathrm{R\Gamma_{\Prism}}(R/A) \widehat{\otimes}_{A,\varphi_A} A 
\widehat{\otimes}_A \dR_A(I)^\wedge \rightarrow
\dR_{R/A}^\wedge.
\]
Given any triple $A \to P \twoheadrightarrow R$ as in the setting of \Cref{weakly initial comparison},
we have a natural morphism
\[
\mathrm{R\Gamma_{\Prism}}(R/A) \widehat{\otimes}_{A,\varphi_A} A 
\widehat{\otimes}_A \dR_A(I)^\wedge \rightarrow
F\{\frac{J \cdot F}{I}\}^{\wedge} \widehat{\otimes}_{F,\varphi_F} F \widehat{\otimes}_{A}
\dR_A(I)^\wedge \cong \dR_{S/F}^\wedge,
\]
which is functorial in $A \to P \twoheadrightarrow R$.
By \Cref{lim dR} (3), the limit of right hand side over
all the triples $A \to P \twoheadrightarrow R$ is just $\dR_{R/A}^\wedge$,
hence we get the desired natural morphism.
It is functorial in $A/I \to R$ and satisfies base change in $(A,I)$.


Now we need to show the natural arrow constructed above is a natural isomorphism.
Let us make more reductions.
It suffices to check this is an isomorphism
after a faithfully flat cover, and since both sides commute
with base change in $A$, we may Zariski localize on $A$,
hence we may first reduce to the case where $A$ is oriented, i.e.~$I = (d)$.
Observe that both sides are the left Kan extension of their restriction
to the category of polynomial $A$-algebras,
so it suffices to show that 
the above arrow is a natural isomorphism for algebras of the form
$R = A/I[X_1, \ldots, X_n]^{\wedge}$,
which is the base change of $p$-complete polynomial algebras 
over the universal oriented prism.
Hence we can reduce further to the case that $A$ is the universal oriented prism.
In particular we may assume that $(A,I)$ is transversal and that $\varphi_A$ is flat.

Lastly, we shall prove the statement under the assumption that 
$(A,I)$ is transversal and that $\varphi_A$ is flat.
Choose a $(p,I)$-completely polynomial $A$ algebra $P$ 
with a surjection of $A$-algebras
$P \twoheadrightarrow R$, form the cosimplicial object 
$\left(F^{\bullet}\{\frac{J^\bullet F^\bullet}{I}\}^\wedge\right)$ in $(R/A)_{\Prism}$
computing $\Prism_{R/A}$ as in \Cref{Cech--Alexander complex}.
Notice that we have an identification of cosimplicial $(p,I)$-complete algebras 
$A \xrightarrow{\simeq} F^{\bullet}$.

Since we have reduced ourselves to the case where $(A,I)$ is transversal and that $\varphi_A$ is flat, 
using~\Cref{calA Tor dim 1}, the natural morphism considered above gives rise to the following identification
\begin{align*}
\label{long chain of identifications comparing prism and dR}
\tag{\epsdice{2}}
\mathrm{R\Gamma_{\Prism}}(R/A) \widehat{\otimes}_{A,\varphi_A} A 
\widehat{\otimes}_A \dR_A(I)^\wedge \cong
\lim_{\Delta} \left( (F^{\bullet}\{ \frac{F^{\bullet} \cdot J^{\bullet}}{I} \}^{\wedge}) 
\widehat{\otimes}_{A,\varphi_A} A  \widehat{\otimes}_A \dR_A(I)^\wedge
\right)
\\
\cong \lim_{\Delta} \left((F^{\bullet}\{ \frac{F^{\bullet} \cdot J^{\bullet}}{I} \}^{\wedge})
\widehat{\otimes}_{F^{\bullet}, \varphi_{F^{\bullet}}} 
F^{\bullet} \widehat{\otimes}_A \dR_A(I)^\wedge \right)
\cong \lim_{[n] \in \Delta} \dR_{S^{\hat{\otimes}_R n}/F^n}^\wedge \cong
\dR_{R/A}^\wedge
\end{align*}
as desired.
Let us comment on the identifications above.
Here we have used the cosimplicial replacement 
$(A, \varphi_A) \xrightarrow{\simeq} (F^{\bullet}, \varphi_{F^{\bullet}})$ in the second identification.
The second-to-last identification is provided by \Cref{weakly initial comparison},
and the last identification is because of \Cref{lim dR}.
\end{proof}


\begin{remark}
\label{compatible with Frobenii}
In this paper we have only defined Frobenius action on $\dR_{-/A}^\wedge$ under the assumption
of $A$ being a $p$-torsionfree $\delta$-ring.
Now suppose $(A,I)$ is a $p$-torsionfree prism, 
by \Cref{another Frobenius on limit dR}, we see that
the chain of identifications in (\epsdice{2})
is compatible with Frobenius.
Consequently, the identification in \Cref{comparing pris and crys} is compatible with Frobenius
in a functorial manner.

We expect however that one can remove the $p$-torsionfree condition with additional work
in developing the framework of ``derived $\delta$-rings''.
However since the primary interest of this paper is in the case of $p$-torsionfree prisms,
we choose to not pursue that level of generality here.
\end{remark}

Below let us conclude two consequences from \Cref{comparing pris and crys}.

\begin{corollary}
\label{comparying pris and crys for transversal prisms}
Let $(A,I)$ be a bounded prism.
For any $p$-completely smooth $A/I$-algebra $R$, there is a natural isomorphism
in $\mathrm{CAlg}(A)$:
\[
\mathrm{R\Gamma_{\Prism}}(R/A) \widehat{\otimes}_{A,\varphi_A} A 
\widehat{\otimes}_A \mathcal{A} \cong
\mathrm{R\Gamma_{crys}}(R/\mathcal{A}),
\]
which is functorial in $A/I \to R$ and satisfies base change in $(A,I)$.
\end{corollary}

\begin{proof}
This follows from \Cref{comparing pris and crys}: simply base change along the morphism
$\dR_A(I)^\wedge \to \mathcal{A}$, and by \Cref{dR and crys} we have
\[
\dR_{R/A}^\wedge \widehat{\otimes}_{\dR_A(I)^\wedge} \mathcal{A} \cong \mathrm{R\Gamma_{crys}}(R/\mathcal{A}).
\]
\end{proof}

\begin{remark}
\label{Frobenius pullback is naive}
By chasing diagram, one verifies that
the following diagram of isomorphisms:
\[
\xymatrix{
\varphi_{\mathcal{A}}^*\left(\mathrm{R\Gamma_{\Prism}}(R/A) \widehat{\otimes}_{A,\varphi_A} A 
\widehat{\otimes}_A \mathcal{A}\right) \ar[r]^-{\alpha}    \ar[d]^-{\beta} & 
\varphi_{\mathcal{A}}^*\left(\mathrm{R\Gamma_{crys}}(R/A)\right) \ar[d]^-{\epsilon} \\
\varphi_{\mathcal{A}}^*\left(\mathrm{R\Gamma_{\Prism}}((R \widehat{\otimes}_{A, \varphi_A} \mathcal{A})/\mathcal{A}) \right) \ar[r]^-{\gamma} & 
\mathrm{R\Gamma_{crys}}(\left(R/p \otimes_{A/(p, I), \varphi_A} A/(p, I)\right)/\mathcal{A})
}
\]
is commutative, since all comparisons here are expressed in terms of various explicit
envelopes.
Here these arrows are given by:
\begin{enumerate}
\item $\alpha$ is Frobenius pullback of the arrow in~\Cref{comparying pris and crys for transversal prisms};
\item $\beta$ is the base change of prisms $\varphi_A \colon (A,I) \to (\mathcal{A},p)$;
\item $\gamma$ is the crystalline comparison for crystalline prisms~\cite[Theorem 5.2]{BS19};
\item $\epsilon$ is the base change of crystalline cohomology.
\end{enumerate}
\end{remark}

A bounded prism $(A,I)$ is called a \emph{PD prism}, if there is a PD structure $\gamma$ on $I$, compatible with the canonical one on $(p)$.

\begin{corollary}
\label{comparying pris and crys for PD prisms}
Let $(A,I,\gamma)$ be a bounded PD prism.
Then for any $p$-completely smooth $A/I$-algebra $R$, there is a natural isomorphism
in $\mathrm{CAlg}(A)$:
\[
\mathrm{R\Gamma_{\Prism}}(R/A) \widehat{\otimes}_{A,\varphi_A} A 
\cong \mathrm{R\Gamma_{crys}}(R/(A,I,\gamma)),
\]
which is functorial in $A/I \to R$ and satisfies base change in $(A,I)$.
\end{corollary}

Here $\mathrm{R\Gamma_{crys}}(-/(A,I,\gamma))$ denotes the crystalline cohomology
with respect to the $p$-adic PD base $(A,I,\gamma)$.

\begin{proof}
The additional PD structure gives us a section $\mathcal{A} \to A$, which makes the composition of
\[
A \to \dR_A(I)^\wedge \to \mathcal{A} \to A
\]
being identity.
Take the functorial isomorphism in \Cref{comparying pris and crys for transversal prisms}, 
and base change further along $\mathcal{A} \to A$ gives us
\[
\mathrm{R\Gamma_{\Prism}}(R/A) \widehat{\otimes}_{A,\varphi_A} A  \cong 
\mathrm{R\Gamma_{crys}}(R/\mathcal{A}) \widehat{\otimes}_{\mathcal{A}} A,
\]
and the latter is naturally isomorphic to $\mathrm{R\Gamma_{crys}}(R/(A,I,\gamma))$ 
due to base change in crystalline cohomology theory.
\end{proof}

\begin{remark}
(1) Any derived $p$-complete $\delta$-ring $A$ with bounded $p$-torsion together with the ideal $(p)$
is a PD prism. In this situation, our \Cref{comparying pris and crys for PD prisms}
is simply the crystalline comparison in \cite[Theorem 1.8.(1)]{BS19}.

(2) The left hand side of this comparison does not depend on the PD structure $\gamma$ on $I$,
whereas the right hand side \emph{a priori} does.
Therefore this comparison tells us that the right hand side also does not depend on the
PD structure $\gamma$.
\end{remark}

We can ``globalize'' these comparisons to general quasi-compact quasi-separated smooth formal schemes
over $\Spf(A/I)$.

\begin{theorem}
\label{global comparison}
Let $(A,I)$ be a bounded prism.
Let $X \to \Spf(A/I)$ be a quasi-compact quasi-separated smooth morphism of formal schemes.
Then we have natural isomorphisms in $\mathrm{CAlg}(A)$:
\[
\mathrm{R\Gamma_{\Prism}}(X/A) \widehat{\otimes}_{A,\varphi_A} A 
\widehat{\otimes}_A \dR_A(I)^\wedge \cong
\mathrm{R\Gamma}(X, \dR_{-/A}^\wedge);
\]
and
\[
\mathrm{R\Gamma_{\Prism}}(X/A) \widehat{\otimes}_{A,\varphi_A} A 
\widehat{\otimes}_A \mathcal{A} \cong
\mathrm{R\Gamma_{crys}}(X/\mathcal{A}).
\]

If $(A,I, \gamma)$ is a PD prism, then we have a natural isomorphism in $\mathrm{CAlg}(A)$:
\[
\mathrm{R\Gamma_{\Prism}}(X/A) \widehat{\otimes}_{A,\varphi_A} A 
\cong \mathrm{R\Gamma_{crys}}(X/(A,I,\gamma)).
\]

All the isomorphisms above satisfy base change in $(A,I)$.
Moreover, if $X$ is also proper over $\Spf(A/I)$, then all the completed tensor products above
may be replaced by tensor products.
\end{theorem}

\begin{proof}
Since $X$ is assumed to be quasi-compact and quasi-separated. These cohomologies
are computed by a finite limit of the corresponding cohomologies of affine opens of $X$.
Because completed tensor commutes with finite limit, the comparisons here follow from
\Cref{comparing pris and crys}, \Cref{comparying pris and crys for transversal prisms}, 
and \Cref{comparying pris and crys for PD prisms}.

To justify the replacement of completed tensor with tensor, just note that $\mathrm{R\Gamma_{\Prism}}(X/A)$
is a perfect complex of $A$-modules for smooth proper $X \to \Spf(A/I)$, 
see the last sentence of \cite[Theorem 1.8]{BS19}.
\end{proof}

\subsection{Functorial endomorphisms of derived de Rham complex}
Throughout this subsection, we assume $(A,I)$ to be a transversal prism,
in particular we have $\dR_A(I)^\wedge \cong \mathcal{A}$
and $\dR_{R/A}^\wedge \cong \mathrm{R\Gamma_{crys}}(R/\mathcal{A})$
where $R$ is any $p$-adic formally smooth $A/I$-algebra, see \Cref{dR and crys}.

In this subsection, we aim at understanding all
functorial endomorphisms of the derived de Rham complex functor,
under this transversality assumption.
In particular we shall see that the functorial isomorphism
\[
\mathrm{R\Gamma_{\Prism}}(R/A) \widehat{\otimes}_{A,\varphi_A} A 
\widehat{\otimes}_A \dR_A(I)^\wedge \rightarrow
\dR_{R/A}^\wedge,
\]
appearing in~\Cref{comparing pris and crys} is unique
if we assume $(A,I)$ to be a transversal prism.
In order to show this, we need to first extend the natural isomorphism
to a larger class of $A/I$-algebras.

\begin{construction}[{c.f.~\cite[Construction 7.6]{BS19} and~\cite[Example 5.12]{BMS2}}]
\label{derived prismatic construction}
Fix a bounded prism $(A,I)$, consider the functor 
$R \mapsto \dR_{R/A}^{\wedge}$
on $p$-completely smooth $A/I$-algebras $R$ valued in the category of commutative
algebras in the $\infty$-category of $p$-complete objects in $\mathcal{D}(A)$.
Left Kan extend it to all derived $p$-complete simplicial $A/I$-algebras,
which is nothing but the $p$-adic derived de Rham complex relative to $A$,
still denoted by $\dR_{R/A}^{\wedge}$.
Let us record some properties of this construction:
\begin{enumerate}
\item Since $R$ is an $A/I$-algebra, the $\dR_{R/A}^{\wedge}$ is naturally a
$\dR_{(A/I)/A}^{\wedge}$-algebra.
Hence we may actually view the functor as taking values in the category
$\mathrm{CAlg}(\dR_A(I)^\wedge)$.
\item The formation of $\dR_{R/A}^{\wedge}$ commutes with base change in $A$.
\item Below we shall see that, following the reasoning of \cite[Theorem 3.1 and Example 5.12]{BMS2},
the association $R \mapsto \dR_{R/A}^{\wedge}$ defines
a sheaf on the relative quasisyntomic site $\mathrm{qSyn}_{A/I}$.
\item By left Kan extending the natural isomorphism obtained in \Cref{comparing pris and crys},
we get an isomorphism of sheaves:
\[
\Prism^{(1)}_{R/A} \widehat{\otimes}_A \dR_A(I)^\wedge
\cong \dR_{R/A}^{\wedge},
\]
which is compatible with base change in $A$.
Here $\Prism^{(1)}_{R/A} \coloneqq \Prism_{R/A} \widehat{\otimes}_{A,\varphi_A} A$
is the Frobenius pullback of the derived prismatic cohomology.

\item Moreover if we assume that $(A,I)$ is a transversal prism, 
then for any $R$ which is large quasisyntomic over $A/I$,
the value $\dR_{R/A}^{\wedge}$ is $p$-completely flat over $\mathcal{A}$ 
and lives in cohomological degree $0$.
\end{enumerate}

Let us justify the claim (3) above.
\begin{proposition}
\label{sheaf property on relative to A dR}
The association $R \mapsto \dR_{R/A}^{\wedge}$ defines
a sheaf on the relative quasisyntomic site $\mathrm{qSyn}_{A/I}$.
\end{proposition}

\begin{proof}
Let $R \to S$ be a quasisyntomic cover of objects in $\mathrm{qSyn}_{A/I}$,
with C\v{e}ch nerve $S^{\bullet}$.
Our task is to show $\dR_{R/A}^{\wedge} = \lim_{\Delta^{\mathrm{op}}} \dR_{S^{\bullet}/A}^{\wedge}$.
Since both sides are $p$-complete, we may check this after derived modulo $p$.
Below we shall always use $-/p$ to denote derived modulo $p$ .

Now we follow closely the argument in \cite[Example 5.12]{BMS2}, correcting a typo thereof.
First there is a functorial exhaustive increasing $\mathbb{N}$-index filtration, i.e.~the conjugate filtration,
on $\dR_{R/A}/p \cong \dR_{(R/p)/(A/p)}$ 
with graded pieces given by $\big(\wedge^i_{(R/p)^{(1)}}\mathbb{L}_{(R/p)^{(1)}/(A/p)}\big)[-i]$
(and similarly for $\dR_{(S^{\bullet}/p)/(A/p)}^{\wedge}$).
Here $(-/p)^{(1)}$ denotes the base change along the Frobenius on $A/p$,
and the loc.~cit.~has a typo of not adding this Frobenius twist.
For a discussion of $p$-complete derived de Rham complex and conjugate filtration 
in the realm of animated rings, we refer readers to \cite[p.~33-35]{KP21}.

We stare at the following diagram (with its $S^{\bullet}$ analogs in mind):
\[
\xymatrix{
R \ar[r] & R/p \ar[r]^-{\varphi_{A/p}} & (R/p)^{(1)} \\
A/I \ar[u] \ar[r] & (A/I)/p \ar[u] \ar[r]^-{\varphi_{A/p}} & ((A/I)/p)^{(1)} \cong (A/I^p)/p \ar[u] \\
A \ar[u] \ar[r] & A/p \ar[u] \ar[r] ^-{\varphi_{A/p}} & A/p. \ar[u]
}
\]
Note that every square above is Cartesian.
Base change property of cotangent complex implies that
these graded pieces $\big(\wedge^i_{(R/p)^{(1)}}\mathbb{L}_{(R/p)^{(1)}/(A/p)}\big)[-i]$
(and their $S^{\bullet}$ analogs) can be identified with:
\begin{enumerate}
    \item either $\wedge^i_R \mathbb{L}_{R/A}[-i] \otimes_{R} \varphi_{A/p,*}(R/p)^{(1)}$;
    \item or $\wedge^i_R \mathbb{L}_{R/A} \otimes_{A} \varphi_{A/p,*}(A/p)$,
\end{enumerate}
where the $A$-module (resp.~$R$-module)
structure on $\varphi_{A/p,*}(A/p)$ (resp.~$\varphi_{A/p,*}(R/p)^{(1)}$) is given by 
the top and bottom row of the above diagram.

The identification (1) above implies that all these graded pieces live in $D^{\geq -1}$.
Indeed, $\wedge^i_R \mathbb{L}_{R/A}[-i]$ as an $R$-module has Tor-amplitude in $[0,i]$,
and $(R/p)^{(1)}$ is flat over $((A/I)/p)^{(1)} \cong (A/I^p)/p$ which lives in $[-1,0]$,
so $(R/p)^{(1)}$ lives in $D^{\geq -1}$.
Similar statements for $S^{\bullet}$ hold as well.
Hence we are reduced to checking these graded pieces satisfy the descent property,
here we are using the reasoning of \cite[last sentence of Example 5.12]{BMS2}.
Now using the identification (2) above, we are reduced to flat descent for
 ``tensored'' wedge powers of cotangent complex, see \cite[Proposition 3.2]{LM21}
 (which is itself a generalization of \cite[Theorem 3.1]{BMS2}).
\end{proof}

Recall that an $A/I$-algebra is called \emph{large quasisyntomic over $A/I$} (see~\cite[Definition 15.1]{BS19})
if 
\begin{itemize}
    \item $A/I \to R$ is quasisyntomic; and
    \item there is a surjection
$A/I\langle X_j^{1/p^\infty} \mid j \in J \rangle \twoheadrightarrow R$
where $J$ is a set.
\end{itemize}

\end{construction}

The following is inspired by~\cite[Sections 10.3 and 10.4]{BLM18},
and our proof is a modification of the proof thereof.

\begin{theorem}
\label{functorial endomorphism theorem}
Let $(A,I)$ be a transversal prism, and assume that $\Spf(A/I)$ is connected.
Then
\begin{enumerate}
\item The mapping space
\[
{\rm End}_{{\rm Shv}({\qsyn}_{A/I},{\rm CAlg}(\mathcal{A}))}\left(\dR_{-/A}^\wedge, \dR_{-/A}^\wedge \right)
\]
has contractible components given by a submonoid in $\mathbb{N}$.
In particular, 
the automorphism space has only one contractible component given by identity.
\item The automorphism space
\[
\mathrm{Aut}_{{\rm Shv}({\qsyn}_{A/I}, {\rm CAlg}(A/I))}\left(\dR_{-/(A/I)}^\wedge, \dR_{-/(A/I)}^\wedge \right)
\]
has only one contractible component given by identity.
\end{enumerate}
\end{theorem}

Since $\mathcal{A}/p \to A/(I,p)$ is a locally nilpotent thickening, we get that
$\Spf(\mathcal{A})$ is also connected.
In particular, the only idempotents in $\mathcal{A}$ are $0$ and $1$.
It is easy to see that the statement concerning automorphism spaces for these functors
hold true without the connectedness assumption, as on each connected component
the automorphism must be identity.

\begin{proof}
The assertion that all components are contractible 
follows from the fact that on the basis of large quasisyntomic over $A/I$-algebras,
the sheaves $\dR_{-/A}^\wedge$ and $\dR_{-/(A/I)}^\wedge$ are discrete.

All we need to check is that there are not many functorial endomorphisms (resp.~automorphisms)
for these two sheaves.
Since (2) follows from the same proof as that of (1), let us only
present the proof of (1) here.
To simplify notation, let us denote the set of functorial endomorphisms
by $\mathrm{End}(\dR_{-/A}^\wedge)$.
By restriction, any functorial endomorphism induces a functorial endomorphism
of the functor restricted to the subcategory of $A/I$-algebras of the form
$A/I\langle X_h^{1/p^\infty} \mid h \in H \rangle$ for some set $H$.
We denote the latter monoidal space by $\mathrm{End}(\dR_{-/A}^\wedge |_{\mathrm{perf}})$,
all of whose components are also contractible by the same reasoning.
By definition there is a natural map
\[
\mathrm{res} \colon \mathrm{End}(\dR_{-/A}^\wedge) \to \mathrm{End}(\dR_{-/A}^\wedge |_{\mathrm{perf}})
\]
of monoids.

Now we make following three claims:
\begin{itemize}
\item the natural map $\mathrm{res}$ is injective;
\item the monoid $\mathrm{End}(\dR_{-/A}^\wedge |_{\mathrm{perf}})$
is a submonoid of $\mathbb{Z}$; and
\item the image of $\mathrm{res}$ is contained in $\mathbb{N}$.
\end{itemize}

Below let us show the map $\mathrm{res}$ is injective.
In other words, we need to show that any functorial endomorphism of $\dR_{-/A}^\wedge$
is determined by its restriction to the algebras of the form 
$A/I\langle X_h^{1/p^\infty} \mid h \in H \rangle$ for some set $H$.
To see this, notice that $\mathrm{qSyn}_{A/I}$ has a basis given by
large quasi-syntomic over $A/I$-algebras.
Any large quasi-syntomic over $A/I$-algebra $S$, by definition, admits a surjection
from an algebra of the form $A/I\langle X_l^{1/p^\infty} \mid l \in L \rangle$
for some set $L$.
By choosing a set of generators $\{f_j \mid j \in J\}$ of the kernel,
we may form a surjection (c.f.~\cite[proof of Proposition 7.10]{BS19})
\[
S' \coloneqq A/I\langle X_l^{1/p^\infty}, Y_j^{1/p^\infty} \mid l \in L, j \in J \rangle/(Y_j - f_j \mid j \in J)^\wedge \twoheadrightarrow S
\colon Y_j^m \mapsto 0.
\]
This induces a surjection of shifted cotangent complexes:
$\mathbb{L}_{S'/A}[-1] \twoheadrightarrow \mathbb{L}_{S/A}[-1]$, therefore it induces
a surjection of $p$-adic derived de Rham complexes:
$\dR_{S'/A}^\wedge \twoheadrightarrow \dR_{S/A}^\wedge$.
For any such $S'$, we have
\[
\dR_{S'/A}^\wedge \cong D_{\mathcal{A}\langle X_l^{1/p^\infty}, Y_j^{1/p^\infty} \mid l \in L, j \in J \rangle} (Y_j - f_j \mid j \in J)^\wedge,
\]
i.e.~$p$-completely adjoining divided powers of $Y_j - f_j$ for all $j \in J$
to $\mathcal{A}\langle X_l^{1/p^\infty}, Y_j^{1/p^\infty} \mid l \in L, j \in J \rangle$.
Since $\mathcal{A}$ is $p$-torsionfree, any endomorphism of $\dR_{S'/A}^\wedge$
is determined by its restriction to $\mathcal{A}\langle X_l^{1/p^\infty}, Y_j^{1/p^\infty} \mid l \in L, j \in J \rangle$.
Lastly, we know that applying $\dR_{-/A}^\wedge$ functor to the map
\[
A/I\langle X_l^{1/p^\infty}, Y_j^{1/p^\infty} \mid l \in L, j \in J \rangle \to S'
\]
exactly induces the natural map
\[
\mathcal{A}\langle X_l^{1/p^\infty}, Y_j^{1/p^\infty} \mid l \in L, j \in J \rangle \to \dR_{S'/A}.
\]
Therefore we know that any functorial endomorphism of $\dR_{-/A}^\wedge$ must be determined by its restriction
to algebras of the form $A/I\langle X_h^{1/p^\infty} \mid h \in H \rangle$.

Next, let us try to understand $\mathrm{End}(\dR_{-/A}^\wedge |_{\mathrm{perf}})$
and show it is a submonoid of integers.
Consider a functorial endomorphism $f$.
It is determined by its restriction to the one-variable ``perfect'' $A/I$-algebra
$R = A/I\langle X^{1/p^\infty} \rangle$.
We know $\dR_{R/A}^\wedge \cong \mathcal{A}\langle X^{1/p^\infty} \rangle$.
Suppose $f(x) = \sum_{i \in \mathbb{N}[1/p]} a_i X^i \in \mathcal{A}\langle X^{1/p^\infty} \rangle$.
Consider the map $R \to S \coloneqq A/I\langle Y^{1/p^\infty}, Z^{1/p^\infty} \rangle$ 
sending $X^i \mapsto Y^i Z^i$.
This map induces the corresponding map 
$\mathcal{A}\langle X^{1/p^\infty} \rangle \to \mathcal{A}\langle Y^{1/p^\infty}, Z^{1/p^\infty} \rangle$
which also sends $X^i \mapsto Y^i Z^i$.
Now the functoriality of the endomorphism tells us that $f(YZ) = f(Y) \cdot f(Z)$.
We immediately get
$a_i^2 = a_i$ and $a_i \cdot a_j = 0$ for any pair of distinct indices
$i, j \in \mathbb{N}[1/p]$.
By connectedness assumption of $\Spf(A/I)$, we see there is at most one index $i \in \mathbb{N}[1/p]$,
with nonzero $a_i = 1$.
To see there is at least one nonzero $a_i$, we use the map
$R \to A/I$ given by $X^i \mapsto 1$ for all $i \in \mathbb{N}[1/p]$.

We want to show the $i \in \mathbb{N}[1/p]$ got in the previous paragraph defining the functorial
endomorphism $f$ must in fact lie in $p^{\mathbb{Z}}$.
Assume $i = \frac{\ell}{p^N}$ where $\ell$ is an integer coprime to $p$.
Now we contemplate the map
$R \to S$ given by $X \mapsto \lim_n(Y^{1/p^n} + Z^{1/p^n})^{p^n}$,
it induces a map of $\dR_{-/A}^\wedge$ with the image of $X$
given by the same formula.
Functoriality of $f$ implies that we have
\[
(\lim_n(Y^{1/p^n} + Z^{1/p^n})^{p^{n-N}})^\ell = \lim_n(Y^{\ell/p^{n-N}} + Z^{\ell/p^{n-N}})^{p^n}.
\]
Reduction modulo $p$ tells us that
\[
(Y^{1/p^N} + Z^{1/p^N})^\ell = Y^{\ell/p^N} + Z^{\ell/p^N} \in 
\mathbb{F}_p[Y^{1/p^\infty},Z^{1/p^\infty}],
\]
forcing $\ell = 1$.
Therefore we see that $\mathrm{End}(\dR_{-/A}^\wedge |_{\mathrm{perf}}) \subset p^{\mathbb{Z}}$,
i.e.~it is a submonoid inside $\mathbb{Z}$.

Finally let us prove the image of $\mathrm{res}$ lands in $p^{\mathbb{N}}$.
We want to rule out negative powers of $p$.
To that end consider $R \to R/(X)$, which induces the map of $p$-adic derived de Rham complex:
\[
\widetilde{R} \coloneqq \mathcal{A}\langle X^{1/p^\infty} \rangle \rightarrow
\widetilde{S} \coloneqq D_{\mathcal{A}\langle X^{1/p^\infty}\rangle}(X)^\wedge.
\]
Here the latter denotes the $p$-complete PD envelope of the former along the ideal $X$,
and this is the natural map.
Take a positive integer $j$, and we need to argue that $X \mapsto X^{1/p^j}$
on $\widetilde{R}$ does not extend to an endomorphism of $\widetilde{S}$.
Suppose otherwise, then the extended endomorphism of $\widetilde{S}$
must send $X^p$ to $X^{p^{1-j}}$, but $X^p$ is divisible by $p$ in $\widetilde{S}$
whereas $X^{p^{1-j}}$ is not (here we use the fact that $j > 0$),
hence we get a contradiction.

The only invertible element in the additive monoid $\mathbb{N}$ is $0$,
corresponding to $X \mapsto X^{(p^0)} = X$,
hence the only functorial automorphism of $\dR_{-/A}^\wedge$ is identity.
\end{proof}

\begin{remark}
\label{functorial endo remark}
Let $(A,I)$ be a transversal prism. Then
\begin{enumerate}
\item By the same argument, there are not many functorial homomorphisms
from $\varphi_A^* \dR_{-/A}^\wedge$ to $\dR_{-/A}^\wedge$.
Similarly, these are determined by its restriction to $R = A/I\langle X^{1/p^\infty} \rangle$.
If we require the restriction sends $X$ to $X^p$, then there is a unique one
given by Frobenius constructed in \Cref{Frobenii}.
Therefore in a strong sense, there is a unique Frobenius.
\item Due to previous remark, we see the comparison in \Cref{comparing pris and crys}
must be compatible with Frobenius.
\item It is unclear which positive integer $i$, corresponding to $X \mapsto X^{p^i}$,
can occur as a functorial endomorphism.
When $A/(p,I)$ has transcendental (relative to $\mathbb{F}_p$) elements, then none of these
can occur.
This can be seen by considering the map $R \to R/(X - a)$ for some lift $a$ of the transcendental element $\bar{a} \in A/(p,I)$.
\end{enumerate}
\end{remark}

Consequently we get the following uniqueness of the functorial comparison established
in~\Cref{comparing pris and crys}, readers shall compare with~\cite[Section 18]{BS19}.

\begin{corollary}
\label{uniqueness of comparison}
Fix a transversal prism $(A,I)$.
There is a unique natural isomorphism of $p$-complete
commutative algebra objects in $\mathcal{D}(\mathcal{A})$:
\[
\mathrm{R\Gamma_{\Prism}}(R/A) \widehat{\otimes}_{A,\varphi_A} A 
\widehat{\otimes}_A \mathcal{A} \rightarrow
\mathrm{R\Gamma_{crys}}(R/A),
\]
which is functorial in the $p$-completely
smooth $A/I$-algebra $R$.
\end{corollary}

\begin{proof}
The existence part is given by~\Cref{comparing pris and crys},
we need to show uniqueness.
Suppose there are two such functorial isomorphisms.
Then consider the composition of one with the inverse of the other,
we get a natural automorphism of the functor $\dR_{-/A} \cong \mathrm{R\Gamma_{crys}}(-/A)$
on smooth $A/I$-algebras.
By left Kan extension, this will induce a natural automorphism 
of the functor $\dR_{-/A}$ on quasisyntomic $A/I$-algebras.
We conclude by~\Cref{functorial endomorphism theorem} that this automorphism must be identity.
\end{proof}

\begin{corollary}
\label{match with crys comparison in BMS1}
Let $\C$ be an algebraically closed complete non-Archimedean field extension of $\mathbb{Q}_p$,
and let $(A,I)$ be the associated perfect prism (denoted as $(A_{\inf}, \ker(\theta))$
in literature).
Then the comparison in~\Cref{comparing pris and crys} is compatible
with the crystalline comparison over $\mathcal A = A_{\mathrm{crys}}$ 
of the $A\Omega$-theory obtained in~\cite{BMS1}.
Concretely, the following diagram of isomorphisms is commutative:
\[
\xymatrix{
\mathrm{R\Gamma_{\Prism}}(R/A) \widehat{\otimes}_{A,\varphi_A} A 
\widehat{\otimes}_A \mathcal{A} \ar[r] \ar[d] &
\mathrm{R\Gamma_{crys}}(R/A) \ar[d] \\
A\Omega(R) \widehat{\otimes}_A \mathcal{A} \ar[r] & \mathrm{R\Gamma_{crys}}((R/p)/A),
}
\]
where the left vertical arrow is given by~\cite[Theorem 17.2]{BS19}
and the bottom horizontal arrow is given by~\cite[Theorem 12.1]{BMS1} or \cite{Yao19}.
\end{corollary}

\begin{proof}
This follows from the uniqueness statement in~\Cref{uniqueness of comparison}.
\end{proof}

Both sides of the isomorphism obtained in~\Cref{comparing pris and crys}
after completely tensoring $\mathcal{A}/\mathcal{I} \cong A/I$ over $\mathcal{A}$,
are naturally isomorphic to $\dR_{R/(A/I)}^{\wedge}$.
For the left hand side this follows from the de Rham comparison
of (Frobenius pullback of) the prismatic cohomology:
\[
\Prism^{(1)}_{R/A} \widehat{\otimes}_A \mathcal{A} \widehat{\otimes}_{\mathcal{A}} A/I
\cong \Prism^{(1)}_{R/A} \widehat{\otimes}_A A/I
\cong \dR_{R/(A/I)}^{\wedge},
\]
where the last equality follows from~\cite[Theorem 6.4 or Corollary 15.4]{BS19}.
For the right hand side this is just the base change of the derived de Rham complex
(or base change of crystalline cohomology and
the comparison of de Rham and crystalline cohomology for smooth morphism).
We observe that similar argument as above forces these natural isomorphisms
to be compatible with each other.

\begin{corollary}
\label{compatibility when modulo I}
Let $(A,I)$ be a transversal prism.
The following triangle of natural isomorphisms
\[
\xymatrix{
& \dR_{R/(A/I)}^{\wedge} & \\
\Prism^{(1)}_{R/A} \widehat{\otimes}_A \mathcal{A} \widehat{\otimes}_{\mathcal{A}} A/I \ar[ru]^{\cong} \ar[rr]^{\cong} &  &
\dR_{R/A}^{\wedge}  \widehat{\otimes}_{\mathcal{A}} A/I \ar[lu]_{\cong}
}
\]
is a commutative diagram.
\end{corollary}

\begin{proof}
Observe that all three natural isomorphisms are functorial in $R$,
hence going around the circle produces a functorial automorphism of
$\dR_{R/(A/I)}^{\wedge}$.

Now we argue as in the proof of \Cref{uniqueness of comparison}: using \Cref{functorial endomorphism theorem} we may conclude that this functorial automorphism
must be identity. Hence the above diagram must commute functorially.
\end{proof}

\section{Filtrations}
\label{section filtrations}

Throughout this section, 
we assume $(A,I)$ to be a transversal prism and let $(\mathcal{A}, \mathcal{I})$ be the $p$-adic PD envelope of $A$ along $I$.
By~\Cref{comparing pris and crys}, for any $p$-completely
smooth $A/I$-algebra $R$ we have a functorial isomorphism:
\[
\varphi^*(\mathrm{R\Gamma_{\Prism}}(R/A))
\widehat{\otimes}_A \mathcal{A} \cong
\dR_{R/A}^{\wedge}.
\]
All objects involved here have interesting filtrations, they are:
Nygaard filtration on $\varphi^*(\mathrm{R\Gamma_{\Prism}}(R/A))$, $I$-adic filtration
on $A$, PD ideal filtration $\mathcal{I}^{[\bullet]}$ on $\mathcal{A}$,
and Hodge filtration on $\dR_{R/A}^{\wedge}$.
In this section, we discuss how these filtrations are related.

Unless otherwise specified,
we shall use $R$ to denote a general $A/I$-algebra, and $S$ will be used to denote
a large quasi-syntomic over $A/I$-algebra (see the discussion
right after \Cref{derived prismatic construction}).

Let us briefly remind readers how these filtrations are defined and their properties.

\subsection{Hodge filtration on $\dR_{R/A}^{\wedge}$} 
Recall that $\mathrm{R\Gamma_{crys}}(R/A)$ is the 
cohomology of the structure sheaf $\mathcal{O}_{\mathrm{crys}}$
on the (absolute) crystalline site $(R/A)_{\mathrm{crys}}$.
The crystalline structure sheaf admits a natural surjection to the Zariski structure sheaf,
whose kernel is an ideal sheaf $\mathcal{I}_{\mathrm{crys}}$ admitting divided powers.
Concretely, given a PD thickening $(U,T)$, with $U$ a $p$-adic formal $\Spf(A)$-scheme
with an $\Spf(A)$-map $U \to \mathrm{Spf}(R)$ and $U \hookrightarrow T$ a
$p$-completely nilpotent PD thickening, then
we have $\mathcal{O}_{\mathrm{crys}} \mid_{(U,T)} = \mathcal{O}_T$
and $\mathcal{I}_{\mathrm{crys}} \mid_{(U,T)} = 
\ker(\mathcal{O}_T \twoheadrightarrow \mathcal{O}_U)$ which is a PD ideal sheaf
inside $\mathcal{O}_{\mathrm{crys}}$.
For any integer $r \geq 0$, we get a natural filtration on $\mathrm{R\Gamma_{crys}}(R/A)$
given by $\mathrm{R\Gamma_{crys}}(R/A, \mathcal{I}^{[r]}_{\mathrm{crys}})$.
Results of Bhatt~\cite[Section 3.3]{Bha12} and Illusie~\cite[Section VIII.2]{Ill72}
help us to understand this natural filtration in terms of $p$-adic
derived de Rham complex and its Hodge filtrations.
\begin{theorem}[{see~\cite[Proposition 3.25 and Theorem 3.27]{Bha12} and~\cite[Corollaire VIII.2.2.8]{Ill72}}, {see also~\cite[Theorem 3.4.(4)]{GL20}}]
\label{Illusie--Bhatt}
Let $R$ be a $p$-completely locally complete intersection $A/I$-algebra.
Then there is a natural identification of filtered $\mathbb{E}_{\infty}$-$\mathcal{A}$-algebras:
\[
(\dR_{R/A}^{\wedge}, \Fil^r_{\mathrm{H}}) 
\xrightarrow{\cong} 
(\mathrm{R\Gamma_{crys}}(R/A),\mathrm{R\Gamma_{crys}}(R/A, \mathcal{I}^{[r]}_{\mathrm{crys}})).
\]
\end{theorem}

Here $\Fil^{\bullet}_{\mathrm{H}}$ denotes the (derived $p$-completed)
Hodge filtration on $\dR_{R/A}^{\wedge}$,
whose graded pieces are given by 
\[
\gr^*_{\mathrm{H}}(\dR_{R/A}^{\wedge}) \cong \Gamma^*_{R}(\LL_{R/A}^{\wedge}[-1]),
\]
where $\Gamma^*$ denotes the derived divided power algebra construction and
$\LL_{R/A}^{\wedge}$ denotes the derived $p$-completed cotangent complex
of $R$ over $A$.
The triangle $A \to A/I \to R$ now gives us a triangle relating various $p$-completed
cotangent complexes:
\[
R \widehat{\otimes}_{A/I} {I/I^2}[1] \cong R \widehat{\otimes}_{A/I} \LL_{(A/I)/A}^{\wedge}
\to \LL_{R/A}^{\wedge} \to \LL_{R/(A/I)}^{\wedge},
\]
where the (shifted) map $R \widehat{\otimes}_{A/I} {I/I^2} \to \LL_{R/A}^{\wedge}[-1]$
comes from the $\mathcal{A}$-algebra structure on $\dR_{R/A}^{\wedge}$.
Indeed the multiplicativity of Hodge filtrations and the fact that
$I/I^2 \cong \mathcal{I}/\mathcal{I}^{[2]} \cong \gr^1_{\mathrm{H}}(\dR_{(A/I)/A}^{\wedge})$
naturally sits inside $\gr^1_{\mathrm{H}}(\dR_{R/A}^{\wedge})$
giving rise to
\[
\gr^0_{\mathrm{H}}(\dR_{R/A}^{\wedge}) \widehat{\otimes}_{\gr^0_{\mathrm{H}}(\dR_{(A/I)/A}^{\wedge})} \gr^1_{\mathrm{H}}(\dR_{(A/I)/A}^{\wedge})
\to \gr^1_{\mathrm{H}}(\dR_{R/A}^{\wedge}),
\]
which is identified with the shifted map $R \widehat{\otimes}_{A/I} {I/I^2} \to \LL_{R/A}^{\wedge}[-1]$.

The above discussion naturally extends to all $A/I$-algebras via left Kan extension.
We restrict ourselves to those algebras that are quasisyntomic over $A/I$ so that everything in sight
is a sheaf with respect to the quasisyntomic topology.
Recall that a basis of the quasisyntomic site is given by algebras that are large quasisyntomic 
over $A/I$ (see~\cite[Definition 15.1]{BS19}).
Below we shall show that, on this basis, 
all these sheaves have values living in cohomological degree zero.
The proof is inspired by~\cite[Subsection 12.5]{BS19}.

\begin{lemma}
\label{lives in cohdeg 0 in char p}
Let $B$ be an $\mathbb{F}_p$-algebra and let $S$ be a $B$-algebra which is
relatively semiperfect with $\mathbb{L}_{S/B}[-1]$ given by a flat $S$-module.
Then $\dR_{S/B}$ and its Hodge filtrations all live in cohomological degree $0$.
\end{lemma}

\begin{proof}
Using the conjugate filtration and Cartier isomorphism, we see that
$\dR_{S/B}$ (being its $0$-th Hodge filtration) lives in degree $0$.
On the other hand, we also know that the graded pieces of the Hodge
filtrations are given by divided powers $\Gamma^*_S(\LL_{S/B}[-1])$,
hence all the graded pieces live in degree $0$ as well.
In order to prove the statement about Hodge filtrations, we
need to show the natural map
$\dR_{S/B} \to \dR_{S/B}/\Fil^r_{\mathrm{H}}$ is surjective
(note that both sides live in degree $0$ by last sentence).

To this end, we proceed by mimicking~\cite[proof of Theorem 12.2]{BS19}.
First we may replace $B$ by the relative perfection of $S$, as the relevant
cotangent complexes $\LL_{S/B}$ and $\LL_{S^{(1)}/B}$ are unchanged.
Hence we may assume $B \to S$ is a surjection, as $S/B$ is assumed to be
relatively semiperfect.
Next, by choosing the surjection $\mathbb{F}_p[X_b \mid b \in B] \twoheadrightarrow B$
and base change along the fully faithful map
$\mathbb{F}_p[X_b \mid b \in B] \to \mathbb{F}_p[X_b^{1/p^\infty} \mid b \in B]$,
we may further assume that $B$ is semiperfect
(as surjectiveness of a map can be tested after fully faithful base change).
In particular, any element in the kernel of $B \to S$ admits compatible $p$-power
roots in $B$.

Now if the kernel is generated by a regular sequence, then the map
$\dR_{S/B} \to \dR_{S/B}/\Fil^r_{\mathrm{H}}$ is identified as
$D_B(S) \to D_B(S)/J^{[r]}$ where $D_B(S)$ denotes the PD envelope
and $J^{[r]}$ is the $r$-th divided power ideal of $J = \ker(D_B(S) \twoheadrightarrow S)$.
Therefore $\dR_{S/B} \to \dR_{S/B}/\Fil^r_{\mathrm{H}}$ 
is surjective by this concrete description.

Lastly given any such surjection $B \twoheadrightarrow S$,
call the underlying set of its kernel by $I$.
Then we look at the surjection of $B$-algebras
\[
\widetilde{S} \coloneqq B[X_i^{1/p^\infty} \mid i \in I]/(X_i \mid i \in I) \twoheadrightarrow S,
\]
where $X_i^{1/p^\infty}$ is sent to (the image of) a compatible $p$-power roots
of the corresponding element $f_i \in I$ in $S$.
We have that the induced map $\LL_{\widetilde{S}/B}[-1] \to \LL_{S/B}[-1]$ sends
$X_i$ to $f_i$, hence is a surjection.
Therefore we get that the map
$\gr^*_{\mathrm{H}}(\dR_{\widetilde{S}/B}) \to \gr^*_{\mathrm{H}}(\dR_{S/B})$
is a surjection.
Since $\widetilde{S}$ is a quotient of a relatively perfect algebra over $B$
by an ind-regular sequence.
Applying (filtered colimit of) what we proved in the previous paragraph, we get that
$\dR_{\widetilde{S}/B} \to \dR_{\widetilde{S}/B}/\Fil^r_{\mathrm{H}}$
is also a surjection.
Looking at the following commutative diagram
\[
\xymatrix{
\dR_{\widetilde{S}/B} \ar[r] \ar@{->>}[d] & \dR_{S/B} \ar[d]    \\
\dR_{\widetilde{S}/B}/\Fil^r_{\mathrm{H}} \ar@{->>}[r] & \dR_{S/B}/\Fil^r_{\mathrm{H}},
}
\]
we conclude that the right arrow must be surjective, which is what we need to show.
\end{proof}

\begin{lemma}
\label{Hodge filtrations on dR}
Let $S$ be a large quasisyntomic over $A/I$ algebra.
Then all of the Hodge filtrations on $\dR_{S/A}^{\wedge}$ and $\dR_{S/(A/I)}^{\wedge}$
are given by submodules, equivalently all the filtrations and their graded pieces
are cohomologically supported in degree $0$.
Moreover the Hodge filtrations of $\dR_{S/(A/I)}^{\wedge}$ are $p$-completely
flat over $A/I$.
\end{lemma}

\begin{proof}
Derived modulo $p$, we see that the first claim follows from~\Cref{lives in cohdeg 0 in char p}.
Also we see that $\dR_{S/(A/I)}^{\wedge} \to \dR_{S/(A/I)}^{\wedge}/\Fil^r_{\mathrm{H}}$
 is surjective.
So the statement of $p$-completely flatness of $\dR_{S/(A/I)}^{\wedge}$
and its Hodge filtrations now follows from $p$-completeness
of $\dR_{S/(A/I)}^{\wedge}$ and the graded pieces of its Hodge filtrations.
Using conjugate filtration and Cartier isomorphism, both $p$-completely
flatness follow from the fact that $\LL_{S/(A/I)}^{\wedge}[-1]$ is $p$-completely
flat over $S$ and $S$ is $p$-completely flat over $A/I$ 
(as $S$ is large quasisyntomic over $A/I$).
\end{proof}

Since $\dR_{R/A}^{\wedge}$ is naturally an $\mathcal{A}$-algebra for any $A/I$-algebra $R$,
the filtration on $\mathcal{A}$ by the divided powers of $\mathcal{I}$ gives rise
to another functorial decreasing filtration on $\dR_{R/A}^{\wedge}$:
\[
\Fil^r_{\mathcal{I}}(\dR_{R/A}^{\wedge}) \coloneqq 
\dR_{R/A}^{\wedge} \widehat{\otimes}_{\mathcal{A}} \mathcal{I}^{[r]}.
\]
We caution readers that this is \emph{not} the $\mathcal{I}$-adic filtration,
as we are using divided powers of $\mathcal{I}$ instead of symmetric powers.
A basic understanding of these filtrations are given by the following:
\begin{lemma}
\label{calIfil in deg 0}
All of these $\Fil^r_{\mathcal{I}}(\dR_{R/A}^{\wedge})$ are quasisyntomic sheaves,
whose values on large quasisyntomic over $A/I$ algebras are supported in degree $0$.
The graded pieces are given by
\[
\gr^r_{\mathcal{I}} \cong 
\dR_{R/(A/I)}^{\wedge} \widehat{\otimes}_{A/I} \mathcal{I}^{[r]}/\mathcal{I}^{[r+1]}.
\]
\end{lemma}

\begin{proof}
The statement about graded pieces follows from the following chain of identifications
\[
\gr^r_{\mathcal{I}} \cong \dR_{R/A}^{\wedge} \widehat{\otimes}_{\mathcal{A}} \mathcal{I}^{[r]}/\mathcal{I}^{[r+1]}
\cong \dR_{R/A}^{\wedge} \widehat{\otimes}_{\mathcal{A}} \mathcal{A}/\mathcal{I}
\widehat{\otimes}_{\mathcal{A}/\mathcal{I}} \mathcal{I}^{[r]}/\mathcal{I}^{[r+1]}
\cong \dR_{R/(A/I)}^{\wedge} \widehat{\otimes}_{A/I} \mathcal{I}^{[r]}/\mathcal{I}^{[r+1]},
\]
where the last identification comes from
$\dR_{R/A}^{\wedge} \widehat{\otimes}_{\mathcal{A}} A/I \cong \dR_{R/(A/I)}^{\wedge}$
(c.f.~\cite[Proposition 3.11]{GL20})
and $\mathcal{A}/\mathcal{I} \cong A/I$.
In particular, these graded pieces are given by $\dR_{R/(A/I)}^{\wedge}$ twisted by
a rank $1$ locally free sheaf on $\Spf(A/I)$, hence are quasisyntomic sheaves themselves.

Since $\dR_{R/A}^{\wedge}$ and all these graded pieces are quasisyntomic sheaves,
each $\Fil^r_{\mathcal{I}}$ is also a quasisyntomic sheaf.

If $S$ is large quasisyntomic over $A/I$, then $\dR_{S/A}^{\wedge}$ and all these graded pieces
are supported in cohomological degree $0$ by~\Cref{Hodge filtrations on dR}.
By induction, in order to show the filtrations are in degree $0$, it suffices to show
$\dR_{S/A}^{\wedge} \widehat{\otimes}_{\mathcal{A}} \mathcal{I}^{[r]}
\to \dR_{S/A}^{\wedge} \widehat{\otimes}_{\mathcal{A}} \mathcal{I}^{[r]}/\mathcal{I}^{[r+1]}$
is surjective for any $r$, which follows from the right exactness of $p$-complete tensor.
\end{proof}

The filtration $\Fil^{\bullet}_{\mathcal{I}}(\dR_{R/A}^{\wedge})$ is a disguise of the Katz--Oda filtration 
$\Fil^{\bullet}_{\mathrm{KO}}(\dR_{C/A})$
discussed in~\cite{GL20}, applied to the triple $(A \to B \to C) = (A \to A/I \to R)$.
More precisely, we have
\[
\Fil^{i}_{\mathcal{I}}\dR_{R/A}^{\wedge} \cong \Fil^{\bullet}_{\mathrm{KO}}(\dR_{R/A})^\wedge.
\]
We refer readers to the Subsection 3.2 of loc.~cit.~for a general discussion of additional
structures on the derived de Rham complex of
$A \to C$ when it factorizes through $A \to B \to C$.

Let $R$ be an $A/I$-algebra.
By $p$-completing the double filtrations obtained in \cite[Construction 3.12]{GL20},
we see that $\dR_{R/A}^\wedge$ can be naturally equipped a decreasing filtration indexed by
$\mathbb{N} \times \mathbb{N}$:
\[
\Fil^{i,j}(\dR_{R/A}^\wedge) \coloneqq \bigg(\Fil^i_{\mathrm{KO}}\Fil^j_{\mathrm{H}}(\dR_{R/A})\bigg)^\wedge
\]
The following proposition will describe $\Fil^{i,j}(\dR_{R/A}^\wedge)$
and declare its relation with
the two systems of filtrations $\Fil^{\bullet}_{\mathrm{H}} \dR_{R/A}^{\wedge}$ and
$\Fil^{\bullet}_{\mathcal{I}} \dR_{R/A}^{\wedge}$.

\begin{proposition}
\label{general KO filtration properties}
Let $R$ be an $A/I$-algebra.
Then:
\begin{enumerate}
    \item For any $j$, we have an identification $\Fil^{0,j}(\dR_{R/A}^\wedge) \cong \Fil^j_{\mathrm{H}}(\dR_{R/A}^\wedge)$.
    \item For each pair $0 \leq j \leq i$, we have an identification 
    \[
    \Fil^{i,j}(\dR_{R/A}^\wedge) \cong \Fil^i_{\mathcal{I}}(\dR_{R/A}^\wedge).
    \]
    \item For each pair $0 \leq i \leq j$, we have a natural identification
    \[
    \mathrm{Cone}\bigg(\Fil^{i+1, j}(\dR_{R/A}^\wedge) \to \Fil^{i,j}(\dR_{R/A}^\wedge)\bigg) \cong 
    \Fil^{j-i}_{\mathrm{H}} \dR_{R/(A/I)}^\wedge \widehat{\otimes}_{A/I} \Gamma^i_{A/I}(I/I^2).
    \]
    Moreover this identification is compatible with 
    \[
    \xymatrix{
    \mathrm{Cone}\bigg(\Fil^{i+1, j}(\dR_{R/A}^\wedge) \to \Fil^{i,j}(\dR_{R/A}^\wedge)\bigg) \ar[d] \ar[r]^-{\cong} &
    \Fil^{j-i}_{\mathrm{H}} \dR_{R/(A/I)}^\wedge \widehat{\otimes}_{A/I} \Gamma^i_{A/I}(I/I^2) \ar[d] \\
    \mathrm{Cone}\bigg(\Fil^{i+1, 0}(\dR_{R/A}^\wedge) \to \Fil^{i,0}(\dR_{R/A}^\wedge)\bigg) \ar[r]^-\cong &
    \dR_{R/(A/I)}^\wedge \widehat{\otimes}_{A/I} \Gamma^i_{A/I}(I/I^2).
    }
    \]
    \item The association $R \mapsto \Fil^{i,j}(\dR_{R/A}^\wedge)$ defines a sheaf on the quasi-syntomic site of $A/I$
    for any $(i,j)$.
\end{enumerate}
\end{proposition}

\begin{proof}
For (1): this follows from the \cite[Construction 3.12]{GL20},
$\Fil^{0,j}$ is the $p$-completed $j$-th filtration on 
$\dR_{R/A} \otimes_{\dR_A(I)} \Fil^0_{\mathrm{H}}(\dR_A(I))
\cong \dR_{R/A}$.
Since this is a filtered isomorphism, we see that this is nothing but
$p$-completed $j$-th Hodge filtration on $\dR_{R/A}$,
hence it is $\Fil^j_{\mathrm{H}}(\dR_{R/A}^\wedge)$.

For (2): this follows from the \cite[Construction 3.9]{GL20}.
Indeed, the inequality $j \leq i$ implies that the $Fil^j$ of each term
showing in \cite[Construction 3.9]{GL20} is the whole term.
Hence the colimit just gives $\dR_{R/A} \otimes_{\dR_A(I)} \Fil^i_{\mathrm{H}}(\dR_A(I))$
back.
After $p$-completing, we see that by definition we have
$\Fil^{i,j}(\dR_{R/A}^\wedge) \cong \Fil^i_{\mathcal{I}}(\dR_{R/A}^\wedge)$.

(3) follows from $p$-completing \cite[Proposition 3.13.(1)]{GL20}.

For (4): first we claim the associations $R \mapsto \Fil^m_{\mathrm{H}}(\dR_{R/A}^\wedge)$
and $R \mapsto \Fil^n_{\mathrm{H}}(\dR_{R/(A/I)}^\wedge)$ define sheaves
for all $m$ and $n$.
For $m=0$ this is \Cref{sheaf property on relative to A dR},
and for $n=0$ this is \cite[Example 5.12]{BMS2}.
Induction on $m$ and $n$ reduces us to showing the sheaf property of graded pieces,
which are given by $\wedge^i_R \mathbb{L}_{R/A}^{\wedge}[-i]$
and $\wedge^i_R \mathbb{L}_{R/(A/I)}^{\wedge}[-i]$.
$p$-Completing \cite[Theorem 3.1]{BMS2} gives the desired sheaf property of these graded pieces.

Fix a natural number $j$, then by (1) we see that $\Fil^{0,j}$
is a quasisyntomic sheaf.
Each graded piece with respect to $i$, by (2) and (3), is also a sheaf.
Therefore we see that by increasing induction on $i$, each $\Fil^{i,j}$ defines a sheaf.
\end{proof}

To understand these sheaves more concretely, we look at their value on the basis of
large quasisyntomic over $A/I$-algebras.

\begin{proposition}
\label{qsyn KO filtration properties}
Let $S$ be a large quasisyntomic over $A/I$ algebra.
Then:
\begin{enumerate}
    \item For any pair $(i,j) \in \mathbb{N} \times \mathbb{N}$, the $\Fil^{i,j}(\dR_{S/A}^\wedge)$
    is concentrated in degree $0$, and the natural map $\Fil^{i,j}(\dR_{S/A}^\wedge) \to \dR_{S/A}^\wedge$
    is injective.
    \item For any $j$, the natural map
    \[
    \Fil^j_{\mathrm{H}}(\dR_{S/A}^\wedge) \to \Fil^j_{\mathrm{H}}(\dR_{S/(A/I)}^\wedge)
    \]
    is surjective.
    \item For each pair $0 \leq i \leq j$, we have an equality:
    \[
    \Fil^{i,j}(\dR_{S/A}^{\wedge})
    = \sum_{r = i}^j \left(\Fil^{j-r}_{\mathrm{H}} \dR_{S/A}^{\wedge} \cdot \mathcal{I}^{[r]}\right),
    \]
    where $\Fil^{j-r}_{\mathrm{H}} \dR_{S/A}^{\wedge} \cdot \mathcal{I}^{[r]}$ denotes the image of
    $\Fil^{j-r}_{\mathrm{H}} \dR_{S/A}^{\wedge} \widehat{\otimes}_{\mathcal{A}} \mathcal{I}^{[r]}
    \to \dR_{S/A}^\wedge$,
    and the sum is inside the algebra $\dR_{S/A}^{\wedge}$.
    \item We have another description:
    \[
    \Fil^{i,j}(\dR_{S/A}^{\wedge}) = \left(\Fil^j_{\mathrm{H}} \dR_{S/A}^{\wedge}\right) \cap 
     \left(\Fil^i_{\mathcal{I}} \dR_{S/A}^{\wedge}\right),
    \]
    where the intersection happens inside the algebra $\dR_{S/A}^{\wedge}$.
\end{enumerate}
\end{proposition}

\begin{proof}
Proof of (1): we shall prove by decreasing induction on $i$.
When $j \leq i$, by \Cref{general KO filtration properties} (2) we see that 
$\Fil^{i,j}(\dR_{S/A}^\wedge) \cong \Fil^i_{\mathcal{I}}(\dR_{S/A}^\wedge)$,
which is concentrated in degree $0$ by \Cref{calIfil in deg 0}.
By \Cref{general KO filtration properties} (3), the graded pieces with respect to $i$ are all concentrated in degree $0$ by \Cref{Hodge filtrations on dR}.
This in turn implies that,
\begin{itemize}
    \item All of $\Fil^{i,j}(\dR_{S/A}^\wedge)$ are in degree $0$ for any $(i,j)$; and
    \item We have short exact sequences: 
    \[
    0 \to \Fil^{i+1,j}(\dR_{S/A}^\wedge) \to \Fil^{i,j}(\dR_{S/A}^\wedge)
    \to \Fil^{j-i}_{\mathrm{H}} \dR_{R/(A/I)}^\wedge \widehat{\otimes}_{A/I} \Gamma^i_{A/I}(I/I^2) \to 0.
    \]
\end{itemize}
In particular $\Fil^{i+1,j}(\dR_{S/A}^\wedge) \to \Fil^{i,j}(\dR_{S/A}^\wedge)$
is injective.
Using \Cref{general KO filtration properties} (1) and \Cref{Hodge filtrations on dR},
we see that the map $\Fil^{0,j}(\dR_{S/A}^\wedge) \cong \Fil^j_{\mathrm{H}}(\dR_{S/A}^\wedge) \to \dR_{S/A}^\wedge$ is also injective.
Therefore the composition $\Fil^{i,j}(\dR_{S/A}^\wedge) \to \dR_{S/A}^\wedge$
is injective as well for any $(i,j)$.

(2) follows from the short exact sequence obtained in the previous
paragraph, specializing to $i = 0$.

(3) follows from the combination of (2), \Cref{general KO filtration properties} (3),
and the fact that $p$-completed tensor is right exact.

For (4): first notice that this is true for $i = 0$, due to
\Cref{general KO filtration properties} (1).
Next let us look at the commutative diagram in \Cref{general KO filtration properties} (3).
Since the right hand side is an injection, we see that the map
\[
\Fil^{i,j}(\dR_{S/A}^\wedge)/\Fil^{i+1, j}(\dR_{S/A}^\wedge)
\to \Fil^{i,0}(\dR_{S/A}^\wedge)/\Fil^{i+1, 0}(\dR_{S/A}^\wedge)
\]
is injective.
Therefore, by \Cref{general KO filtration properties} (2), we know that
\[
\Fil^{i+1, j}(\dR_{S/A}^\wedge) = \left(\Fil^{i,j}(\dR_{S/A}^\wedge)\right) \cap 
\left(\Fil^{i+1}_{\mathcal{I}}\dR_{S/A}^\wedge\right).
\]
By increasing induction on $i$, we may assume
\[
\Fil^{i,j}(\dR_{S/A}^\wedge) = 
\left(\Fil^j_{\mathrm{H}} \dR_{S/A}^{\wedge}\right) \cap 
\left(\Fil^i_{\mathcal{I}} \dR_{S/A}^{\wedge}\right).
\]
Hence we have
\[
\Fil^{i+1, j}(\dR_{S/A}^\wedge) =
\left(\Fil^j_{\mathrm{H}} \dR_{S/A}^{\wedge}\right) \cap 
\left(\Fil^i_{\mathcal{I}} \dR_{S/A}^{\wedge}\right) \cap
\left(\Fil^{i+1}_{\mathcal{I}}(\dR_{S/A}^\wedge)\right)
= \left(\Fil^j_{\mathrm{H}} \dR_{S/A}^{\wedge}\right) \cap 
\left(\Fil^{i+1}_{\mathcal{I}}\dR_{S/A}^\wedge\right).
\]
\end{proof}

Let us draw a table to summarize these filtrations on $\dR_{R/A}^{\wedge}$:
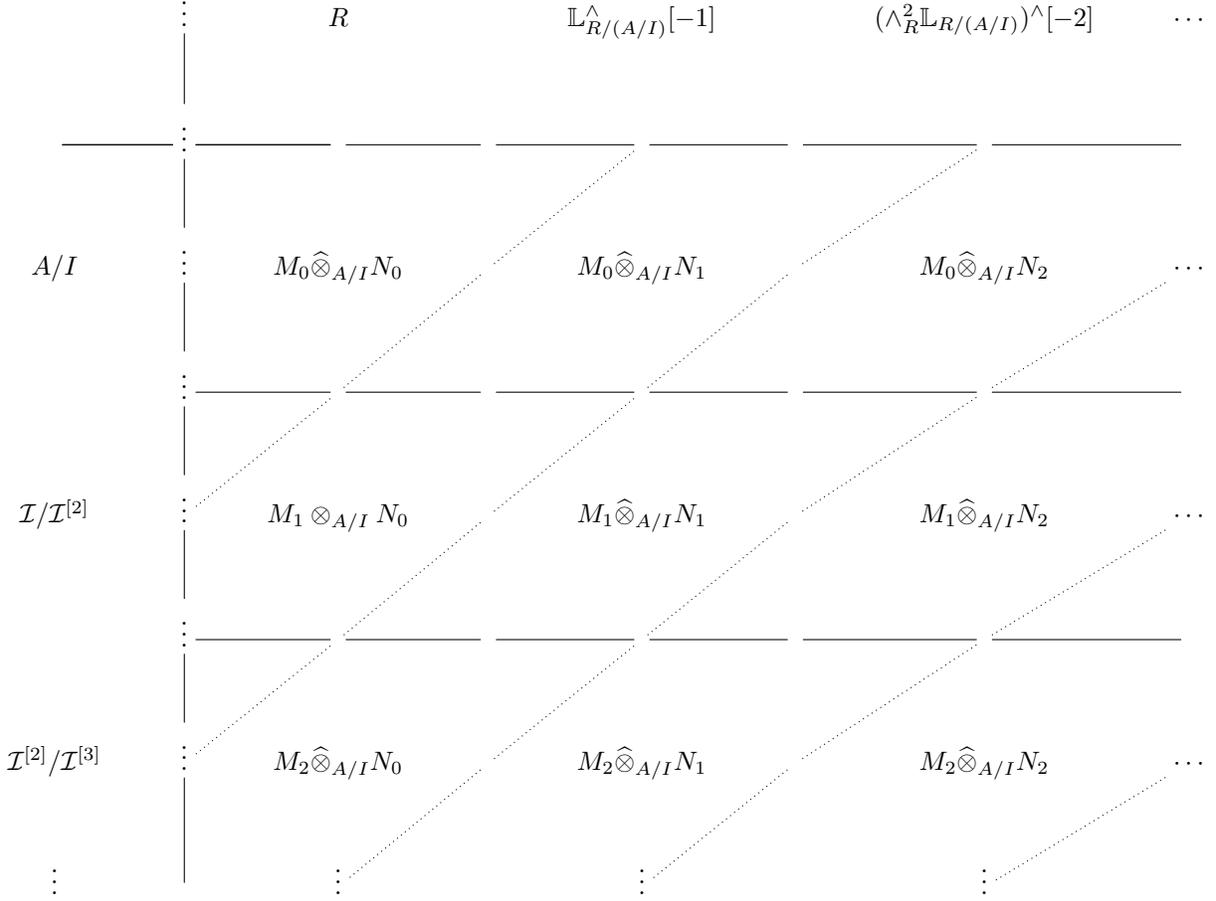
\[
\xymatrix{
  & \vdots \ar@{-}[d] & R  &  & \LL_{R/(A/I)}^{\wedge}[-1]  &   & 
  (\wedge^2_R \LL_{R/(A/I)})^{\wedge}[-2]  &  \cdots  \\
\ar@{-}[r] & \vdots \ar@{-}[d] \ar@{-}[l] \ar@{-}[r] & \ar@{-}[l]  \ar@{-}[r]  & \ar@{-}[r] & \ar@{-}[r] \ar@{.}[ld]  & \ar@{-}[r] & \ar@{-}[r] \ar@{.}[ld]   &    \\
A/I  & \vdots \ar@{-}[d] &  M_0 \widehat{\otimes}_{A/I} N_0  & \ar@{.}[ld] &  M_0 \widehat{\otimes}_{A/I} N_1  & \ar@{.}[ld]  &  M_0 \widehat{\otimes}_{A/I} N_2 & \ar@{.}[ld]  \cdots  \\
 & \vdots \ar@{-}[d] & \ar@{-}[l] \ar@{.}[ld]  \ar@{-}[r]  & \ar@{-}[r] & \ar@{-}[r] \ar@{.}[ld]  & \ar@{-}[r] & \ar@{-}[r] \ar@{.}[ld]   &    \\
\mathcal{I}/\mathcal{I}^{[2]} & \vdots \ar@{-}[d]  &  M_1 \otimes_{A/I} N_0   & \ar@{.}[ld] &  
M_1 \widehat{\otimes}_{A/I} N_1   & \ar@{.}[ld] &  M_1 \widehat{\otimes}_{A/I} N_2 & \ar@{.}[ld]  \cdots  \\
 & \vdots \ar@{-}[d] &  \ar@{-}[l] \ar@{-}[r] \ar@{.}[ld]  &  \ar@{-}[r] &  \ar@{-}[r] \ar@{.}[ld] &  \ar@{-}[r] & \ar@{.}[ld]  \ar@{-}[r] &    \\
\mathcal{I}^{[2]}/\mathcal{I}^{[3]} & \vdots \ar@{-}[d] &  M_2 \widehat{\otimes}_{A/I} N_0   & \ar@{.}[ld] &  M_2 \widehat{\otimes}_{A/I} N_1  & \ar@{.}[ld] & M_2 \widehat{\otimes}_{A/I} N_2 & \ar@{.}[ld]  \cdots  \\
  \vdots    &  & \vdots &    &  \vdots  &   &  \vdots  &       \\
}
\]
In the diagram above, $M_i=\mathcal{I}^{[i]}/\mathcal{I}^{[i+1]}$, and $N_j=(\wedge^j_R \LL_{R/(A/I)})^{\wedge}[-j]$, for $i,j\in \NN$.
Here rows indicate graded pieces of the filtration $\Fil^r_{\mathcal{I}}$,
and each term in $i$-th row indicates the graded piece of the induced filtration on
$\dR_{R/(A/I)}^{\wedge} \widehat{\otimes}_{A/I} \Gamma^i_{A/I}(I/I^2)$.
The skewed dotted line indicate the Hodge filtration on $\dR_{R/A}^{\wedge}$
(given by things below the dotted line).
See also \cite[p.10]{GL20}.

As a consequence we get a structural result on the graded algebra associated with
the Hodge filtration on $\dR_{R/A}^{\wedge}$.
\begin{lemma}
\label{increasing filtration on Hodge graded}
There is a functorial increasing exhaustive filtration $\Fil^v_i$ on the graded algebra
 $\gr^*_{\mathrm{H}}(\dR_{R/A}^{\wedge})$ by
graded-$\left(\gr^*_{\mathcal{I}} \mathcal{A} \cong \Gamma^*_{A/I}(I/I^2)\right)$-submodules
with graded pieces given by
\[
\gr^v_i\left(\gr^*_{\mathrm{H}}(\dR_{R/A}^{\wedge})\right) \cong
(\wedge^i_R \LL_{R/(A/I)})^{\wedge}[-i] \widehat{\otimes}_{A/I} \Gamma^*_{A/I}(I/I^2).
\]
Here $(\wedge^i_R \LL_{R/(A/I)})^{\wedge}[-i]$ has degree $i$
and the above is a graded isomorphism.
\end{lemma}

We refer to this filtration $\Fil^v_i$ on $\gr^*_{\mathrm{H}}(\dR_{R/A}^{\wedge})$
as the \emph{vertical filtration} from now on,
c.f.~\cite[Construction 3.14]{GL20}.
This choice of name is because the $\Fil^v_i$ is literally the filtration given by vertical columns
in the table before this Lemma.

\begin{proof}
Use the above table one can see this directly.
Equivalently, we may use
\[
\gr^*_{\mathrm{H}}(\dR_{R/A}^{\wedge}) \cong
\left(\Gamma^*_R(\LL_{R/A}^{\wedge}[-1])\right)^{\wedge},
\]
and the triangle
\[
R \widehat{\otimes}_{A/I} I/I^2 \to \LL_{R/A}^{\wedge}[-1] \to \LL_{R/(A/I)}^{\wedge}[-1].
\]
\end{proof}

\begin{remark}
\label{flatness of dR}
Let $(A,I)$ be a general bounded prism, and let $S$ be a large quasisyntomic over $A/I$-algebra.
Combining \Cref{comparing pris and crys}, \Cref{derived prismatic construction} (4),
and \cite[Theorem 15.2.(1)]{BS19}, we can see that $\dR_{S/A}^\wedge$ is $p$-completely flat over $\dR_A(I)^\wedge$.

Below is suggested to us by Bhatt.
Using conjugate filtration and the same argument of \Cref{increasing filtration on Hodge graded},
we can give an alternative proof of this fact.
Indeed we can check this after mod $p$, hence we shall assume $A$ to be $p$-torsion. 
Next we want to appeal to the conjugate filtrations on both algebras:
we have the following pushout diagram:
\[
\xymatrix{
A \ar[r] & A/I^p \ar[r] & R^{(1)} \\
A \ar[u]^-{\varphi_A} \ar[r] & A/I \ar[r] \ar[u] & R \ar[u].
}
\]
There is a similar functorial increasing exhaustive filtration on the graded
algebra of the conjugate filtered $\dR_{S/A}$,
with graded pieces given by 
$(\wedge^i_{R^{(1)}} \LL_{R^{(1)}/(A/I^p)})[-i] {\otimes}_{A/I^p} \Gamma^*_{A/I^p}(I^p/I^{2p})$.
It is flat over $\Gamma^*_{A/I^p}(I^p/I^{2p})$, which is the conjugate graded algebra
of $\dR_A(I)$.
Lastly we conclude by recalling that an increasingly exhaustive filtered module
of an increasingly exhaustive filtered algebra is flat if the graded counterpart
is flat.
\end{remark}

\subsection{Nygaard filtration}
\label{subsec-NygaardFil}
Recall in~\cite[Section 15]{BS19}, there is a natural decreasing filtration
of quasisyntomic subsheaves on $\Prism^{(1)}_{-/A}$ called the Nygaard filtration with the following properties:
\begin{theorem}[{see~\cite[Theorem 15.2 and 15.3]{BS19} and proof therein}]
\label{Nygaard filtration}
Let $S$ be a large quasisyntomic over $A/I$ algebra. Then
\begin{enumerate}
\item The Nygaard filtrations $\Fil^{\bullet}_{\mathrm{N}}$ on $\Prism^{(1)}_{S/A}$
are given by $p$-completely flat $A$-submodules inside $\Prism^{(1)}_{S/A}$.
\item We have an identification of algebras $\Prism^{(1)}_{S/A}/I \cong \dR_{S/(A/I)}^{\wedge}$,
under which the image of Nygaard filtration becomes the Hodge filtration.
\item For each $i \geq 0$, we have a short exact sequence:
\[
0 \to \Fil^{i}_{\mathrm{N}}\Prism^{(1)}_{S/A} \otimes_A I \to \Fil^{i+1}_{\mathrm{N}}\Prism^{(1)}_{S/A} \to 
\Fil^{i+1}_{\mathrm{H}} \dR_{R/(S/I)}^{\wedge} \to 0.
\]
\end{enumerate}
\end{theorem}

Let $R$ be a general quasisyntomic $A/I$-algebra.
On $\Prism^{(1)}_{R/A}$ there is also an $I$-adic filtration 
$\Fil^r_{I}\Prism^{(1)}_{R/A} \coloneqq \Prism^{(1)}_{R/A} \otimes_A I^r$,
by~\Cref{Nygaard filtration} (2), we identify the graded pieces as
\[
\gr^r_{I} \cong \Prism^{(1)}_{R/A}/I \otimes_{A/I} I^r/I^{r+1}
\cong \dR_{R/(A/I)}^{\wedge} \otimes_{A/I} \Sym_{A/I}^r(I/I^2).
\]
The $I$-adic filtration and the Nygaard filtration are related by the following.
For any $(i,j) \in \mathbb{N} \times \mathbb{N}$, we define 
\[
\Fil^{i,j} \Prism^{(1)}_{R/A} \coloneqq \Fil^{j-i}_{\mathrm{N}}\Prism^{(1)}_{R/A} \otimes_A I^i,
\]
where we adopt the convention that 
$\Fil^{l}_{\mathrm{N}}\Prism^{(1)}_{R/A} = \Prism^{(1)}_{R/A}$ if $l \leq 0$.
One checks easily that this puts a decreasing filtration on $\Prism^{(1)}_{R/A}$
indexed by $\mathbb{N} \times \mathbb{N}$.
This filtration has very similar behavior as the $\Fil^{i,j}(\dR_{R/A}^\wedge)$ studied in previous subsection.
The following is the analogue of \Cref{general KO filtration properties}.

\begin{proposition}
\label{general KON filtration properties}
Let $R$ be an $A/I$-algebra. Then:
\begin{enumerate}
\item For any $j$, we have $\Fil^{0,j}\Prism^{(1)}_{R/A} \cong \Fil^j_{\mathrm{N}}\Prism^{(1)}_{R/A}$.
\item For each pair $0 \leq j \leq i$, we have
\[
\Fil^{i,j}\Prism^{(1)}_{R/A} \cong \Fil^i_{I}\Prism^{(1)}_{R/A}.
\]
\item For each pair $0 \leq i \leq j$, we have a
natural identification
\[
\mathrm{Cone}\bigg(\Fil^{i+1,j}\Prism^{(1)}_{R/A} \to \Fil^{i,j}\Prism^{(1)}_{R/A}\bigg) \cong 
\Fil^{j-i}_{\mathrm{H}}(\dR_{R/(A/I)}^\wedge) \otimes_{A/I} \Sym_{A/I}^i(I/I^2).
\]
Moreover these identifications fit in the following commutative diagram:
\[
\xymatrix{
\mathrm{Cone}\bigg(\Fil^{i+1,j}\Prism^{(1)}_{R/A} \to \Fil^{i,j}\Prism^{(1)}_{R/A}\bigg) \ar[r]^-{\cong} \ar[d] &
\Fil^{j-i}_{\mathrm{H}}(\dR_{R/(A/I)}^\wedge) \otimes_{A/I} \Sym_{A/I}^i(I/I^2) \ar[d] \\
\mathrm{Cone}\bigg(\Fil^{i+1,0}\Prism^{(1)}_{R/A} \to \Fil^{i,0}\Prism^{(1)}_{R/A}\bigg) \ar[r]^-{\cong} & 
\dR_{R/(A/I)}^\wedge \otimes_{A/I} \Sym_{A/I}^i(I/I^2)
}.
\]
\item The association $R \mapsto \Fil^{i,j}\Prism^{(1)}_{R/A}$ defines
a sheaf on $\mathrm{qSyn}_{A/I}$ for any $(i,j)$.
\end{enumerate}
\end{proposition}

\begin{proof}
(1) and (2) follows from definition. (3) follows from \Cref{Nygaard filtration} (3). (4) follows from (3).
\end{proof}

\begin{proposition}
\label{qsyn KON filtration properties}
Let $S$ be a large quasisyntomic over $A/I$ algebra.
Then:
\begin{enumerate}
    \item We have an equality:
    \[
    \Fil^{i,j}\Prism^{(1)}_{S/A}
    = \sum_{r = i}^j \left(\Fil^{j-r}_{\mathrm{N}}(\Prism^{(1)}_{S/A}) \cdot I^{r}\right),
    \]
    where the sum is inside the algebra $\Prism^{(1)}_{S/A}$.
    \item We have another equality:
    \[
    \Fil^{i,j}\Prism^{(1)}_{S/A} = \left(\Fil^j_{\mathrm{N}} \Prism^{(1)}_{S/A}\right) \cap 
     \left(\Fil^i_{I} \Prism^{(1)}_{S/A}\right),
    \]
    where the intersection happens inside the algebra $\Prism^{(1)}_{S/A}$.
\end{enumerate}
\end{proposition}

\begin{proof}
The proof is similar to \Cref{qsyn KO filtration properties} (3) and (4).
Notice that $\Fil^j_{\mathrm{N}} \Prism^{(1)}_{S/A} \to \Fil^j_{\mathrm{H}}(\dR_{R/(A/I)}^\wedge)$ is surjective
by \Cref{Nygaard filtration} (2).
\end{proof}

We can express all these structures on $\Prism^{(1)}_{R/A}$
in the following graph similar to what was drawn
in the previous subsection.
One observes that the distinction is just that divided powers of $I/I^2$ get replaced by symmetric powers of $I/I^2$.
\[
\xymatrix{
  & \vdots \ar@{-}[d] & R  &  & \LL_{R/(A/I)}^{\wedge}[-1]  &   & 
  (\wedge^2_R \LL_{R/(A/I)})^{\wedge}[-2]  &  \cdots  \\
\ar@{-}[r] & \vdots \ar@{-}[d] \ar@{-}[l] \ar@{-}[r] & \ar@{-}[l]  \ar@{-}[r]  & \ar@{-}[r] & \ar@{-}[r] \ar@{.}[ld]  & \ar@{-}[r] & \ar@{-}[r] \ar@{.}[ld]   &    \\
A/I  & \vdots \ar@{-}[d] &  M_0 \widehat{\otimes}_{A/I} N_0  & \ar@{.}[ld] &  M_0 \widehat{\otimes}_{A/I} N_1  & \ar@{.}[ld]  &  M_0 \widehat{\otimes}_{A/I} N_2 & \ar@{.}[ld]  \cdots  \\
 & \vdots \ar@{-}[d] & \ar@{-}[l] \ar@{.}[ld]  \ar@{-}[r]  & \ar@{-}[r] & \ar@{-}[r] \ar@{.}[ld]  & \ar@{-}[r] & \ar@{-}[r] \ar@{.}[ld]   &    \\
I/I^2 & \vdots \ar@{-}[d]  &  M_1 \otimes_{A/I} N_0   & \ar@{.}[ld] &  
M_1 \widehat{\otimes}_{A/I} N_1   & \ar@{.}[ld] &  M_1 \widehat{\otimes}_{A/I} N_2 & \ar@{.}[ld]  \cdots  \\
 & \vdots \ar@{-}[d] &  \ar@{-}[l] \ar@{-}[r] \ar@{.}[ld]  &  \ar@{-}[r] &  \ar@{-}[r] \ar@{.}[ld] &  \ar@{-}[r] & \ar@{.}[ld]  \ar@{-}[r] &    \\
I^2/I^3 & \vdots \ar@{-}[d] &  M_2 \widehat{\otimes}_{A/I} N_0   & \ar@{.}[ld] &  M_2 \widehat{\otimes}_{A/I} N_1  & \ar@{.}[ld] & M_2 \widehat{\otimes}_{A/I} N_2 & \ar@{.}[ld]  \cdots  \\
  \vdots    &  & \vdots &    &  \vdots  &   &  \vdots  &       \\
}
\]

Here rows indicate graded pieces of the filtration $\Fil^r_{I}$,
and each term in each row indicates the graded piece of the Hodge filtration on
$\dR_{R/(A/I)}^{\wedge}$.
The skewed dotted line indicate the Nygaard filtration on $\Prism^{(1)}_{R/A}$
(given by things below the dotted line).

Also as a consequence we get a structural result on the graded algebra associated with
the Nygaard filtration on $\Prism^{(1)}_{R/A}$.
\begin{lemma}
\label{increasing filtration on Nygaard graded}
There is a functorial increasing exhaustive filtration $\Fil^v_i$ on the graded algebra
 $\gr^*_{\mathrm{N}}(\Prism^{(1)}_{R/A})$ by
graded-$\left(\gr^*_{I} A \cong \Sym^*_{A/I}(I/I^2)\right)$-submodules
with graded pieces given by
\[
\gr^v_i\left(\gr^*_{\mathrm{N}}(\Prism^{(1)}_{R/A})\right) \cong
(\wedge^i_R \LL_{R/(A/I)})^{\wedge}[-i] \widehat{\otimes}_{A/I} \Sym^*_{A/I}(I/I^2).
\]
Here $(\wedge^i_R \LL_{R/(A/I)})^{\wedge}[-i]$ has degree $i$ and the above is a graded isomorphism.
\end{lemma}

We also call this filtration $\Fil^v_i$ on $\gr^*_{\mathrm{N}}(\Prism^{(1)}_{R/A})$
as the \emph{vertical filtration} from now on.

\begin{proof}
This follows from~\Cref{Nygaard filtration} (3), see also the proof of \Cref{increasing filtration on Hodge graded}.
\end{proof}

\subsection{Promote to filtered map}
Recall that we use $(A,I)$ to denote a transversal prism.
For the rest of this section, we shall use $(B,J)$ to denote a general bounded prism.
By~\Cref{comparing pris and crys}, we have a map
$\Prism^{(1)}_{R/B} \to \dR_{R/B}^{\wedge}$
functorial in $B/J \to R$.
The goal of this subsection is to show that this map can be
promoted to a filtered map where the left hand side is equipped with the Nygaard filtration
and the right hand side is equipped with the Hodge filtration.
Our plan is:
\begin{itemize}
\item show certain rigidity of the map being filtered;
\item show the map is filtered when the base prism $(A,I)$ is transversal;
\item show the map is filtered when the algebra $R$ is large quasisyntomic over $B/J$
of a particular type; and
\item show the map is functorially filtered when $R$ is a $p$-completely smooth $B/J$-algebra,
and hence finish the argument by left Kan extension.
\end{itemize}

Fix a natural number $i$. 
The main diagram that we shall stare at in this subsection is:
\[
\label{NygHdg diagram}
\tag{\epsdice{3}}
\xymatrix{
\Fil^i_{\mathrm{N}} \ar[r] \ar@{.>}[d]^-{g_i} & \Prism^{(1)} \ar[r] \ar[d] & Q_{1,i} \ar@{.>}[d]^-{f_i} \\
\Fil^i_{\mathrm{H}} \ar[r] & \dR \ar[r] & Q_{2,i},
}
\]
viewed as a commutative diagram of sheaves on $\qsyn_{B/J}$.
Here $Q_{1,i}$ and $Q_{2,i}$ are the cones of the natural maps, so both rows are distinguished triangles
of quasisyntomic sheaves.
All the solid arrows are defined: for instance, the middle vertical arrow is given by \Cref{comparing pris and crys}.
Our main task is to show that one can fill in the dotted arrows $f_i$ and $g_i$ making the diagram commute.

We first need a few lemmas to illustrate that the situation is pretty rigid and there is at most one choice of
these dotted arrows.
\begin{lemma}
\label{cohomology estimate}
Let $S$ be a large quasisyntomic over $B/J$ algebra, then the values of
\[
\Fil^i_{\mathrm{N}},~\Prism^{(1)},~Q_{1,i},~\text{and } Q_{2,i}
\]
at $S$ are concentrated in cohomological degree $0$.
\end{lemma}

\begin{proof}
The first three follows from how they are defined, see \cite[Subsection 15.1]{BS19}.
The claim for $Q_{2,i}$ follows from the fact that $\mathbb{L}_{S/B}^\wedge[-1]$ lives in cohomological degree $0$.
\end{proof}

\begin{lemma}
\label{rigidity of the situation}
Let $S$ be a large quasisyntomic over $B/J$ algebra. Then,
\begin{enumerate}
\item there is at most one choice of $f_i$ making the right square of~\ref{NygHdg diagram} commute;
\item if $S \to T$ is a morphism of large quasisyntomic over $B/J$ algebras,
and suppose $f_i$ are defined on both of them, then the diagram
\[
\xymatrix{
Q_{1,i}(S) \ar[r] \ar[d]^{f_i(S)} & Q_{1,i}(T) \ar[d]^{f_i(T)} \\
Q_{2,i}(S) \ar[r] & Q_{2,i}(T)
}
\]
is commutative;
\item the existence of $f_i(S)$ is equivalent to the existence of $g_i(S)$ making the left square of~\ref{NygHdg diagram}
commute;
\item the $g_i(S)$, if exists, must be unique.
\end{enumerate} 
\end{lemma}

\begin{proof}
(1): suppose there are two of them, and take their difference. Since precomposing with
$\Prism^{(1)} \to Q_{1,i}$ of this difference is the zero map $\Prism^{(1)} \xrightarrow{0} Q_{2,i}$ due to commutativity,
the difference must factor through $\Fil^i_{\mathrm{N}}[1]$. 
But $\mathrm{Hom}(\Fil^i_{\mathrm{N}}[1], Q_{2,i}) = \{0\}$ by cohomological considerations in \Cref{cohomology estimate}.
Hence the difference must be zero.

(2): the argument is similar to (1). The difference of the two arrows from $Q_{1,i}(S)$ to $Q_{2,i}(T)$ will again factor
through $\Fil^i_{\mathrm{N}}(S)[1]$, hence must again be zero.

(3): just apply TR3, notice that the two rows of~\ref{NygHdg diagram} are exact triangles.

(4): similar to (1). The difference of two possible $g_i(S)$'s factors through an arrow 
$\Fil^i_{\mathrm{N}}(S) \to Q_{2,i}[-1](S)$ which is again zero by cohomological considerations.
\end{proof}

After knowing the rigidity of our situation, we shall start proving the existence of $f_i$
following the plan outlined above.

\begin{proposition}
Let $(A,I)$ be a transversal prism, and let 
$S = A/I \langle X_1^{1/p^\infty}, \ldots, X_n^{1/p^\infty} \rangle/(f_1, \ldots, f_r)$
where $(f_i)$ is a $p$-completely regular sequence.
We have $\Fil^i_{\mathrm{N}} \Prism^{(1)}_{S/A} \subset \Fil^i_{\mathrm{H}}
\dR_{S/A}^{\wedge}$.
\end{proposition}

\begin{proof}
When $i = 0$, there is nothing to prove,
when $i = 1$, the triangle in Corollary \ref{compatibility when modulo I} gives us
a commutative diagram
\[
\xymatrix{
& S & \\
\Prism^{(1)}_{S/A} \ar@{->>}[ru] \ar[rr] &  &
\dR_{S/A}^{\wedge},  \ar@{->>}[lu]
}
\]
since the kernels of these two surjections define the first 
Nygaard and Hodge filtrations respectively, we see the containment for $i = 1$.
For general $i$, we prove by induction. Let us look at the induced map
$g \colon \Fil^i_{\mathrm{N}} \Prism^{(1)}_{S/A} \to \dR_{S/A}^{\wedge}/\Fil^i_{\mathrm{H}}$.
We first notice that by induction and $I \subset \mathcal{I}$
the submodule $I \cdot \Fil^{i-1}_{\mathrm{N}} \Prism^{(1)}_{S/A}$
is sent to zero under $g$.
By multiplicativity and the containment for $i = 1$, 
we have that $\Sym^i(\Fil^1_{\mathrm{N}} \Prism^{(1)}_{S/A})$
is also sent to zero under $g$.
Now we use~\Cref{Nygaard filtration} (3) to see that
$\Fil^i_{\mathrm{N}}/I \cdot \Fil^{i-1}_{\mathrm{N}}$ is identified with
$\Fil^i_{\mathrm{H}} \dR_{S/(A/I)}^{\wedge}$
and the image of $\Sym^i(\Fil^1_{\mathrm{N}} \Prism^{(1)}_{S/A})$ becomes
$\Sym^i(\Fil^1_{\mathrm{H}} \dR_{S/(A/I)}^{\wedge})$, so we get an induced map
$\bar{g} \colon \left(
\Fil^i_{\mathrm{H}} \dR_{S/(A/I)}^{\wedge}/\Sym^i(\Fil^1_{\mathrm{H}} \dR_{S/(A/I)}^{\wedge})
\right) \to \dR_{S/A}^{\wedge}/\Fil^i_{\mathrm{H}}$.
But the source of this map has its $p$-power torsions submodule being $p$-adically dense
and the target of this map is $p$-torsionfree and $p$-adically complete,
so the map $\bar{g}$ must in fact be zero.
This proves the containment $\Fil^i_{\mathrm{N}} \subset \Fil^i_{\mathrm{H}}$ as claimed.
\end{proof}

The following is inspired by \cite[Subsection 12.4]{BS19}.

\begin{proposition}
\label{filtered map for transversal}
Let $(A,I)$ be a transversal prism, then for any $p$-completely
smooth $A/I$-algebra $R$, the map $\Prism^{(1)}_{R/A} \to \dR_{R/A}^{\wedge}$
can be promoted to a map of filtered algebras.
Moreover this lift is functorial in the $A/I$-algebra $R$, hence left Kan extends
to all animated $A/I$-algebras.
\end{proposition}

\begin{proof}
For any surjection $A/I\langle X_1, \ldots, X_n\rangle \to R$,
the ring 
\[
\tilde{R} = A/I\langle X_1^{1/p^\infty}, \ldots, X_n^{1/p^\infty} \rangle 
\otimes_{A/I\langle X_1, \ldots, X_n\rangle} R
\]
is large quasisyntomic over $A/I$ and Zariski locally of the form considered in the previous proposition.
Therefore the map $\Prism^{(1)}_{\tilde{R}/A} \to \dR_{\tilde{R}/A}^{\wedge}$
is canonically filtered: (Zariski locally)
existence follows from the previous proposition,
Zariski glue as well as uniqueness is provided by \Cref{rigidity of the situation}.
The same applies to all terms of the \v{C}ech nerve $\tilde{R}^{\bullet}$ of
$R \to \tilde{R}$.

The filtered cosimplicial rings $\Prism^{(1)}_{\tilde{R}^{\bullet}/A}$
and $\dR_{\tilde{R}^{\bullet}/A}^{\wedge}$ computes the filtered rings
$\Prism^{(1)}_{\tilde{R}/A}$ and $\dR_{\tilde{R}/A}^{\wedge}$ separately.
By \Cref{rigidity of the situation}, we get a map of filtered cosimplicial rings.

This construction is independent of the choice of the surjection
$A/I\langle X_1, \ldots, X_n\rangle \to R$: adding extra variables to the $X_i$,
one gets a square of maps between filtered cosimplicial algebras,
we use \Cref{rigidity of the situation} to see the maps commute for each term
associated with $[m] \in \Delta$.
Since the category of such surjections admits pairwise coproducts, it is therefore sifted.
The naturality in $R$ follows from exactly the same argument.
This way we get the desired functorial map.
\end{proof}

Now we turn to the general situation where the base prism $(B,J)$ is not necessarily
transversal.
We bootstrap the previous two propositions.

\begin{proposition}
Let $S = B/J \langle X_1^{1/p^\infty}, \ldots, X_n^{1/p^\infty} \rangle/(f_1, \ldots, f_r)$
where $(f_i)$ is a $p$-completely regular sequence.
Then there exists $f_i(S)$ and $g_i(S)$ making the diagram~\ref{NygHdg diagram} commute.
\end{proposition}

\begin{proof}
We shall utilize the knowledge when the base prism is transversal.
Without loss of generality, let us assume $(B,J) = (B, (d))$ is oriented:
Zariski locally it is oriented, and the locally defined $f_i(S)$ and $g_i(S)$
will necessarily glue due to \Cref{rigidity of the situation}.

Following a private communication with Illusie, let us define a transversal prism $(A, (a))$ together with a surjection
of prisms $(A, (a)) \twoheadrightarrow (B, (d))$ as follows.
Let 
\[
A \coloneqq \mathbb{Z}_p\{x_{b}; b \in B\}\{\delta(x_d)^{-1}\}^\wedge_{(x_d, p)}
\]
be given by first adjoining a free $\delta$-variable corresponding to each element in $B$ to $\mathbb{Z}_p$ 
together with an inverse of the $\delta$ of the variable corresponding to $d \in B$, 
completed in the end with respect to $(p, x_d)$.
Denote $a \coloneqq x_d$.
Then $(A, (a))$ is a transversal prism, and there is an evident surjection of prisms 
$(A, (a)) \twoheadrightarrow (B, (d))$.

Consider the surjection $A/a \langle X_1^{1/p^\infty}, \ldots, X_n^{1/p^\infty} \rangle \to
B/J \langle X_1^{1/p^\infty}, \ldots, X_n^{1/p^\infty} \rangle$
and lift the elements $f_i$ to $\tilde{f}_i$.
Let 
\[
\tilde{S} \coloneqq \mathrm{Kos}(A/a \langle X_1^{1/p^\infty}, \ldots, X_n^{1/p^\infty} \rangle; \tilde{f}_1,
\ldots, \tilde{f}_r).
\]
We have $S = \tilde{S} \otimes^{\BL}_{(A/a)} B/d$.
We know the analogous map for $\tilde{S}/(A/a)$ exists, thanks to \Cref{filtered map for transversal}.
Since both Nygaard and Hodge filtrations satisfy base change, we may base change the maps
for $\tilde{S}/(A/a)$ to obtain our desired maps for $S/(B/b)$.
\end{proof}

Following the same reasoning as in the proof of \Cref{filtered map for transversal},
one obtains the following:
\begin{proposition}
\label{filtered map for general}
For any $p$-completely smooth $B/J$-algebra $R$, 
the map $\Prism^{(1)}_{R/B} \to \dR_{R/B}^{\wedge}$
can be promoted to a map of filtered algebras.
Moreover this lift is functorial in the $B/J$-algebra $R$, hence left Kan extends
to all animated $B/J$-algebras.
\end{proposition}

\begin{proof}
The argument is exactly the same as \Cref{filtered map for transversal},
note that \Cref{rigidity of the situation} applies to general bounded base prism
$(B,J)$.
\end{proof}

\begin{remark}
Fix a bounded base prism $(B,J)$.
Any such functorial filtered map is determined by its effect on $p$-complete
polynomial algebras, by left Kan extension.
Then by quasisyntomic descent, such a functorial filtered map is determined
by its effect on a basis of $\qsyn_{B/J}$, such as the full subcategory
generated by large quasisyntomic over $B/J$ algebras.
Therefore \Cref{rigidity of the situation} implies that there is 
at most one such functorial filtered map.
Combining with the above proposition, we have both existence and uniqueness
of it.
\end{remark}

\begin{remark}
\label{weak compatibility with base change}
Following the way these filtered maps are constructed, we have certain
compatibility with base change:
Let $R$ be an animated $B/J$-algebra, let $(B,J) \to (C, JC)$
be a map of bounded prisms and denote $R' \coloneqq R \otimes^{\BL}_{B/J} C/JC$,
then the filtered map for $R'/C$ arises as the filtered map for $R/B$
base changed along $B \to C$.
Indeed it suffices to prove this when $R/(B/J)$ is $p$-completely smooth.
Then one simply notices that a surjection $B/J\langle X_1, \ldots, X_n\rangle \to R$
base change along $B \to C$ will give rise to a surjection
$C/JC\langle X_1, \ldots, X_n\rangle \to R'$.
\end{remark}

\subsection{Comparing Hodge and Nygaard filtrations}
We again use $(A,I)$ to denote a transversal prism, and use
$(B,J)$ to denote a general bounded prism.
All sheaves referred to in this subsection are viewed as objects
in $\mathrm{Shv}(\qsyn_{A/I})$ or $\mathrm{Shv}(\qsyn_{B/J})$
depending on the context.
Combining \Cref{comparing pris and crys} and \Cref{filtered map for general},
we get a natural map of sheaves of filtered rings:
\[
(\Prism^{(1)}_{-/B}, \Fil^{\bullet}_{\mathrm{N}}) \widehat{\otimes}_{(B, J^{\bullet})} 
(\dR_{(B/J)/B}^\wedge, \Fil^{\bullet}_{\mathrm{H}})
\longrightarrow (\dR_{-/A}^\wedge, \Fil^{\bullet}_{\mathrm{H}}),
\]
which is an isomorphism on the underlying sheaf of rings.
Our objective in this subsection is to show that the above map is an isomorphism of sheaves
of filtered rings.


Our plan is again to first understand the case of transversal base prism,
then bootstrap to general base prisms.
Let us begin by discussing the case of transversal base prism.
\begin{theorem}
\label{lqsyn H and N filtration}
Let $S$ be a large quasisyntomic over $A/I$ algebra.
\begin{enumerate}
\item The map $\Prism^{(1)}_{S/A} \to \dR_{S/A}^{\wedge}$ is injective.
\item We have 
\[
\Fil^r_I \Prism^{(1)}_{S/A} = \left( \Fil^r_{\mathcal{I}} \dR_{S/A}^{\wedge} \right)
\cap \left( \Prism^{(1)}_{S/A} \right).
\]
\item We have 
\[
\Fil^i_{\mathrm{N}} \Prism^{(1)}_{S/A} = \left( \Fil^i_{\mathrm{H}} \dR_{S/A}^{\wedge} \right)
\cap \left( \Prism^{(1)}_{S/A} \right).
\]
\item For any $i$, the natural map 
$\Prism^{(1)}_{S/A}/\Fil^i_{\mathrm{N}} \to \dR_{S/A}^{\wedge}/\Fil^i_{\mathrm{H}}$
is an injection of $p$-torsionfree modules,
whose cokernel is $(i-1)!$-torsion.
Hence multiplying by $(i-1)!$ gives a natural map backward and compose
the two maps in either direction is the same as multiplying by $(i-1)!$.
In particular, the natural map
\[
\Prism^{(1)}_{S/A}/\Fil^i_{\mathrm{N}} \to \dR_{S/A}^{\wedge}/\Fil^i_{\mathrm{H}}
\]
is an isomorphism for any $i \leq p$.
\item The induced map 
$\gr^*_{\mathrm{N}}\Prism^{(1)}_{S/A} \to \gr^*_{\mathrm{H}}\dR_{S/A}^{\wedge}$
is compatible with the vertical filtrations on both sides, and the induced map
on the graded pieces of the vertical filtrations
$(\wedge^i_S \LL_{S/(A/I)})^{\wedge}[-i] \widehat{\otimes}_{A/I} \Sym^*_{A/I}(I/I^2)
\to (\wedge^i_S \LL_{S/(A/I)})^{\wedge}[-i] \widehat{\otimes}_{A/I} \Gamma^*_{A/I}(I/I^2)$
is given by $\id \otimes_{A/I} \left(\gr^*_I A \to \gr^*_{\mathcal{I}} \mathcal{A} \right)$.
\end{enumerate}
\end{theorem}

Contemplating with $R = A/I$ suggests that our estimate in (4) is sharp.

Before the proof, 
let us remark that $p$-completely tensor over $A$ with an $\mathcal{A}$-module is the same
as $(p,I)$-completely tensor.
This is because $I^p \mathcal{A} \subset p \mathcal{A}$.

\begin{proof}
(1): the map is given by $(p,I)$-completely tensoring the inclusion $A \hookrightarrow \mathcal{A}$ 
with $\Prism^{(1)}_{S/A}$ over $A$.
Since $\Prism^{(1)}_{S/A}$ is $(p,I)$-completely flat over $A$, see \Cref{flatness of dR}, we get the injectivity
of $\Prism^{(1)}_{S/A} \to \dR_{S/A}^{\wedge}$.

(2): clearly we have $I^r \Prism^{(1)}_{S/A}$ contained in 
$\mathcal{I}^{[r]} \dR_{S/A}^{\wedge}$. To check the equality of intersection,
it suffices to show the induced map
$\Prism^{(1)}_{S/A}/I^r \to \dR_{S/A}^{\wedge}/\mathcal{I}^{[r]}$ is injective.
But this map is given by $(p,I)$-completely tensoring $\Prism^{(1)}_{S/A}$ with
the inclusion $A/I^r \hookrightarrow \mathcal{A}/\mathcal{I}^{[r]}$ over $A$,
so we get the desired injectivity again by $(p,I)$-completely flatness of
$\Prism^{(1)}_{S/A}$ over $A$.

(3): it suffices to show that the induced map
$\tilde{g} \colon \Prism^{(1)}_{S/A}/\Fil^i_{\mathrm{N}} \to \dR_{S/A}^{\wedge}/\Fil^i_{\mathrm{H}}$
is injective.
The $I$-adic and $\mathcal{I}^{[\bullet]}$-filtrations on each side induces
maps of graded pieces as
\[
\dR_{S/(A/I)}^{\wedge}/\Fil^j_{\mathrm{H}} \widehat{\otimes}_{A/I} I^{i-j}/I^{i-j+1} \to
\dR_{S/(A/I)}^{\wedge}/\Fil^j_{\mathrm{H}} \widehat{\otimes}_{A/I} \mathcal{I}^{[i-j]}/\mathcal{I}^{[i-j+1]}.
\]
Here we have used \Cref{general KO filtration properties} (3) and \Cref{general KON filtration properties} (3).
We conclude that the map $\tilde{g}$ is injective as $\dR_{S/(A/I)}^{\wedge}/\Fil^j_{\mathrm{H}}$
is $p$-completely flat over $A/I$ for any $j$ and the natural map
$I^{i-j}/I^{i-j+1} \to \mathcal{I}^{[i-j]}/\mathcal{I}^{[i-j+1]}$ is injective.

(4): injectivity follows from the previous paragraph. 
Let $S = A/I \langle X_l^{1/p^\infty} \mid l \in L \rangle /M $, 
with each element $m \in M$ corresponding  to a series $f_m$.
Below we shall not distinguish $m$ and $f_m$.
Consider 
\[
S' = A/I \langle X_l^{1/p^\infty}, Y_m^{1/p^\infty} \mid l \in L, m \in M \rangle /(Y_m - f_m; m \in M)
\eqqcolon \widetilde{S}/(Y_m - f_m; m \in M).
\]
There is a surjection $S' \twoheadrightarrow S$ of $A/I$-algebras, sending powers of $Y_m$ to $0$.
This induces a surjection on $\mathbb{L}_{-/A}^\wedge$, hence also a surjection on $\dR_{-/A}^\wedge$.
Therefore it suffices to prove the statement for $S'$.

Now we know $\dR_{S'/A}^\wedge$ is given by $p$-completely adjoining divided powers of
$I$ and $Y_m - f_m$ to $\widetilde{S}$,
and the $i$-th Hodge filtration is given by the ideal $p$-completely generated by those
degree-at-least-$i$ divided monomials.
Since the image of $\Prism^{(1)}_{S'/A}$ contains 
$\widetilde{S}$ already,
it suffices to show that $(i-1)!$ times those degree-less-than-$i$ divided monomials
lies in $\widetilde{S}$, which follows from definition.

(5): since the generating factor $(\wedge^i_S \LL_{S/(A/I)})^{\wedge}[-i]$
of both vertical filtrations comes from the $i$-th graded piece of the Hodge filtration
on $\dR_{S/(A/I)}^{\wedge}$ (via modulo $I$ and $\mathcal{I}$ respectively),
our statement follows from the commutative triangle in~\Cref{compatibility when modulo I}.
\end{proof}

The above statements can be immediately extended to our desired statement
via several reduction steps.
\begin{corollary}
\label{affine H and N filtration}
Let $R$ be a $B/J$-algebra.
The natural map of filtered algebras:
\[
(\Prism^{(1)}_{R/B}, \Fil^{\bullet}_{\mathrm{N}}) \widehat{\otimes}_{(B, J^{\bullet})} 
(\dR_{(B/J)/B}^\wedge, \Fil^{\bullet}_{\mathrm{H}})
\longrightarrow (\dR_{R/A}^\wedge, \Fil^{\bullet}_{\mathrm{H}}),
\]
is a filtered isomorphism.
In particular, filtrations on the left hand side define quasisyntomic sheaves.
\end{corollary}

We refer readers to \cite[3.8-3.10]{GL20} for a discussion of the filtration on tensor
of filtered modules over a filtered algebra. Here we use $\widehat{\otimes}$ to mean
that we derived $p$-complete the \cite[Construction 3.9]{GL20}.

\begin{proof}
We make a few reduction steps.
First of all, both sides are left Kan extended from the case of $p$-complete
polynomial algebras, therefore it suffices to show the map is a filtered isomorphism
when $R = B/J \langle X_1, \ldots, X_n\rangle$.
Secondly, it suffices to prove the statement Zariski locally on $\Spf(B/J)$,
hence we may assume $(B,J)$ is oriented.
Now we look at the universal map from universal oriented prism $(A^{\mathrm{univ}}, I) \to (B,J)$.
Let $R^{\mathrm{univ}}$ be the corresponding $p$-complete polynomial algebras over 
the reduction of the universal oriented prism.
By \Cref{weak compatibility with base change}, one sees that the filtered map
for $R$ is the base change of the analogous map for $R^{\mathrm{uinv}}$.
Therefore we are finally reduced to the case where the base prism $(A,I)$ is transversal
and $R$ is a $p$-completely smooth $A/I$-algebra.

Since the underlying algebra is an isomorphism by \Cref{comparing pris and crys}, it suffices
to show the induced map of graded algebra is an isomorphism.
By derived $p$-completing \cite[Lemma 3.10]{GL20}, we see that the graded algebra of left hand side
becomes 
\[
\gr^*_{\mathrm{N}}(\Prism^{(1)}_{R/A}) \widehat{\otimes}_{\Sym^*_{A/I}(I/I^2)} \Gamma^*_{A/I}(I/I^2).
\]
Now we invoke the vertical filtrations on graded algebras of both sides, see \Cref{increasing filtration on Hodge graded}
and \Cref{increasing filtration on Nygaard graded}.
The vertical filtration on $\gr^*_{\mathrm{N}}(\Prism^{(1)}_{R/A})$ induces an increasing filtration
by $(-) \widehat{\otimes}_{\Sym^*_{A/I}(I/I^2)} \Gamma^*_{A/I}(I/I^2)$,
and our morphism induces identifications
\[
\gr^v_i\bigg(\gr^*_{\mathrm{N}}(\Prism^{(1)}_{R/A}) \widehat{\otimes}_{\Sym^*_{A/I}(I/I^2)} \Gamma^*_{A/I}(I/I^2)\bigg)
\cong \gr^v_i\bigg(\gr^*_{\mathrm{H}}(\dR_{R/A}^\wedge)\bigg)
\]
for all $i$.
Here we have used \Cref{lqsyn H and N filtration} (5).
Since these vertical filtrations are increasing, exhaustive, and uniformly bounded below by $0$,
we conclude that the natural map
\[
\gr^*_{\mathrm{N}}(\Prism^{(1)}_{R/A}) \widehat{\otimes}_{\Sym^*_{A/I}(I/I^2)} \Gamma^*_{A/I}(I/I^2) \longrightarrow
\gr^*_{\mathrm{H}}(\dR_{R/A}^\wedge)
\]
is also an isomorphism.
\end{proof}

In particular, we can specialize to the case of quasi-compact quasi-separated smooth formal schemes over $\Spf(B/J)$.

\begin{corollary}[{c.f.~\cite[Theorem 2.9]{Ill20}}]
\label{smooth H and N filtration}
Let $X$ be a quasi-compact quasi-separated smooth formal scheme over $\Spf(B/J)$.
Then we have a natural filtered isomorphism:
\[
\bigg(\mathrm{R\Gamma}(X, \Fil^{\bullet}_{\mathrm{N}}(\Prism^{(1)}_{-/B}))\bigg) 
\widehat{\otimes}_{(B, J^{\bullet})} 
(\dR_{(B/J)/B}^\wedge, \Fil^{\bullet}_{\mathrm{H}})
\xrightarrow{\cong} \mathrm{R\Gamma}(X, \Fil^{\bullet}_{\mathrm{H}}(\dR_{-/B}^\wedge));
\]
they are furthermore naturally filtered isomorphic to
$\mathrm{R\Gamma_{crys}}(X, \mathcal{I}^{[\bullet]}_{\mathrm{crys}})$
if $(B,J)$ is transversal.
Similarly whenever $i \leq p$, we also have natural isomorphisms
\[
\mathrm{R\Gamma}(X, \Prism^{(1)}_{-/A}/\Fil^i_{\mathrm{N}}) \xrightarrow{\cong} 
\mathrm{R\Gamma}(X, \dR_{-/A}^\wedge/\Fil^i_{\mathrm{H}});
\]
and are furthermore naturally isomorphic to
$\mathrm{R\Gamma_{crys}}(X, \mathcal{O}_{\mathrm{crys}}/\mathcal{I}^{i}_{\mathrm{crys}})$
if $(B,J)$ is transversal.
These isomorphisms are functorial in $X$, and satisfies the base change property
as in \Cref{weak compatibility with base change}.
\end{corollary}

\begin{proof}
These functorial isomorphisms are provided by \Cref{affine H and N filtration}
and \Cref{lqsyn H and N filtration} (4) respectively.
The furthermore equality, when the base prism is transversal, follows from \Cref{Illusie--Bhatt}.
\end{proof}

\begin{remark}
\label{rem-N and H fil}
Back to the transversal base prism case.
A posteriori the filtration on the left hand side of \Cref{affine H and N filtration} is a quasisyntomic sheaf,
hence we may define it as the unfolding of its restriction to the basis of large quasisyntomic over $A/I$-algebras.
Also a posteriori, we know the value on such an algebra $S$ must be concentrated in cohomological degree $0$,
therefore they have to be the image of the augmentation map
\[
\Fil^i(\Prism^{(1)}_{S/A} \widehat{\otimes}_{\mathbb{Z}_p} \mathcal{A}) \to \dR_{S/A}^\wedge,
\]
where the filtration on the left hand side is given by the usual Day convolution.
This implies an equality 
\[
\Fil^n_{\mathrm{H}}(\dR_{S/A}^\wedge) = \sum_{i = 0}^n 
\bigg(\Fil^i_{\mathrm{N}}(\Prism^{(1)}_{S/A}) \cdot \mathcal{I}^{[n-i]}    \bigg),
\]
which also follows from combining \Cref{general KO filtration properties} (1), \Cref{qsyn KO filtration properties} (3),
and \Cref{Nygaard filtration} (2).

Therefore, for any $0 \leq r \leq p-1$, we see that the Frobenius on derived de Rham complex when restricted
to the $r$-th Hodge filtration
\[
\Fil^r_{\mathrm{H}}(\dR_{S/A}^\wedge) \xrightarrow{\varphi} \dR_{S/A}^\wedge
\]
factors through multiplication by $p^r$.
Since for large quasisyntomic over $A/I$-algebras $S$, the $\dR_{S/A}^\wedge$
is $p$-completely flat over $\mathcal{A}$ (see \Cref{flatness of dR}) which is $p$-torsionfree,
we may uniquely divide the restriction $\varphi$ by $p^r$.
By unfolding, this gives rise to divided Frobenii as maps of sheaves on $\mathrm{qSyn}_{A/I}$:
\[
\varphi_r \colon \Fil^r_{\mathrm{H}}(\dR_{-/A}^\wedge) \rightarrow \dR_{-/A}^\wedge.
\]
By definition, they also satisfy $\varphi_r \mid_{\Fil^{r+1}_{\mathrm{H}}} = p \varphi_{r+1}$ when $r \leq p-2$.
Following the same argument of \Cref{functorial endomorphism theorem},
see also \Cref{functorial endo remark}, such a functorial
divided Frobenius is unique for each $0 \leq r \leq p-1$.

When $(A,I)$ is the Breuil--Kisin prism, this gives rise to an alternative definition of the divided Frobenii appeared in
\cite[p.~10]{Bre98}. 
\end{remark}

\section{Connection on $\dR_{-/\mathfrak{S}}^\wedge$ and structure of torsion crystalline cohomology}
From this section onward, we focus on the Breuil--Kisin prism $A=(\mathfrak{S}, E)$  and crystalline cohomology over $S= \dR^\wedge_{\O_K/\s}$. 
Let $k$ be a prefect field with characteristic $p$,
and let $K$ be a finite totally ramified extension over $K_0 = W(k)[\frac 1 p]$ with a fixed uniformizer $\pi\in \O_K$. Fix an algebraic closure $\Kbar $ of $K$ and let $\C$ be $p$-adic completion of $\Kbar$. Write $G_K : = \Gal (\Kbar / K)$ and $e=[K:K_0]$.  
Let $E= E(u)\in W(k)[u]$ be the Eisenstein polynomial of $\pi$ with constant term $a_0 p$, recall $\s : = W(k)[\![u]\!]$ is equipped with a Frobenius 
$\varphi$ naturally extends that on $W(k)$ by $\varphi (u) = u ^p$. 
Pick $\pi_n \in \O_{\Kbar }$ so that $\pi_{0} = \pi$ and $\pi_{n +1}^p = p$. 
Then $\upi: = (\pi_n)_{n \geq 0}\in \O_{\C}^\flat$. 
We embed $\s \inj A_{\inf}$ via $u \mapsto [\upi]$ which is a map of prisms. 
Let $K_\infty : = \bigcup\limits_{n = 0}^\infty K(\pi_n)$ and $G_\infty : = \Gal (\Kbar / K_\infty)$. 
It is clear that the embedding $\s \subset A_{\inf}$ is compatible with $G_\infty$-actions. 
We extend $\varphi$ from $\s$ to $S$ and let $\Fil^m S$ be the $p$-adic closure of the ideal generated by  $\gamma_i (E): = \frac{E^i}{i!}, i \geq m$. 
We embed $S \inj A_{\cris}$ also via $u \mapsto [\upi]$. 
For $m \leq p-1$,  $\varphi(\Fil^m S)\subset p ^m S$. 
We set $\varphi_m : = \frac{\varphi}{p ^m} :\Fil ^m S \to S$. 
Similar notation also applies to $A_{\cris}$. Write $c_1: = \frac{\varphi (E)}{a_0 p} \in S^\times$. 
Finally, there exists a $W(k)$-linear derivation   $\nabla_S : S \to S$ by $\nabla_S(f(u))= f'(u)$.    

 For $n \geq 1$, if $M$ is an $\Z_p$-module then we always use $M_n$ to denote $M / p ^n M$. 
 Similar notation applies to ($p$-adic formal) schemes: i.e., $X_n : = X \times_{\Spf(\Z_p)} \Spec (\Z/ p ^n \Z)$. Write $W= W(k)$ and reserve $\gamma_i (\cdot)$ for the $i$-th divided power. 

\subsection{Connection on $\dR_{-/\mathfrak{S}}^\wedge$}
Under the philosophy that derived de Rham cohomology behaves a lot like crystalline cohomology,
one expects there to be a connection on $\dR_{-/\mathfrak{S}}^\wedge$.
We explain it in this section.


\begin{lemma}
\label{decompleting the base}
Let $R$ be an $\mathcal{O}_K$-algebra. Then the natural morphism
$\dR_{R/W[u]}^\wedge \to \dR_{R/\mathfrak{S}}^\wedge$ is an isomorphism,
where $R$ is regarded as an $\mathfrak{S}$ and $W[u]$ algebra via
$W[u] \to \mathfrak{S} \to \mathcal{O}_K \to R$.
\end{lemma}

\begin{proof}
Just notice the following $p$-completely pushout diagram:
\[
\xymatrix{
\mathfrak{S} \ar[r] & \mathcal{O}_K \ar[r] & R \\
W[u] \ar[r] \ar[u] & \mathcal{O}_K \ar[u] \ar[r] & R \ar[u],
}
\]
and appeal to the $p$-completely base change formula of derived de Rham complexes
to get 
\[
\dR_{R/W[u]}^\wedge \widehat{\otimes}_{W[u]} \mathfrak{S} \xrightarrow{\cong} \dR_{R/\mathfrak{S}}^\wedge.
\]
Next we observe that $\dR_{R/W[u]}^\wedge$ is an $S = \dR_{\mathcal{O}_K/W[u]}^\wedge$-complex
and $S \widehat{\otimes}_{W[u]} \mathfrak{S} = S$, hence the base change on the left hand side
gives $\dR_{R/W[u]}^\wedge$ back.
\end{proof}

\begin{construction}[{see also \cite{KO68}}]
\label{connection construction}
For any $W[u]$-algebra $R$, 
by ($p$-completely) applying \cite[Lemma 3.13.(4)]{GL20} to the triple $W \to W[u] \to R$, 
we see that there is a functorial triangle in filtered derived $\infty$-category of $W$-modules:
\[
\dR_{R/W[u]}^\wedge \widehat{\otimes}_{W[u]} \Omega^1_{W[u]/W}[-1] \to \dR_{R/W}^\wedge
\to \dR_{R/W[u]}^\wedge.
\]
Here $\Omega^1_{W[u]/W}[-1]$ is completely put in the first filtration.
By choosing the generator $du \in \Omega^1_{W[u]/W}$, the above becomes
\[
\dR_{R/W}^\wedge \to \dR_{R/W[u]}^\wedge \xrightarrow{\nabla} \dR_{R/W[u]}^\wedge(-1),
\]
where $(-1)$ indicates the shift of filtrations: 
$\Fil^i(\dR_{R/W[u]}^\wedge(-1)) = \Fil^{i-1}_{\mathrm{H}}(\dR_{R/W[u]}^\wedge)$.
When $R$ is smooth over $W[u]$, then $\nabla$ is given by Lie derivative
with respect to $\frac{\partial}{\partial u}$:
\[
\nabla(\omega) = \mathcal{L}_{\frac{\partial}{\partial u}}(\omega).
\]
\end{construction}

\begin{lemma}
\label{connection on ddR}
Let $R$ be an $\mathcal{O}_K$-algebra. Then we have a functorial 
triangle in the filtered derived $\infty$-category:
\[
\dR_{R/W}^\wedge \to \dR_{R/\mathfrak{S}}^\wedge \xrightarrow{\nabla} \dR_{R/\mathfrak{S}}^\wedge(-1).
\]
Moreover we have
\[
p u^{p-1} \cdot \varphi \circ \nabla = \nabla \circ \varphi,
\]
where $\varphi \colon \dR_{R/\mathfrak{S}}^\wedge \to \dR_{R/\mathfrak{S}}^\wedge$ is the Frobenius
defined in \Cref{Frobenii}.
\end{lemma}

\begin{proof}
The first statement follows from \Cref{connection construction} and \Cref{decompleting the base}.

To check the equality, by left Kan extension it suffices to check
it for the polynomials.
Then by quasisyntomic descent, it suffices to check the equality for
large quasisyntomic over $\mathcal{O}_K$ algebras.
Following the proof of \Cref{uniqueness of comparison},
we are reduced to showing the equality for algebras of the form
\[
R = \mathcal{O}_K \langle X_i^{1/p^\infty}, Y_j^{1/p^\infty} \mid i \in I, j \in J \rangle/(Y_j - f_j \mid j \in J) \eqqcolon \widetilde{R}/(Y_j - f_j \mid j \in J).
\]
Now the map $\widetilde{R} \to R$ induces a map between $\dR_{-/\mathfrak{S}}^\wedge$ given by
\[
S \langle X_i^{1/p^\infty}, Y_j^{1/p^\infty} \mid i \in I, j \in J \rangle  \eqqcolon T
\to D_T(Y_j - f_j; j \in J)^\wedge.
\]
Here $S$ is the $p$-complete PD envelope of $\mathfrak{S}$ along $(E)$ and the latter denotes $p$-completely adjoining divided powers of $(Y_j - f_j)$ in $T$.
Since $D_T(Y_j - f_j; j \in J)^\wedge$ is $p$-complete and $p$-torsionfree, it suffices
to check the identity on $T$.
On $T$, the Frobenius $\varphi$ acts by sending variables $X, Y, u$ to their $p$-th power,
and $\nabla$ acts via $\frac{\partial}{\partial u}$.
Finally we are reduced to checking the equality
\[
p u^{p-1} \cdot \varphi\left(\frac{\partial}{\partial u} (F(u, \underline{X}, \underline{Y}))\right) = \frac{\partial}{\partial u} \left(\varphi( F(u, \underline{X}, \underline{Y})) \right),
\]
for any $F(u, \underline{X}, \underline{Y}) \in T$.
\end{proof}

Consequently, for any $\mathcal{O}_K$-algebra $R$, we always have a long exact sequence:
\[
\label{LES of connection}
\tag{\epsdice{4}}
\ldots \to \rH^i(\dR_{R/W}^\wedge) \to \rH^i(\dR_{R/\mathfrak{S}}^\wedge) \xrightarrow{\nabla} \rH^i(\dR_{R/\mathfrak{S}}^\wedge(-1))
\xrightarrow{+1} \ldots
\]
and its $r$-th filtration analogues for all $r \in \mathbb{N}$.
In special situation, these will break into short exact sequences.
Let us introduce some more notation.
Let $L$ be a perfectoid field extension of $K$ containing all $p$-power roots of $\pi$.
For instance $L$ could be $p$-adic completion of $K_{\infty}$ or $\C$.
Let $A_{\inf}(L) \coloneqq W(\mathcal{O}_L^\flat)$ be Fontaine's $A_{\inf}$ ring associated with
$L$, and recall there is a natural map $\theta \coloneqq A_{\inf}(L) \to \mathcal{O}_L$.
Fix a compatible system of $p$-power roots of $\pi$, we obtain a map $\mathfrak{S} \to A_{\inf}(L)$
with $u \mapsto [\underline{\pi}]$ compatible with $\theta$ and the inclusion $\mathcal{O}_K \to \mathcal{O}_L$.

\begin{proposition}
\label{prop-connection for perfectoid}
With notation as above. Let $R$ be a quasisyntomic $\mathcal{O}_L$-algebra.
Then we have
\begin{enumerate}
\item The natural map $\dR_{R/W}^\wedge \to \dR_{R/A_{\inf}(L)}^\wedge$ is a filtered isomorphism.
\item The sequence \epsdice{4} and its $r$-th filtration analogues break into short exact sequence:
\[
0 \to \rH^i(\Fil^r_{\mathrm{H}}\dR_{R/W}^\wedge) \to \rH^i(\Fil^r_{\mathrm{H}}\dR_{R/\mathfrak{S}}^\wedge)
\xrightarrow{\nabla} \rH^i(\Fil^{r-1}_{\mathrm{H}}\dR_{R/\mathfrak{S}}^\wedge) \to 0,
\]
for all $i$ and $r$.
In particular $\dR_{R/\mathfrak{S}}^\wedge \xrightarrow{\nabla} \dR_{R/\mathfrak{S}}^\wedge(-1)$
is surjective on each $\rH^i$, and 
\[
\rH^i(\dR_{R/W}^\wedge) = \rH^i(\dR_{R/\mathfrak{S}}^\wedge)^{\nabla = 0}.
\]
\end{enumerate}
\end{proposition}

\begin{proof}
(1) is \cite[Theorem 3.4.(2)]{GL20}.

As for (2), it suffices to show that the maps 
$\rH^i(\Fil^r_{\mathrm{H}}\dR_{R/W}^\wedge) \to \rH^i(\Fil^r_{\mathrm{H}}\dR_{R/\mathfrak{S}}^\wedge)$
are injective for all $i$ and $r$.
By functoriality, we have maps of filtered algebras
\[
\dR_{R/W}^\wedge \to \dR_{R/\mathfrak{S}}^\wedge \to \dR_{R/A_{\inf}(L)}^\wedge
\]
whose composition is a filtered isomorphism by (1).
Therefore the first morphism factorizing isomorphism induces injection at the level of cohomology.
This explains why the long exact sequence \epsdice{4} breaks into short exact sequences.
The last statement follows easily by letting $r = 0$.
\end{proof}
\subsection{Structures of torsion crystalline cohomology}
\label{subsection-stru of crys} 
Let $X$ be a proper smooth formal scheme over $\O_K$. 
Let us summarize the structures on $\rH ^i _{\cris} (X/S) \coloneqq \rH ^i_{\cris} ( X/S , \O_{\cris} )$ constructed from previous sections. 

By Corollary \ref{affine H and N filtration} and Theorem \ref{Illusie--Bhatt}, we obtain the following commutative diagram. 
\begin{equation}\label{dig-summary of comp}
\begin{split}
\xymatrix{ \RG _{\qsyn} (X, \Prism^{(1)}_{-/\s})\ar[r] & \RG_{\cris} (X/S, \O_{\cris}) \ar@{= }[r]& S \otimes^\BL_{\varphi, \s}\RG_\Prism ( X /\s) \\ \RG _{\qsyn} (X, \Fil^m_{\rm N} \Prism ^{1}_{-/ \s})\ar[u] \ar[r] & \RG _{\cris} (X/S, \mathcal I^{[m]}_{\cris})\ar[u] & }
\end{split}
\end{equation}
Here the second isomorphism of the first row follows the canonical isomorphism $\RG_{\qsyn} ( X, \Prism_{-/ \s}) \simeq \RG_\Prism (X/\s ) $ and the fact that $\varphi : \s \to \s$ is flat. 

For $m \leq p-1$, Remark \ref{rem-N and H fil} allows to us to define $\varphi$-semi-linear map $\varphi_m : \rH^ i _{\cris} (X/S , \cI^{[m]}_{\cris}) \to \rH ^i_{\cris} (X/S)$ so that the following diagram commutes for $m +1 \leq p-1$
$$\xymatrix{\rH^ i _{\cris} (X/S , \cI^{[m]}_{\cris}) \ar[r]^-{\varphi_{m}} & \rH ^i_{\cris} (X/S) \\ \rH^ i _{\cris} (X/S , \cI^{[m+1]}_{\cris})\ar[u]\ar[ur]^{p\varphi_{m +1}} &  }$$
 We simply denote the above diagram by $\varphi_{m}|_{\rH^ i _{\cris} (X/S , \cI^{[m+1]}_{\cris})} = p \varphi_{m +1}$. It is also clear that for any $s \in \Fil^m S$ and $x\in \rH^ i _{\cris} (X/S ) $
	we have \[\varphi_m (sx)= (c_1)^{-m }\varphi_m (s) \varphi_h(E(u)^m x).\]
 Finally, the above subsection construct a connection $\nabla:   \rH ^i_{\cris} (X/S) \to \rH ^i_{\cris} (X/S) $. 
 By Proposition \ref{prop-connection for perfectoid} and Lemma \ref{connection on ddR}, we conclude that 
 \begin{enumerate}
     \item $\nabla: \rH ^i_{\cris} (X/S) \to \rH ^i_{\cris} (X/S)$ is $W(k)$-linear derivative satisfying   \[\nabla (f(u)x) = f'(u) x + f(u) {\nabla} (x)  \]
     \item (Griffiths Transversality)  $\nabla (\rH ^i _{\cris} (X/ S, \cI^{[m]}_{\cris}))$ factors through $\rH ^i _{\cris} (X/ S, \cI^{[m-1]}_{\cris})$. 
     
     \item The following diagram commutes: $$
	\xymatrix@C=3em{ \rH ^i _{\cris} (X/S , \cI_{\cris} ^{[m]})\ar[d]_{E(u) \nabla} \ar[r]^-{\varphi_m } & \rH ^i _{\cris} (X/S)  \ar[d]^{c_1 \nabla}\\
	\rH ^i _{\cris} (X/S , \cI_{\cris} ^{[m]})   \ar[r]^-{ u ^{p-1} \varphi_m } &\rH ^i _{\cris} (X/S)  }
$$
 \end{enumerate}
 The last diagram follows that $pu^{p-1} \varphi \circ \nabla= \nabla \circ \varphi$ by Lemma \ref{connection on ddR} and that $\varphi (E) = p c_1$.

Now consider the $p^n$ -torsion crystalline cohomology $\rH ^i_{\cris} (X_n /S_ n)$ together with filtration $\rH ^i _{\cris} (X_n /S_n, \mathcal I ^{[m]}_{\cris})$. We claim that $\rH ^i_{\cris} (X_n /S_n)$ admits all the above structures $\varphi_m: \rH^i _{\cris} (X_n/ S_n, \cI ^{[m]}_{\cris}) \to \rH ^i_{\cris} (X_n /S_n) $ for $m \leq p-1$ and $\nabla : \rH ^i_{\cris} (X_n /S_n) \to \rH ^i_{\cris} (X_n /S_n) $ satisfying all the above properties. 
To see, note that $\RG_{\cris} ( X_n /S_n, \cI ^{[m]}_{\cris})\simeq \RG _{\cris} (X/S, \cI ^ {[m]}_{\cris}) \otimes^\BL_{\Z} \Z/ p ^n\Z$ 
where $\cI ^{[0]}_{\cris}= \O_{\cris}$ then all the above properties follow by taking $\otimes^{\BL}_\Z \Z/ p ^n \Z$, except the Diagram \eqref{dig-summary of comp} which requires torsion quasi-syntomic cohomology. For this, we define the following torsion cohomologies: 
For $m \geq 0$,  $\RG _{\dR} (X_n/\s_n, \Fil^m_{\rm H} ): = \RG_{\dR} (X/\s , \Fil^m _{\mathrm{H}}) \otimes ^\BL_{\Z} \Z/ p ^n \Z $,  $ \RG_{\qsyn} (X_n / \s_n, \Fil^m_N \Prism^{(1)}):  = \RG_{\qsyn}(X/ \s , \Fil^m_N\Prism^{(1)}_{-/ \s}) \otimes_\Z \Z/ p ^n \Z $, and finally  $\RG_{\Prism} (X_n/\s_n): = \RG _\Prism( X/ \s) \otimes_{\Z} \Z/ p ^n \Z$. 
Then the derived modulo $p^n$ version of Diagram \eqref{dig-summary of comp} still holds by taking the original diagram and derived modulo $p^n$.

\subsection{Galois action on torsion crystalline cohomology}\label{subsection-GK-action} Keep the notations as the above. 
Set $\mathcal X$ to be the base change of $X$ to $\Spf \O_{\C}$ and  $\mathcal X _n : = \mathcal{X} \otimes_{\Z} {\Z/ p ^n\Z}$.  Then $\rH ^i_{\cris} ({\mathcal X}_n /S_n)$ has an $S$-linear $G_K$-action when we define the $G_K$-action on $S$ is trivial. Note that $\rH ^i_{\cris} ({\mathcal X}_n /A_{\cris, n})$ also has $A_{\cris}$-semi-linear $G_K$-action which is induced by $G_K$-actions on ${\mathcal X}$ and $A_{\cris}$. By Proposition \ref{prop-connection for perfectoid} and its proof, we see that the   natural map 
$W(k) \to \s \to A_{\inf}$ induces the following commutative diagram 
$$\xymatrix{ \rH ^i_{\cris} ({\mathcal X}_n / W_n (k)) \ar[rd]_\sim^\beta \ar@{^{(}->}[r]^\alpha & \rH ^i_{\cris} ({\mathcal X}_n / S_n) \ar[d]^\iota & \rH ^i _{\cris} (X_n/ S_n)\ar@{_{(}->}[l]_{\tilde \alpha} \ar@{^{(}->}[d]^{\tilde \iota}\\ & \rH ^i _{\cris} ({\mathcal X}_n/ A_{\cris, n})\ar@{=}[r] &   \rH^i_{\cris} (X_n / S_n) \otimes_{S_n} A_{\cris, n} . }$$
Note that the second row is an isomorphism because $ {\mathcal X} = X \times_{\Spec(S)} \Spec (A_{\cris})$ and that $A_{\cris, n}$ is flat over $S_n$. Thus $\tilde \iota$ is an injection. So is $\tilde \alpha$. Also we note that $\alpha$ and $\beta$ are both compatible with $G_K$-actions because both the map $ W(k) \to \s$ and $W(k)\to A_{\inf}$ are $G_K$-compatible. But $\iota$ is not as $\s \subset A_{\inf}$ is only stable under $G_\infty$-action. It is also clear that $\rH ^i _{\cris} (X_n /S_n) \subset ( \rH ^i _{\cris} ({\mathcal X}_n /S_n))^{G_K}$ via $\tilde \alpha$ and $\tilde \alpha$ is also compatible with connection on both sides. Now we claim the $G_K$-action on $\rH ^i _{\cris} ( {\mathcal X}_n / A_{\cris, n})$ is given by the following formula: For any $\sigma \in
G_K$, 
any  $x \otimes a \in \rH ^i _{\cris} (X_n /S_n)\otimes_S A_{\cris}\simeq \rH ^i_{\cris} ({\mathcal X}_n / A_{\cris, n})$,
\begin{equation}\label{eqn-action-3}
\sigma(x \otimes a)= \sum_{i=0}^\infty \nabla^i (x) \otimes  \gamma_i
\left ( \sigma ([\upi])- [\upi] \right ) \sigma (a) .
\end{equation}
To see this, for any $x\in \calM ^i : = \rH ^i _{\cris} (X_n /S_n)$, set 
$$ x^\nabla : = \sum_{m  = 0} ^\infty  \nabla (x) \gamma_m  ([\upi] -u) \in \rH^i  _{\cris} ({\mathcal X}_n /S_n). $$
Then we immediately see that $x^\nabla \in \rH ^ i _{\cris} ({\mathcal X}_n/ W_n (k)) =  \rH ^ i _{\cris} ({\mathcal X}_n/ S_n) ^{\nabla= 0}$. Now we claim that $\rH ^i _{\cris} ({\mathcal X}_n/W_n (k))$ is generated by $\{x ^\nabla | x \in \rH ^i _{\cris} (X_n /S_n)\}$ as an $A_{\cris}$-module. If so then \eqref{eqn-action-3} follows the fact that $\beta$ is $G_K$-equivariant and the construction of $x^\nabla$ (note that both $x$ and $u$ are $G_K$-invariants).   

To prove the claim, for any $y \in \rH ^i _{\cris} ({\mathcal X}_n /W_n (k))$, suppose that $\beta (y) = \sum_j a_j \tilde\iota (x_j)$ with $a_j \in A_{\cris} $ and  $x_j \in \rH^i _{\cris} (X_n /S_n)$. Then we see that $y ^\nabla : = \sum_j a_j x_j ^\nabla  \in \rH ^i _{\cris} ({\mathcal X}_n /W_n (k))$. It suffices to that check $y = y ^\nabla$. Since $\beta$ is an isomorphism, it suffices to show that $\beta (y) = \beta (y ^\nabla)$. This follows that $\beta(x^\nabla) =  \iota (x)$ for $x \in \rH ^i _{\cris} (X_n /S_n)$ as $ \iota ([\upi]- u)= [\upi]- [\upi]= 0$.

\section{Torsion Kisin module, Breuil module and associated Galois representations} In this section, we set up the theory of generalized torsion Kisin modules which extends theory of Kisin modules, which is discussed,  for example,  \cite[\S2]{liu-Fontaine}. The key point for the generalized Kisin modules is that it may have $u$-torsion, and it return to classical torsion Kisin modules when modulo $u$-torsion.
\subsection{(Generalized) Kisin modules}\label{subsec-generalized-Kisin} Let $(\s , E(u))$ be the Breuil--Kisin prism over $\O_K$ with $d = E(u )= E$ the Eisenstein polynomial of fixed uniformizer $\pi\in \O_K$. 
A $\varphi$-module $\m$ over $\s$ is an $\s$-module  $\m$ together $\varphi_\s$-semilinear map $\varphi_\m: \m \to \m$. Write $ \varphi^*\m = \s \otimes_{\varphi, \s} \m$. Note that $1 \otimes \varphi_\m  : \varphi ^* \m \to \m$ is an $\s$-linear map.   
A \emph{(generalized) Kisin module $\m$ of height $h$} is a $\varphi$-module $\m$ of finite $\s$-type so that there exists an $\s$-linear map $\psi: \m \to \varphi ^* \m$ so that $ \psi \circ (1 \otimes \varphi) = E^h \id_{\varphi^*\m}$ and $(1\otimes \varphi) \circ \psi = E^h\id_{\m}$. 
Maps between generalized Kisin modules are given by $\s$-linear maps which are compatible with $\varphi$ and $\psi$.  We denote by $\Mod^{\varphi, h}_{\s}$ the category of (generalized) Kisin module of height $h$.

In \cite{liu-Fontaine},  A \emph{Kisin module $\m$ of height $h$} is defined to be an \emph{\'etale}  $\varphi$-module $\m$ of finite $\s$-type so that
$\coker(1 \otimes \varphi) $ is killed by $E^h$. Here \'etale $\varphi$-module means that the natural map $\m \to \wh{\s [\frac 1 u]} \otimes_{\s} \m$ is injective. Since $E(u)$ is a unit in $\wh{\s [\frac 1 u]}$,
we easily see that the \'etale assumption implies that $(1\otimes \varphi) : \varphi^* \m \to \m$ is injective. 
Then existence and uniqueness of $\psi: \m \to \varphi ^*\m$,  in definition of (generalized) Kisin modules of height $h$,   then follows. 
That is, the {Kisin module $\m$ of height $h$} defined classically is (generalized) Kisin module of height $h$. So  in the following, we drop ``generalized" when we mention the object in $\Mod^{\varphi, h}_{\s}$.
If we need to emphasize $\m$ is also a Kisin modules of height $m$ classically defined, we will mention that it is \'etale. 
\begin{lemma}\label{lem-etale}
\leavevmode
\begin{enumerate}
	\item $\Mod^{\varphi, h}_\s$ is an abelian category.
	
	\item 	$\FM$ is \'etale if and only if $\FM$ has no $u$-torsion. 
	
	\item $\FM [\frac 1 p]$ is finite $\s [\frac 1 p]$-free. 
	 
	\end{enumerate}
\end{lemma}
\begin{proof} (1) is easy to check because $\varphi: \s \to \s$ is faithfully flat. (2) It is clear from the definition that if $\FM $ is \'etale then it has no $u$-torsion. Conversely, 
	let $\FM[p ^\infty]: = \{ x \in \FM| p ^n x = 0 \text{ for some } n > 0\}$ and $\FM' : = \FM/ \FM[p ^\infty]$. We get the short exact sequence 
	$0 \to \FM[p ^\infty ] \to \FM \to \FM ' \to 0$. It is clear that both $\FM [p ^\infty]$ and $\FM'$ are objects in $\Mod_\s^{\varphi , h}$
 and $\FM$ has no $u$-torsion then both $\FM[p ^\infty]$ and $\FM'$ has no $u$-torsion. Since $\FM[p ^\infty]$	is killed by some $p$-power, $\FM\otimes \wh \s [\frac 1 u]= \FM[\frac 1 u ]$. So $\FM[p ^\infty]$ has no $u$-torsion if and only if that $\FM [p ^\infty]$ is \'etale. Now $\FM'$ has no $p$-torsion, now we claim that $\FM'[\frac 1 p]$  is finite $\s [\frac 1 p]$-free, which will implies (3) and \'etaleness of $\FM'$. By \cite[\S1.2.1]{Fontaine}, $\FM' [\frac 1 p]\simeq \bigoplus \s[\frac 1 p]/ P_i^{a_i} $  with $P_i \in W(k)[u]$, monic irreducible and $P_i \equiv u ^{b_i} \mod p$, or $P_i = 0$. Without loss of generality, we may assume that $P_i \not = 0$ and show such $\FM'$ does not exist when $\FM' \in \Mod_\s^{\varphi, h}$. Consider wedge product $\n $ of $\FM'[\frac 1 p]$, then $\n\simeq \s [\frac 1 p]/ f$ with $f = \prod P_i ^{a_i}$ and  write $\varphi ^* : = (1 \otimes \varphi)$.  We also obtain $\varphi ^*: \varphi ^* \n \to \n $ and $\psi : \n \to \varphi^* \n$ so that $\psi \circ \varphi ^* = E(u)^h \id_{\varphi^* \n}$ and $\varphi ^* \circ \psi = E(u)^h \id_{\n}$ for some $h$. Since $\varphi^* \n \simeq  \s[\frac 1 p ]/ \varphi (f)$, we can write the above maps explicitly as $$ \s[\frac 1 p ]/ \varphi (f) \overset {\varphi^*} {\to}  \s[\frac 1 p ]/ f \overset {\psi}{\to} \s[\frac 1 p ]/ \varphi (f)$$
Write $x = \varphi ^*(1)$ and $y = \psi(1)$. We have $\varphi (f) x = f z'$ and $f y = \varphi (f) w'$ for some $z', w' \in \s[\frac 1 p]$. The condition $\psi \circ \varphi ^* = E(u)^h \id_{\varphi^* \n}$ and $\varphi ^* \circ \psi = E(u)^h \id_{\n}$ implies that $\varphi (f) E(u)^h = f z$ and $f E(u)^h = \varphi(f) w$ with $z, w \in \s [\frac 1 p]$. So $E(u)^{2h} = zw$. Since $E(u)$ is an Eisenstein polynomial, $z = z_0 E(u) ^l$ with $z_0$ a unit in $\s [\frac 1 p]$. Then $\varphi (f) = z_0 f E(u) ^{l-h}$.  We easily see $z_0 \in\s^\times $ as both $f$ and $E(u)$ monic.    So  $l -h >  0$ by mod $p$ on the both sides. Let $a_0= f(0)$ be the constant term of $f(u)$. Since $\varphi (f) (0) = \varphi (a_0) = z_0 (0) a_0 p^{l-h}$. Comparing $p$-adic valuation on the both sides, we see that $a_0 = 0$. Then we may write $f= u^m g$ with $g(0) \not = 0$. But then we have $u ^{pm-m} \varphi (g) = z_0 g E(u)^{l-h} $, which is impossible by comparing constant terms on the both sides.  
In summary, such $\FM'$ can not exist and $\FM ' [\frac 1 p]$ is finite $\s[\frac  1 p]$-free. 
\end{proof}


Let $\m$ be  a Kisin module of height $h$ and set  $\m [u ^{\infty}]: =\{x \in \m| u ^l x = 0 \text{ for some }  l\}$. It is that both $(1 \otimes \varphi_\m)(\varphi ^*\m[u ^\infty]) \subset \m[u ^\infty]$ and $\psi(\m [u ^\infty]) \subset \varphi ^*\m[u ^\infty]$. The above lemma shows that $\FM[u ^\infty] \subset \FM [p ^\infty] $ and $\FM/ \FM[u ^\infty]$ is \'etale. 
\begin{lemma}\label{lem-easy} The following short exact sequence is in $\Mod^{\varphi, h}_{\s}$
	$$ 0 \to \m [u ^\infty] \to \m \to \m/ \m[u ^\infty] \to 0  $$
with $\m / \m[u ^\infty]$ being \'etale.
\end{lemma}
It turns out that \'etale Kisin module enjoys many nice properties. Let $\Mod_{\s , \tor}^{\varphi , h}$ denote the full subcategory of $\Mod_{\s }^{\varphi , h}$ whose object $\m$ is torsion, i.e., killed by $p ^n$ for some $n$. 
The following Lemma is a part of \cite[Proposition 2.3.2]{liu-Fontaine}.

\begin{lemma}\label{lem-recall} 
	The following statements are equivalent for a torsion Kisin module $\m\in \Mod_{\s , \tor}^{\varphi , h}$: 
	\begin{enumerate}
		\item $\m$ is \'etale.
		\item $\m$ can be written as a successive quotient of $\m_i$ so that $\m_i \in  \Mod_{\s , \tor}^{\varphi , h}$ and $\m_i$ is finite $\ku$-free.
		\item $\m = \n/ \n'$ where $\n' \subset \n$ are Kisin modules of height $h$ and $\n'$ and $\n$ are finite free $\s$-modules.
	\end{enumerate}
\end{lemma}
\begin{corollary}\label{cor-devissage} Give an \'etale Kisin module $\m\in \Mod^{\varphi, h}_\s $. There exists \'etale Kisin module $\m _n \in \Mod^{\varphi, h}_\s$ killed by $p ^n $  satisfying $\m /p ^n \m [\frac 1 u] = \m_n [\frac 1 u]$ and $\m = \varprojlim_n \m_n$ 
\end{corollary}
\begin{proof}Let $M =  \m\otimes_{\s} \wh {\s[\frac 1 u]} $. Consider the exact sequence 
$0 \to p ^n M \to M \overset q \to M / p ^n M\to 0 $. Since $\m$ is \'etale, we see the natural map $\m\to M$ is injective. Set $\m_n = q(\m)\subset M / p ^n M$. It is easy to check that $\m_n [\frac 1 u] = \m / p ^n \m [\frac 1 u] = M / p ^n M$, $\m_n$ has no $u$-torsion and $\m = \varprojlim_n \m_n$ (since $\m$ is $p$-adically closed in $M$).   We just need to check that $\m_n$ has height $h$. This was proved by \cite[Proposition B 1.3.5]{Fontaine}. 
\end{proof}

In general, the category of \'etale Kisin modules is not abelian  but under some restrictions it could be abelian. Given $\m \in \Mod_{\s , \tor}^{\varphi , h}$,  let $M = \m[u ^\infty, p]: = \{ x\in \m[u ^\infty]|\  p x = 0\}$.   
\begin{lemma}\label{lem-control-torsion} If  $e h < p-1$ then  $M = 0$ and if  $eh  < 2( p -1)  $	then $M \simeq \bigoplus k$ or $0$. \end{lemma}
\begin{proof}
	So we have $\psi : M \to \varphi ^*M $ so that $\psi \circ (1 \otimes \varphi)   = d ^h\id_{\varphi ^*M } $.   We can write $M = \bigoplus_{j=1}^m  k [\![u]\!]/ u ^{a_j}$ with $a_j \geq 1$, and then $\varphi ^*M\simeq \bigoplus_{j=1}^m  k [\![u]\!]/ u ^{pa_j}$. Assume that $a = \max_j \{ a_j\}$ and let $x \in \varphi ^*M$ so that $u^{pa} x = 0 $ but $u ^{pa -1} x \not = 0$. Since  $\psi \circ (1 \otimes \varphi)   = u ^{eh}\id_{\varphi^*M}$, we conclude that $u^{eh} x \in \psi (M)$. Note that $u^a M = \{0\}$ and $\psi$ is $\ku$-linear,  we have $u ^{a+ eh} x =0 $. This forces that $a+ eh \geq pa$. That is, $a \leq \frac{eh}{p-1}$. Hence such $a$ can not exists if $eh <p -1$. If $eh< 2(p-1)       $ then  $a =1$ or 0. 	This proves the Lemma.   
	\end{proof}

\begin{proposition}\label{prop-abelian} If $eh< p -1$ then $\Mod_\s^{\varphi, h}$ is an abelian category. 
\end{proposition}
\begin{proof} By Lemma \ref{lem-control-torsion}, $\FM[u ^\infty] = 0$. 
\end{proof}

\begin{example} Let $E(u)= u -p$,  $\FM= k \simeq \ku/ u$ and $\varphi (1)= 1$. Let $\psi : \ku/ u \to \ku/ u ^p$ by $\psi(1) = u ^{p-1}$. 
	Then $\FM \in \Mod_{\s, \tor}^{\varphi, p -1}$. 
\end{example}

Let $\FM\in \Mod_{\s} ^{\varphi, h}$. 
Define \emph{Breuil--Kisin filtration} on $\varphi ^* \FM $ by 
\[
\Fil_{\BK} ^h\varphi^* \FM \coloneqq \mathrm{Im}(\psi \colon \m \to \varphi^*\m).
\]
In the case that $\m$ is \'etale then $\psi$ is injective as explained above, and we have an identification
\begin{equation}\label{eqn-BK-filtration}
\Fil_{\BK} ^h\varphi^* \FM \cong \{  x \in \varphi ^* \FM | (1\otimes \varphi) (x) \in E(u)^h \FM\}
\end{equation}
of submodules in $\varphi^* \FM$.
Since there is only filtration considered for Kisin modules in this section, we drop $\BK$ from the notation for this section. Finally there is $\varphi_\s$-semi-linear map 
$\varphi : = \varphi \otimes \varphi: \varphi^*\FM \to \varphi ^* \FM$.
It is clear that $\varphi (\Fil^ i \varphi^* \FM) \subset \varphi (E(u)) ^i \varphi^* \FM$.
If $\FM$ is \'{e}tale, then we define 
$\varphi _i : \Fil ^ i \varphi ^* \FM \to \varphi ^* \FM $ via $$\varphi _i(x) : = \frac {\varphi (x)}{\varphi (a_0^{-1}E(u)^i)},$$ where $a_0 p = E(0)$. 
\begin{lemma}\label{lem-filh-exact} Suppose that $0 \to \FM' \to \FM \to \FM '' \to 0$ is an exact sequence inside $\Mod_\s ^{\varphi, h}$ and all modules are \'etale. Then the following sequence is exact: 
	$$ 0 \to \Fil ^h \varphi ^*\FM' \to \Fil ^h \varphi ^*\FM \to\Fil ^h \varphi ^*  \FM '' \to 0$$
\end{lemma}
\begin{proof}
This easily follows that $\varphi^* : \Fil^h \varphi^* \FM \to E^h \FM$ is bijective. 
\end{proof}
\begin{remark} The above Lemma fails in general if $i < h$ or if the modules are not \'etale. 
\end{remark}
\subsection{Galois representation attached to \'etale Kisin modules} 
\label{subsec-Galos rep and Kisin} 
Recall that we  fix $\pi_n \in \overline K$ so that $\underline \pi : = (\pi _n) \in \O_{\C} ^\flat$ and $\pi_0= \pi$;  $K _\infty : = \bigcup_{n \geq 0} K (\pi _n)$ and $G_\infty : = \Gal (\Kbar/ K_\infty)$. We embed $\s \to A_{\inf}$ via $u \mapsto [\underline \pi]$. This embedding is compatible with $\varphi$, but not with the $G_K$-action. We have $\s \subset A_{\inf}^{G_\infty}$. 

For a Kisin module  $\FM \in \Mod_\s ^{\varphi, h}$, we can associate a representation of $G_\infty$ via
$$T_\s(\m): = \left(\m \otimes_\s W(\C^\flat)\right)^{\varphi =1}=\left(\m/\m [u ^\infty] \otimes_\s W(\C^\flat)\right)^{\varphi =1} . $$
So the Galois representation attached to $\FM$ is insensible to $u$-torsion parts because $\frac 1 u \in W(\C ^\flat)$. It is well-known that $T_\s$ is exact and there exists an $W(\C^\flat)$-linear isomorphism   
$$\m \otimes_\s W(\C^\flat) \simeq T_\s (\FM)\otimes_{\Z_p} W(\C^\flat),  $$
which is compatible with $\varphi$ and $G_\infty$-actions.

For many purposes, we define another variant $T^h_\s$ of $T_\s$: For an \'etale $\FM\in \Mod_{\s}^{\varphi, h}$,  we can naturally extend $\varphi _h : \Fil ^h \varphi ^* \FM  \to \varphi ^* \FM$ to $\varphi _h :  \Fil ^h \varphi ^* \FM \otimes_ \s A_{\inf}  \to \varphi ^* \FM \otimes_\s A_{\inf}$. 
$$T^h_\s(\m): = \left(\Fil^h \varphi ^* \m  \otimes_\s A_{\inf}\right)^{\varphi_h  =1} = \{ x\in \Fil^h \varphi^* \m\otimes_\s A_{\inf}, \  \varphi(x) = \varphi (a_0^{-1}E(u)^h)x\}.$$

\begin{lemma}\label{lem-twistm} Assume that $\FM \in \Mod_{\s} ^{\varphi, h}$ is \'etale. Then 
	\begin{enumerate}
\item 	$T^h _\s (\m) \simeq T_\s(\m) (h )$.

\item The following sequence is short exact 
$$\xymatrix{0 \ar[r]& T^h_\s (\m) \ar[r]&  \Fil^h \varphi ^* \FM \otimes_\s A_{\inf}\ar[r]^- {\varphi_h -1} &  \varphi ^* \FM \otimes _\s A_{\inf}\ar[r] & 0}$$
	\end{enumerate}
	
\end{lemma}
\begin{proof}First it is clear that $T_\s (\FM) = (\varphi ^* \FM \otimes_{\s} W(\C^\flat)) ^{\varphi =1}$ because $\varphi $ on $W(\C^\flat)$ is bijective. 
Recall that $pa_0= E(0)$. Let $\underline{\varepsilon} = (\zeta_{p ^n})_{n \geq 0}\in \O_\C ^\flat$ with $\zeta_{p ^n}$ satisfying $\zeta_1 =1$, $\zeta_{p^n}^p = \zeta_{p ^{n -1}}$ and $\zeta_p \not = 1$. 
By Example 3.2.3 in \cite{liu-notelattice}, there exists nonzero $\gt \in A_{\inf} $ so that $\gt \not = 0 \mod p$,  $\varphi(\gt) = a_0^{-1} E(u) \gt$ and $t \coloneqq \log[\underline \varepsilon]= \mathfrak c \varphi(\gt)$ with $\mathfrak c = \prod\limits_{n = 1}^\infty \varphi^n(\frac{a_0^{-1}E(u)}{p})\in A_{\cris}^*$. Write $\beta = \varphi (\gt)$. Consider map $\iota : T^h _\s (\FM) \to T_\s( \m)$ by $x \mapsto \frac{x}{\beta^h }$ for any $x \in \Fil ^h \varphi ^* \FM \otimes_\s A_{\inf}$. Since $\varphi (\beta)= \varphi (a_0^{-1}E(u)) \beta$, and $\beta \in W(\C^\flat)$ is invertible as $\gt \not = 0 \mod p$,  $\iota$ makes sense. Note that $\mathfrak c \in (A_{\cris}) ^{G_\infty} $. So $g (\beta)/ \beta = g (t)/ t$ is cyclotomic character for any $g \in G_\infty$. So $\iota: T^h _\s (\m) \to T _\s (\m) (h)$ is a map compatible with $G_\infty$-actions. We claim that $T^h _\s$ is an exact functor.  If so since $T_\s$ is also exact,  to show that $\iota$ is an isomorphism, we can reduce to the case that $\FM$ is killed by $p$ by Corollary \ref{cor-devissage}. In this case, $\FM$ is finite $\ku$-free. Picking a basis $e_1 , \dots, e_d$ of $\FM$, then $\varphi_\m (e_1, \dots , e_d)= (e_1, \dots , e_d)A$ with a $\ku$-matrix  $A$  so that there exists a $\ku$-matrix $B$ satisfying $AB = BA =  (a_0^{-1}E(u)) ^hI_d$. Let us still regard $e_i$ as a basis of $\varphi ^*\FM $. Then it is easy to check that $(e_1, \dots , e_d) B$ is a basis of $\Fil ^h \varphi ^* \FM$.  
	Now for any $x = \sum _i e_i \otimes a_i \in \varphi ^*\m \otimes _ \ku \C^\flat$, the equation $\varphi (x) = x$ is equivalent to $\varphi (X) = \varphi(A)^{-1} X$ where $X = (a_1, \dots , a_d)^T$.
	The latter gives $\varphi (\beta^h  X ) = \varphi (a_0^{-h}E(u)^h A^{-1})  (\beta^h X)=  \varphi (B) (\beta^h X)$, 
	which implies that $Y = \beta^h X $ is in 
	$(\O _{\C}^ \flat)^d$. That is $y = \beta^h x \in \varphi ^*\FM \otimes_{\ku} \O_{\C}^\flat$. Furthermore, consider $Z= B^{-1} \beta^h X$, since 
	$\varphi (Z)= \varphi( a_0^{-h}B^{-1} A^{-1} E(u)^h )  B  Z = B Z$. We conclude that $ Z$ has all entries in $\O_{\C}^\flat$. Then $\beta^h x = (e_1, \dots , e_d) B Z$ is inside $\Fil ^h \varphi ^* \FM\otimes\O_{\C}^\flat$. This proves that $\iota$ is surjective. Since $\iota$ is clearly injective, we show that $\iota$ is a isomorphism. 
	
	Now we prove the claim that $T^h_\s$ is exact. For this, it suffices to show that $\varphi _h -1$ is surjective and we once again reduces to the case that $\FM$ is killed by $p$. By writing the $\ku$-basis of $\FM$ as the above, we need to solve the equation $\varphi (X) -BX = Y$ for any $Y = (a_1, \dots , a_d)^T $ for $a_i \in \O_{\C}^\flat$. Since $\C^\flat$ is algebraic closed, we see $X$ exists with entries in $\C^{\flat}$. It is easy to compare valuation of each entry by equation $\varphi (X) = BX + Y$ to show that all entries of $X$ must be in $\O_{\C}^\flat$. 
	
\end{proof}


\subsection{Torsion Breuil modules} \label{subsec-Breuilmodules} We fix $0 \leq h \leq p-2$ for this subsection. Recall that $S= \mathcal A$ is the  $p$-adically completed PD-envelope
of $\theta: \s \twoheadrightarrow \O_K, u\mapsto \pi$, and for $i\ge 1$
write $\Fil^i S\subseteq S$ for the (closure of the) ideal generated by $\{ \gamma_n (E) = E^n/n!\}_{n\ge i}$.
For $i \le p-1$,  one has $\varphi(\Fil^i S) \subseteq p^i S$,
so we may define $\varphi_i:\Fil^i S\rightarrow S$ as $\varphi_i \coloneqq p^{-i}\varphi$. We have  $c_1 : = \varphi(E(u))/p \in S^\times$.

Let $'\textnormal{Mod}^{\varphi, h }_{S}$ denote the category
whose objects are triples $(\mathcal M, \Fil^h \calM, \varphi_h)$, consisting of

\begin{enumerate}
	\item an  $S$-module $\calM$
	\item an $S$-submodule $\Fil^h  \calM \subset \calM $ containing $\Fil ^h S \cdot \calM $.
	\item a $\varphi$-semi-linear map $\varphi_h : \Fil ^h  \calM  \to \calM $ such that for all $s \in \Fil^h S$ and $x\in \calM$
	we have $$\varphi_h (sx)= (c_1)^{-h }\varphi_h (s) \varphi_h(E(u)^hx).$$
	\item $\varphi_h (\Fil ^h \calM)$ generates $\calM$ as $S$-modules. 
\end{enumerate}
Morphisms  are given by $S$-linear maps preserving $\Fil^h $'s and commuting with
$\varphi_h$. A sequence is defined to be \emph{short exact}
if it is short exact as a sequence of $S$-module, and induces a
short exact sequence on $\Fil^h$'s. Let  $\textnormal{Mod}^{\varphi, h }_{S, \tor}$ denote the full subcategory of  $'\textnormal{Mod}^{\varphi, h  }_{S}$ so that $\calM$ is killed by a $p$-power and $\calM$ can be a written as successive quotient of $\calM_i$ in  $'\textnormal{Mod}^{\varphi }_{S}$ and each $\calM_i \simeq \bigoplus S_1 $ where $S_n : = S/ p ^n S$.

For each object $\calM \in  \textnormal{Mod}^{\varphi, h }_{S}$, we can extend $\varphi_h$ and $\Fil^h$ to $A_{\cris} \otimes_ S \calM$ in the following:
Since $A_{\cris}/ p ^nA_{\cris}$ is faithfully flat over $S/ p ^n$ by \cite[Lem 5.6]{Cais-Liu-BK-crys},
$A_{\cris} \otimes_S \Fil^h \calM  \to A_{\cris} \otimes \calM $ is injective and so we can define $\Fil^h (A_{\cris} \otimes _S \calM): = A_{\cris} \otimes _S \Fil ^h \calM$ and then $\varphi_h $ extends to $A_{\cris} \otimes_ S \calM$. 
This allows to define a representation of $G_\infty$ via
$$ T_S (\calM ) \coloneqq (\Fil ^h (A_{\cris} \otimes _S \calM))^{\varphi_h =1}. $$

Now let us recall the relation of classical torsion Kisin modules and objects in $\Mod_{S, \tor}^{\varphi, h}$ and their relationship to torsion Galois representations. Let $\Mod_{\s, \tor\et}^{\varphi , h}$ denote the category of \'etale torsion Kisin module of height  $h$.   In this subsection, all torsion Kisin modules are \'etale torsion Kisin modules, i.e., $\m$ is $u$-torsion free.  For each such $\m$, we construct an object $\calM \in \Mod^{\varphi, h}_{S, \tor}$ as the following: $\calM: = S \otimes_{\varphi, \s} \m$ and
$$\Fil ^h \calM: = \{x \in \calM | (1\otimes \varphi) (x) \in \Fil ^h S \otimes_{\s}\m\}; $$
and $\varphi_h : \Fil ^h \calM \to \calM$ is defined as the composite of following map

\begin{equation*}
\xymatrix{
	{\Fil ^h \calM} \ar[r]^-{1\otimes \varphi_{\m}} & {\Fil ^h S \otimes_{\s}\m }
	\ar[r]^-{\varphi_h \otimes 1} & S\otimes_{\varphi,\s} \m = \calM
} .
\end{equation*}
We write $\u \calM(\m)$ for $\calM \in \Mod^{\varphi , h}_{S, \tor}$ built from Kisin module $\m\in \Mod_{\s, \tor\et}^{\varphi , h}$ as the above. Note that $A_{\cris} \otimes_S \u \calM(\m) = A_{\cris} \otimes_{\varphi, \s}\m$.

\begin{proposition}\label{prop-Galois-compatible} The above functor induces an exact equivalence between $\Mod^{\varphi, h}_{\s, \tor\et }$  and  $\textnormal{Mod}^{\varphi, h}_{S, \tor}$. Furthermore, there exists short exact sequence
	\begin{equation}\label{eqn-exact_S}
	\xymatrix{ 0 \ar[r] & T_S  (\calM) \ar[r] & A_{\cris} \otimes_S \Fil^h \calM  \ar[r]^-{\varphi_h-1} &  A_{\cris}\otimes_S \calM\ar[r]& 0 }
	\end{equation}
and an isomorphism of $G_\infty $-representations	
	$$T_S (\u \calM (\m)) \simeq T_\s(\m)(h).$$
\end{proposition}
\begin{proof} The equivalence of functor together with exactness is \cite[Thm 2.2.1]{CarusoLiu}, which built on Breuil and Kisin's results (see \cite[Proposition 3.3.1]{LiuT-CofBreuil}). Consider an exact sequence in $\textnormal{Mod}^{\varphi, h}_{S, \tor}$, 
$$ 0 \to \calM'' \to \calM \to \calM ' \to 0. $$
Then we have the following diagram 
$$ \xymatrix{  0 \ar[r] & T_S (\calM'' ) \ar[r] \ar[d]  & T_S (\calM) \ar[r]\ar[d] & T_S (\calM')\ar[d]\ar[r] & 0 \\0 \ar[r] & A_{\cris} \otimes_S \Fil ^h \calM '' \ar[r]\ar[d]^{\varphi_h -1}  & A_{\cris} \otimes_S \Fil ^h \calM \ar[r]\ar[d]^{\varphi_h -1} & A_{\cris} \otimes_S \Fil ^h \calM  '\ar[r]\ar[d]^{\varphi_h -1} & 0 \\  0 \ar[r] & A_{\cris} \otimes_S  \calM '' \ar[r]  & A_{\cris} \otimes_S  \calM \ar[r] & A_{\cris} \otimes_S \calM '\ar[r] & 0 \\}$$
By the definition of exactness in $\textnormal{Mod}^{\varphi, h}_{S, \tor}$ and since $A_{\cris} / p ^n$ is flat over $S/ p ^n$, 
we see that last two rows of the above diagram are exact. 
So to show $\varphi_h -1$ is surjective on $\calM$, we reduce to situation that $\calM$ is killed by $p$. 
Also the surjectivity of $\varphi_h -1$ implies that the functor $T_S$ is exact from the above diagram. 
So let us first accept that $\varphi_h-1$ is surjective and postpone the proof in the end.

Now let us construct a natural map $\iota : T^h _\s (\m) \to T_ S (\u \calM (\m))$. Write $\calM : = \u\calM(\m)$. 
It is clear that $\Fil ^h \varphi ^* \FM \subset \Fil ^h \u \calM (\m) $ compatible with the injection $\varphi ^*\FM \inj \calM$. But $\varphi_h$ defined on Kisin modules are slightly different from that on  Breuil modules. By chasing definitions, we see that for any $x \in \Fil^h \varphi^* \m $, $\varphi_{h , \calM} (x) =(a_0^{-1} c _1)^h \varphi _{h , \varphi^*\m}(x)$. Recall  $\mathfrak c = \prod\limits_{n = 1}^\infty \varphi^n(\frac{a_0^{-1}E(u)}{p})\in A_{\cris}^*$ in the proof Lemma \ref{lem-twistm}.
Since $\varphi (\mathfrak c) = a_0 ^{-1}c_1 \mathfrak c$,  the map  $ \iota:  A_{\inf} \otimes _\s \Fil^h \varphi^* \m  \to A_{\cris} \otimes_S \calM$ by  $\iota(x)= \mathfrak c^h x$ induces a map   $\iota: T^h _\s (\m) \to T_S (\calM)$. 

To show that $\iota$ is isomorphism, since $T_\s$, $T_S$ and $\u\calM$ are all exact, 
we reduce to the case that $\m$ is killed by $p$ where $\m$ is finite $\ku$-free. As the same argument in Lemma \ref{lem-twistm}, there exists a basis ${e_1 , \dots , e_d}$ of $\varphi ^*\m$ so that $\Fil ^h\varphi ^* \m$ has basis $(e_1, \dots, e_d) B$,  $\varphi (e_1, \dots , e_d)= (e_1, \dots, e_d)\varphi (A)$ and $AB = B A = (a_0^{-1} E(u))^h I_d$. So any $x\in T^h _\s(\m)$ corresponds to the solution of $\varphi (X)= B X$.  Since $\calM = \u \calM (\m)$, it is straightforward to compute that $\calM$ also has $S_1$-basis $e_1, \dots , e_d$, $\Fil ^h \calM$ is generated by $(e_1, \dots , e_d) B$ and $\Fil^p S_1 \calM $. Note that $a_0^{-1}c_1 \equiv 1 \mod (p,  \Fil^p S) $. So $T_S (\calM)$ corresponds to solutions 
$\varphi (X) = B X \mod \Fil ^p A_{\cris, 1}$ where $A_{\cris ,1 }= A_{\cris}/ p A_{\cris}$. Now it suffices to show the following map is bijective 
$$\{ X | \varphi (X) = B X, \  x_i \in \O_{\C}^\flat \} \longrightarrow \{X | \varphi (X) = B X \mod \Fil ^p A_{\cris, 1},\ x_i\in A_{\cris, 1 }\}$$
Let $v$ denote the valuation on $\O^\flat_{\C}$ which normalized by $v(u ^e)=1$.  Suppose that $X$ is in the kernel then $X \in E(u)^p \O^\flat_{\C}$. So $v(x_j)\geq p, \forall j$. 
Let $x_i$ be the entry with least valuation. Note that $v(\varphi (x_j)) = p v(x_j)$ for any $j$ and $A \varphi (X) = u^{eh} X$. The mini possible of left side valuation is $p v(x_i)$, while the the right side is
$h + v (x_i)$. This is impossible when $v (x_i)\geq p $ because $h \leq p-2$. So this implies that $X= 0$.
Indeed, if $v(x_i)\geq 2$ then the same proof show that $X= 0$. That is, if $X_1, X_2$ are two solution in the left side and $X_1 \equiv X_2 \mod E(u)^2$ then $X_1 = X_2$

Conversely, let $Z$ be the vector inside $A_{\cris, 1 }$ so that $\varphi (Z) = B Z \mod \Fil ^p A_{\cris ,1}$. Then there exists  $Z_0$ with entries in $\O ^{\flat}_{\C}$ so that 
$\varphi (Z_0) = B Z_0  +  E(u)^pC$ where $C$ is a vector with entries in $\O^\flat_{\C}$. Note that $E(u)^p = E(u)^{p-h} BA$. So we may write
$\varphi(Z_0) = B (Z_0+ E(u)^{p-h}AC)$. Let $Z_1= Z_0 + E(u)^{p-h} AC$. Then $\varphi (Z_1) = BZ_1 +  u ^{pe(p-h)}C_1$ with $C_1 = - \varphi(AC) $. Note that $pe (p-h)> pe> h e$. we can write
$BZ_1 +  u ^{pe(p-h)}C_1= B (Z_1 + u^{\alpha}AC)$ with $\alpha = pe(p-h)-h$. Set $Z_2= Z_1 + u ^\alpha AC$ then  $\varphi(Z_2) = Z_2 + u ^{p\alpha} C_2$. Continues this steps, we see that $Z_n$ converges in $\O^\flat_{\C}$ to $Z'$ so that $\varphi(Z') = B Z'$ with $Z' \equiv Z_0 \mod E(u)^{p-h} $. This settles the bijection of these two sets and completes the proof.

It remains to show that $\varphi_h -1: \Fil^h \calM \otimes_S A_{\cris} \to \calM \otimes_S A_{\cris}$ is surjective and we may assume that $\calM = \u\calM (\m)$ with $\m$ killed by $p$. Now that $\calM \otimes_S A_{\cris} = \varphi^* \m \otimes_{\ku} {\mathcal O}^\flat_{\C} + \varphi^* \m \otimes_{\ku} \Fil ^pA_{\cris, 1}.$ By Lemma \ref{lem-twistm} (2), it suffices to show that for $y = m \otimes a$ with $m \in \varphi^* \m$ and $a\in \Fil ^p A_{\cris, 1}$ there exists a $x \in  \Fil^h \calM \otimes_S A_{\cris} $ so that $\varphi_h (x) - x = y$. Since $\varphi_h (a) = 0 $ for 
$a \in \Fil ^p A_{\cris, 1}$, then $y = -x$ is required. 
\end{proof}
\begin{remark}\label{rem-compatible} If we combine the isomorphisms $\eta : T_\s (\m)(h) \to T ^h_\s (\m) \to T_S (\u \calM (\m))$ defined by $x \mapsto \beta^h x \mapsto (\beta \mathfrak c)^h x= t ^h x $. 
	The isomorphism $\eta: T_\s(\m) (h) \simeq T_S (\u \calM(\m))$ is natural in the following sense: Suppose that $\m \otimes_\s A_{\inf}$ has a $G_K$-actions so that $G_K$-action is semi-linear on $G_K$-action on $A_{\inf}$ and commutes with $\varphi_{\m}$. Then this $G_K$-action induces a $G_K$-actions on $\u\calM(\m) \otimes_S A_{\cris}$ compatible with $\Fil ^ h$ and $\varphi$. Then both $T_\s(\m)(h)$ and $T_S(\calM)$ has $G_K$-actions and $\eta$ is $G_K$-compatible isomorphism.
\end{remark}
Regard both $S$ as subring of $K_0 [\![u]\!]$.  Define $I^+ S = S \cap u K_0 [\![u]\!]$ and $I^+ = u \s$. Clearly we have a natural map $q: \FM/ I ^+ \to \underline{\calM}(\FM) / I ^+ S$. By d\'evissage to the situation that  $\FM$ killed by $p$, we obtain
\begin{corollary}\label{cor-length-compare}
Let $\m\in \Mod_{\s, \tor\et}^{\varphi , h}$. Then we have
\[
\textnormal{length}_{W(k)} (\underline{\calM} (\FM) / I ^+ S) = \textnormal{length}_{W(k)} (\FM / u \FM) = \textnormal{length}_{\Z} (T_S (\underline{\calM} (\FM)))= \textnormal{length}_{\Z} (T_\s(\FM)).
\]
\end{corollary}

Now let us add one extra structure to  $\Mod_{S, \tor}^{\varphi, h }$ to make $T_S (\calM)$ a $G_K$-representation. 
Let  $\Mod_{S, \tor}^{\varphi, h, \nabla}$ denote the category of the object $(\calM, \Fil ^h \calM, \varphi _h, \nabla) $ where 
\begin{enumerate}
\item $(\calM, \Fil ^h \calM, \varphi _h)$ is an object in $\Mod_{S, \tor}^{\varphi,  h }$
\item $\nabla: \calM \to \calM$ is a connection satisfying the following: 
\begin{enumerate}
    \item $E \nabla(\Fil^h \calM )\subset \Fil^h \calM $. 
    \item  the following diagram  commutes:
	\begin{equation}
	\begin{split}
	\xymatrix{ \Fil^h  \calM\ar[d]_{E(u) \nabla} \ar[r]^-{\varphi_h } & \calM \ar[d]^{c_1 \nabla}\\
		\Fil ^h  \calM   \ar[r]^-{ u ^{p-1} \varphi_h } &\calM }
	\end{split}
	\end{equation}
\end{enumerate}
\end{enumerate}

Let us explain the relationship between objects in   $\Mod_{S, \tor}^{\varphi, h, \nabla} $ and Breuil modules studied in work of Breuil and Caruso. Let 
$N_S : S \to S$ be $W(k)$-linear differentiation so that $N_S(u)= u$. An object $\calM $ in  $\Mod_{S, \tor}^{\varphi, h }$  is called a \emph{Breuil module} if $\calM$ admits a
$W(k)$-linear morphism $N: \calM \to \calM $ such that :
\begin{enumerate}
	\item for all $s\in S $ and $x\in \calM$, $N(sx)=N_S(s) x + s N(x)$.
	\item $E(u) N(\Fil ^h  \calM )\subset \Fil ^h \calM $.
	\item the following diagram  commutes:
	\begin{equation}\label{eqn-N-diagram}
	\begin{split}
	\xymatrix{ \Fil^h  \calM\ar[d]_{E(u) N} \ar[r]^-{\varphi_h } & \calM\ar[d]^{c_1 N}\\
		\Fil ^h  \calM \ar[r]^-{\varphi_h } &\calM }
	\end{split}
	\end{equation}
\end{enumerate}
\begin{remark} Breuil and Caruso use convention $N_S(u )= -u$. In fact, there is almost no difference for entire theory by using $N_S(u)= u$ except for the formula \eqref{eqn-action-N} need to  change sign comparing with the similar formula \cite[(5.1.1) ]{LiuT-CofBreuil} 
\end{remark}

Let   $\Mod_{S, \tor}^{\varphi, h,  N} $ denote the category of Breuil modules. There is a natural functor   $\Mod_{S, \tor}^{\varphi,  h, \nabla } \to   \Mod_{S, \tor}^{\varphi, h, N}$ by define $N _\calM  = u\nabla$. It is easy to chase the diagram to see this functor makes sense.  So we also call objects in $\textnormal{Mod}_{S, \tor}^{\varphi, h, \nabla}$ Breuil modules. 

Now we can define a $G_K$-action on $\calM \otimes_S A_{\cris}$ as in:
for any $\sigma \in
G_K$, 
any  $x \otimes a \in \calM\otimes_S A_{\cris}$,
define
\begin{equation}\label{eqn-action}
\sigma(x \otimes a)= \sum_{i=0}^\infty \nabla^i (x) \otimes  \gamma_i
\left ( \sigma ([\upi])- [\upi] \right ) \sigma (a) .
\end{equation}

We can also define a $G_K$-action on $\calM \otimes_S A_{\cris}$ as in \cite[\S5.1]{liu-notelattice}:
for any $\sigma \in
G_K$, recall $\ue(\sigma)= \frac{\sigma ([\upi])}{[\upi]}\in A_{\inf}$.  For
any  $x \otimes a \in \calM\otimes_S A_{\cris}$,
define
\begin{equation}\label{eqn-action-N}
\sigma(x \otimes a)= \sum_{i=0}^\infty N^i (x) \otimes  \gamma_i
(\log(\ue(\sigma))) \sigma(a).
\end{equation}
where $\gamma _i(x)= \frac{x^i}{i!} $ is the standard divided power. We claim that \eqref{eqn-action} and \eqref{eqn-action-N} are the same formula. Let us postpone the proof in \S \ref{subsec-two-equations} as the proof is just long combinatoric calculation.  

Note that if $\sigma \in G_\infty$, then $\log(\ue(\sigma))= 0$ and
$\sigma (x \otimes a)= x \otimes \sigma(a)$. Thus $G_K$-action defined
above (if it is well defined) is compatible with the natural
$G_\infty$-action on $\calM \otimes_S A_{\cris}$.
\begin{lemma}\label{lem-action-from_N}
	The above action is well defined $A_{\cris}$-semi-linear $G_K$-action on $\calM
	\otimes_S A_{\cris}$ and compatible with $\Fil^h (\calM \otimes_S A_{\cris})$ and $\varphi_h$.
\end{lemma}
\begin{proof}The proof of  \cite[\S5.1]{liu-notelattice} essentially applies here. It is standard to check that \eqref{eqn-action} is well-defined map; it is $A_{\cris}$-semi-linear-action on $\calM \otimes _S A_{\cris}$ and compatible with $G_K$-action on $A_{\cris}$; and $G_\infty$-acts on $\calM \otimes 1$ trivially. It is clear that $\log(\ue (\sigma))\in \Fil^1 A_{\cris} $. So by  that  $E(u) N(\Fil^h \calM )\subset \Fil ^h \calM $, we see that
	$\sigma (\Fil ^h (\calM \otimes _S A_{\cris}))  \subset \Fil^h  (\calM \otimes _S A_{\cris})$. The only thing left to check is that $\varphi_h$ commutes with $G_K$-action, which can be reduce to check the following: write $a = - \log (\ue (\sigma))$ and pick $x \in \Fil ^h  \calM$, we have
	$$\varphi_h (  \gamma_i(a) \otimes N^i (x))= \gamma _i (a) \otimes N^i (\varphi_h (x)). $$
It is clear that  $\varphi(a)= p a $. So $\varphi (\gamma_i (a))= \gamma_i (a) c_1^{-i} \varphi (E(u)^i)$. So the above equality is reduced to check
$c_1^{-i} \varphi_h (E(u)^i N^i (x)) = N^i(\varphi_h (x))$ and this can be check by induction on $i$.
\end{proof}
\begin{corollary} Given a Breuil module  $\calM \in \textnormal{Mod} ^{\varphi,h,  N} _{S , \tor} $,
then $T_{S}(\calM )$ (as a $G_\infty$-representation) extends to  a $G_K$-representation. 
\end{corollary}

To summarize our section, we return to the situation of \S\ref{subsection-stru of crys}  where $\calM^i : = \rH^i_{\cris} (X_n /S_n)$ is proved to admit  structures $\Fil^i\calM ^i = \rH ^i _{\cris} (X_n /S_n, \cI^{[i]}_{\cris}),$ $\varphi_i: \Fil^i \calM^i \to \calM^i $ and $\nabla : \calM^i \to \calM^i $. Obviously, our axioms of $\textnormal{Mod}^{\varphi, h, \nabla}_{S, \tor}$ is aimed at describing these structures of $\rH^i_{\cris} (X_n /S_n)$. 
\begin{definition}\label{def-Breuil-modules}
For $i \leq p-2$, we call that $\rH^i_{\cris} (X_n /S_n)$ is a Breuil module if the quadruple 
$$ \left (\rH^i_{\cris} (X_n /S_n), \rH ^i _{\cris} (X_n /S_n, \cI^{[i]}_{\cris}), \varphi_i, \nabla \right )$$ 
constructed in \S\ref{subsection-stru of crys} is an object in $\Mod_{S, \tor}^{\varphi, i , \nabla}$, which is equivalent to the triple
\[
\left (\rH^i_{\cris} (X_n /S_n), \rH ^i _{\cris} (X_n /S_n, \cI^{[i]}_{\cris}), \varphi_i \right )
\]
being an object in $\Mod_{S, \tor}^{\varphi, i}$. 
\end{definition}

Our main theorem is to show that $\rH ^i _{\cris} (X_n /S_n)$ together with these structures is indeed a Breuil module  when $ei < p -1$. 

\section{Torsion cohomology and comparison with \'etale cohomology}
In this section, we collect our previous preparations to understand the structures of torsion crystalline cohomology and its relationship with \'etale cohomology via torsion prismatic cohomology. In the end, we 
show that if $ei < p -1$ then $p ^n$-th torsion crystalline cohomology $\rH ^i_{\cris} (X_n /S_n)$ has structure of torsion Breuil module to compare to $\rH ^i_{\et} (X_{\bar \eta}, \Z/ p ^n \Z)$ via $T_S$, where $X_{\bar \eta}$ is a geometric generic fiber of $X$. 

\subsection{Prismatic cohomology and (generalized) Kisin modules}\label{subsec-prism-is-kisin}
Let $(A,I)$ be any prism. As in the end of  \S\ref{subsection-stru of crys}, 
for any $n \geq 1$, we define torsion prismatic cohomology $\mathrm {R\Gamma}_\Prism (X_n/A_n): = \mathrm{R\Gamma} _{\Prism}(X/A, \O _{\Prism}/ p ^n\O _{\Prism} ) = \RG _{\Prism} (X/ A) \otimes^\BL_{\Z}\Z/ p ^n \Z$. We have 
$\RG_\Prism (X_n / A_n) \simeq \RG _{\qsyn}(X, \Prism_{-/A}/p ^n )\simeq \RG _{\qsyn} (X,\Prism_{-/A} ) \otimes^\BL_\Z \Z/ p ^n \Z$. 

\begin{warning}\label{Warning-not-intrinsic}
We warn readers that the notation $\mathrm {R\Gamma}_\Prism (X_n/A_n)$ is misleading, as it might suggest that this cohomology
theory only depends on the mod $p^n$ reduction of $X$ which is not true.
See \cite[Remark 2.4]{BMS1} for a counterexample.
\end{warning}

\begin{proposition}
\label{prop-height} 
Assume that $(A,I)$ is transversal and $\varphi: A \to A$ is flat. Then $\rH ^i_\Prism(X_n/ A_n)$ has height $i$. 
\end{proposition}
\begin{proof} We follow the same idea of 
 \cite[Corollary 15.5]{BS19} which proved that $\rH ^i _\Prism (X/A)$ has height $i$. 
 Examining the proof, it suffices to show that $\varphi ^* \RG_{\Prism} (X_n / A_n ) \simeq L \eta_I \RG_{\Prism} (X_n /A_n)$ when $X = \Spf (R)$ is  an affine smooth $p$-adic formal scheme over $A/I$.
 By Theorem 15.3 of \emph{loc.~cit}, we have
 $\varphi ^* \RG_{\Prism} (X / A ) \simeq L \eta_I \RG_{\Prism} (X /A)$. Since $\varphi: A \to A $ is flat, it suffices to show that 
 \begin{equation}\label{eqn-eta mod p n}
 \left (L \eta_I \RG_{\Prism} (X /A) \right ) \otimes^{\BL} _{\Z} \Z/ p ^n \Z \simeq L \eta_I \left (\RG_{\Prism} (X /A)  \otimes^{\BL}_{\Z} \Z/ p ^n \Z \right ). 
 \end{equation}
 Now we may apply \cite[Lemma 5.16]{Bhatt-special} to the above  by $g = p ^n$ and $f =d$. So we need to check that $\rH^* (\RG_{\Prism} (X /A) \otimes^{\BL} _A A/d)$ has no $p^n$-torsion. This follows from the Hodge--Tate comparison
 \[
 \rH^i  (\RG_{\Prism} (X /A) \otimes^{\BL} _A A/I)\simeq \Omega^i _{X/ (A/I)}\{i\}.
 \]
\end{proof}

\begin{corollary}\label{cor-pris is Kisin}
For $n \in \mathbb{N} \cup \{\infty\}$, the $\varphi$-module $\rH ^i _{\Prism} (X_n /\s_n)$ is an object of $\Mod_{\s} ^{\varphi, i}$, i.e., 
a (generalized) Kisin module of height $i$ and $T_\s (\rH ^i_{\Prism} (X_n / \s_n))\simeq \rH ^i _{\et} (X _{\bar \eta}, \Z/ p ^n \Z)$. 
\end{corollary}
\begin{proof}
It suffices to prove that $T_\s (\rH ^i_{\Prism} (X_n / \s_n))\simeq \rH ^i _{\et} (X _{\bar \eta}, \Z/ p ^n \Z)$. Write $\m ^i _n : = \rH^i _{\Prism} (X_n / \s_n )$, ${\mathcal X} \coloneqq \Spf\O_{\C} \times_{\Spf \O_K} X$. 
For $n \not = \infty$, by \cite[Theorem 1.8 (4) (5)]{BS19}, we have 
\[  \rH^i_{\et} (X_{\bar\eta}, \Z/ p ^n \Z) \simeq \left (\rH ^i(\mathrm{R\Gamma}_{\Prism}({\mathcal X}/ A_{\inf})/ p ^n) [\frac{1}{E(u)}] \right)^{\varphi =1}= ( \m^i_n\otimes _\s  W_n (\O_{\C}^\flat) [\frac 1 u]) ^{\varphi =1}= (\m^i _n\otimes _\s W_n (\C^\flat) ) ^{\varphi =1}, \]
which is just $ T_\s (\m ^i _n)$. The case of $n =\infty$ easily follows by taking inverse limits.
\end{proof}
\begin{remark}\label{rem-G-compatible}  The $G_\infty$-action on $T_\s (\m^i _n)$ discussed in \S \ref{subsec-Galos rep and Kisin} naturally extends to a $G_K$-action by isomorphism $\m^i_n \otimes_\s A_{\inf} \simeq \rH ^i _{\Prism} ({\mathcal X}/ A_{\inf})$, which admits a natural $G_K$-action that commutes with $\varphi$. In this way  $T_\s (\rH ^i_{\Prism} (X_n / \s_n))\simeq \rH ^i _{\et} (X _{\bar \eta}, \Z/ p ^n \Z)$ is an isomorphism of $G_K$-actions. 
\end{remark}
Let $X_k : = X \times_{\Spf(\O_K)} \Spf(k)$ be the closed fiber of $X$. 
\begin{lemma}\label{lem-length of both fibers} If $\Len_{W(k)} \rH ^i_{\cris}(X_k/ W_n (k)) = \Len_{\Z} \rH ^i_{\et} (X_{\bar \eta}, \Z/ p ^n \Z)$ then $\FM^j _n$ has no $u$-torsion for $j = i , i +1$. 
\end{lemma}
\begin{proof} We claim that $\RG _\Prism (X_n / \s_n) \otimes^\BL_{\s} W(k) \simeq \RG _{\cris} (X_k /W_n (k)). $ To see this, first note that $(\s, E)\to (W(k), p)$ by mod $u$ is a map of prisms. So \cite[Theorem 1.8 (5)]{BS19} proves that $\RG _\Prism (X / \s) \otimes^\BL_{\s} W(k) \simeq \RG _{\Prism} (X_k /W (k))$. Then Theorem 1.8 (1) \emph{loc. cit.} shows that  $\RG _\Prism (X / \s) \otimes^\BL_{\s} W(k) \simeq \RG _{\cris} (X_k /W (k))$. Then the claim follows by $\otimes^\BL_{\Z} \Z/ p ^n\Z$ on both sides. 

The claim immediately shows that the exact sequence 
\begin{equation}\label{eqn-exact seq for length}
0 \to \m ^i _n / u \m ^i_n \to \rH^i _{\cris} (X_k / W_n (k)) \to \m^{i +1}_n [u]\to 0 \end{equation}
So $\Len_{W(k)} \m ^i _n / u \m ^i_n \leq \Len _{W(k)} \rH^i _{\cris} (X_k / W_n (k)) $. On the other hand, consider the exact sequence in Lemma \ref{lem-easy} with $\m: = \m ^i _n$
	$$ 0 \to \m [u ^\infty ] \to \m \to \m/ \m[u ^\infty] \to 0  $$
Write $\m^{\et}: = \m/ \m[u ^\infty]$. Since $\m^{\et}$ has no $u$-torsion,  the above exact sequence remaining exact by modulo $u$. So we have 
$\Len _{W(k)} (\m^{\et}/ u \m ^{\et}) \leq \Len_{W(k)} \m/ u \m$ and equality holds only when $\m[u ^\infty] = \{0\}$. Since $T_\s (\m) = T_\s (\m ^{\et})$, and $T_\s (\m)\simeq \rH^i _{\et} (X_{\bar \eta}, \Z/ p ^n \Z)$ by Corollary \ref{cor-pris is Kisin}, Corollary \ref{cor-length-compare} proves the following inequalities 
\[\Len_{\Z} \rH^i_{\et} (X_{\bar \eta}, \Z/ p ^n \Z)= \Len_{W(k)} (\m ^{\et}/ u \m ^{\et})\leq \Len _{W(k)} (\m / u \m). \]
Now combine with the exact sequence \eqref{eqn-exact seq for length}, we conclude that \[ \Len_{\Z} \rH ^i _{\et} (X_{\bar \eta}, \Z/ p ^n \Z) \leq \Len_{W(k)} \rH^i _{\cris} (X_k/ W_n(k))\]
and equality holds only if all the above inequalities become equalities and $\m^{i}_n$ and $\m^{i+1}_n$ have no $u$-torsions. 
\end{proof}

\subsection{Nygaard filtration and Breuil--Kisin filtration} By Corollary \ref{cor-pris is Kisin}, $\m_n ^i: = \rH ^i_{\Prism} (X_n / \s _n)$ is a Kisin module of height $i$. Then $\varphi^* \m^i_n \simeq \rH^i_{\qsyn} (X,  \Prism^{(1)}_{-/\s} \otimes^\BL_{\Z} \Z/ p ^n \Z)$ admits two filtrations: Breuil--Kisin filtration defined in \eqref{eqn-BK-filtration} and Nygaard filtration $\rH^i _{\qsyn} (X , \Fil^i_{\rm N} \Prism^{(1)}_{-/\s} \otimes^\BL_{\Z}{\Z/ p ^n \Z})$. The aim of this subsection is to compare these two filtrations. 

This theme can be put in more general setting for a bounded prism $(A, I)$. Recall that in \cite[\S 15]{BS19} the authors studied 
$\Prism_{-/A}$ and $\Prism^{(1)}_{-/A}: = A \wh \otimes^{\BL}_{\varphi, A} \Prism_{-/A} $ as sheaves on $\mathrm{qSyn}_{A/I}$. 
Also constructed in \emph{loc.~cit.~}is the so-called Nygaard filtration $\Fil^j_{\rm N} \Prism^{(1)} _{-/A}$,  also discussed \S \ref{subsec-NygaardFil}.  
For any $n \in \mathbb{N} \cup \{\infty\}$, set $\Prism^{(1)}_n : = \Prism^{(1)}_{-/A}\otimes^{\BL} _\Z \Z/ p^n \Z$ and 
$\Fil^j _{\rm N} \Prism^{(1)}_n : = \Fil^j_{\rm N} \Prism^{(1)} _{-/A} \otimes^{\BL}_\Z \Z/ p ^n \Z$. 
Here and below, we adopt the convention that $n = \infty$ means we do not perform any base change.

\begin{lemma}
\label{general property of Nygaard filtration}
Let $(A,I)$ be a bounded prism.
Let $X$ be a smooth ($p$-adic) formal scheme over $\Spf(A/I)$ of relative dimension $n$.
Then we have:
\begin{enumerate}
\item The Nygaard filtration $\mathrm{R\Gamma}(X_{\qsyn}, \Fil^{\bullet}_{\mathrm{N}})$ on $\mathrm{R\Gamma}(X_{\qsyn}, \Prism^{(1)}_{-/A})$
is complete.
\item The natural map
\[
\Fil^i_{\mathrm{N}} \otimes_{A} I^j \to \Fil^{i+j}_{\mathrm{N}}
\]
of quasisyntomic sheaves induces a morphism
\[
\rH^l(X_{\qsyn}, \Fil^i_{\mathrm{N}}) \otimes_{A} I^j \rightarrow \rH^l(X_{\qsyn}, \Fil^{i+j}_{\mathrm{N}})
\]
which is an isomorphism when either $l \leq i$, and an injection when $l = i+1$.
When $i \geq n$ this map induces an isomorphism
\[
\mathrm{R\Gamma}(X_{\qsyn}, \Fil^n_{\mathrm{N}}) \otimes_{A} I^j
\cong \mathrm{R\Gamma}(X_{\qsyn}, \Fil^{n+j}_{\mathrm{N}}).
\]
\item The natural map
\[
\varphi \colon \Fil^i_{\mathrm{N}} \to \Prism_{-/A} \otimes_{A} I^i
\]
induces a map on cohomology
\[
\rH^l(X_{\qsyn}, \Fil^i_{\mathrm{N}}) \to \rH^l(X_{\qsyn}, \Prism_{-/A}) \otimes_{A} I^i
\]
which is an isomorphism when $l \leq i$
and injective when $l = i+1$.
\end{enumerate}
Moreover their derived mod $p^m$ counterparts hold true as well.
\end{lemma}

We thank Bhargav for pointing out the statement (3) above, which we did not realize can be proved so easily.
This significantly simplifies an earlier draft.

\begin{proof}
(1) follows from (2). Indeed, (2) implies the Nygaard filtration on 
$\mathrm{R\Gamma}(X_{\qsyn}, \Fil^i_{\mathrm{N}})$ is simply the $I$-adic filtration,
hence it is complete.

(2) follows from the following exact triangle of quasisyntomic sheaves:
\[
\Fil^i_{\mathrm{N}} \otimes_{A} I \to \Fil^{i+1}_{\mathrm{N}} \to
\Fil^{i+1}_{\mathrm{H}}\dR_{-/(A/I)}^\wedge.
\]
Observe that 
\[
\mathrm{R\Gamma}(X_{\qsyn},\Fil^{l}_{\mathrm{H}}\dR_{-/(A/I)}^\wedge)
\cong \mathrm{R\Gamma}(X, \Fil^{l}_{\mathrm{H}}\dR_{-/(A/I)}^\wedge)
\]
lives in $D^{\geq l}(A/I)$,
and vanishes when $l > n$.
An easy induction gives what we want.

As for (3): we look at the map of filtered complexes
\[
\mathrm{R\Gamma}(X_{\qsyn}, \Fil^i_{\mathrm{N}}) \xrightarrow{\varphi} \mathrm{R\Gamma}(X_{\qsyn}, \Prism_{-/A} \otimes_{A} I^i)
\]
where the former is equipped with Nygaard filtration $\mathrm{R\Gamma}(X_{\qsyn}, \Fil^{i+\ast}_{\mathrm{N}})$
and the latter is equipped with $I$-adic filtration $\mathrm{R\Gamma}(X_{\qsyn}, \Prism_{-/A} \otimes_{A} I^{i+\ast})$.
Notice that both filtrations are complete.
Now \cite[Theorem 15.2.(2)]{BS19} implies that the cone of the $(i+\ast)$-th graded piece lives in $D^{>(i+\ast)}(A/I)$.
Hence we conclude that the cone of $\varphi$ lives in $D^{> i}(A)$.
Therefore the induced maps of degree at most $i$ cohomology groups are isomorphisms,
and the induced map in degree $i+1$ is injective.

Their derived mod $p^m$ counterparts are proved in exactly the same way.
\end{proof}

Now let us return to the situation of Breuil--Kisin prism $A= \s$. 
Recall that  $\Prism^{(1)}_n : = \Prism^{(1)}_{-/\s} \otimes_{\Z} \Z/ p ^n \Z$ and $\Fil^i_{\rm N }\Prism^{(1)}_n : = \Fil ^i_{\rm N }\Prism^{(1)}_{-/\s} \otimes_{\Z} \Z/ p ^n \Z$. 
Recall that $\m^i_n \coloneqq \rH ^i_{\Prism} (X_n / \s_n)$ 
and recall that Breuil--Kisin-filtration on $\varphi^*\m^i _n \cong \rH ^i_{\qsyn} (X, \Prism^{(1)}_n)$ is defined as the image of
$\psi \colon \m^i_n \to \varphi^* \m ^i _n$.

\begin{corollary}
\label{quasi-filtration identification}
For any $i \in \mathbb{N}$ and any $n \in \mathbb{N} \cup \{\infty\}$, there is a functorial commutative diagram:
\[
\xymatrix{
\rH^i_{\qsyn} (X, \Fil ^i_{\rm N } \Prism ^{(1)}_n) \ar[rr]^-{\varphi_i} \ar[rd] && \m ^i_n \ar[dl]^-{\psi}\\
&   \varphi^* \m ^i_n    &
}
\]
with $\varphi_i$ an isomorphism.
\end{corollary}

\begin{proof}
First let us justify the existence of the functorial commutative diagram.
We may work with affine formal schemes $Y = \Spf (R)$.
In this case, by the proof of \cite[Theorem 15.3 and Corollary 15.5]{BS19}, 
we see $\psi$ is constructed by the following (right-lower corner) diagram 
\[\xymatrix{   \tau^{\leq i}\RG_{\qsyn}(Y, \Fil ^i_{\rm N} \Prism ^{(1)})\ar[d] \ar@{-->}[rr]  &  & \tau^{\leq i } \RG_{\qsyn} (Y, \Prism) \otimes (E)^i \ar[ld] ^{\psi} \ar[d]\\\tau ^{\leq i} \RG_{\qsyn} (Y, \Prism^{(1)}) \ar[r]^{\sim} \ar@/_1pc/[rr]_{\varphi}  & \tau^{\leq i}  L\eta_E \RG_{\qsyn} (Y, \Prism)\ar[r] & \tau^{\leq i} \RG_{\qsyn} (Y, \Prism)} \]
Here the top row is the same as (truncation by $\leq i$ of) the following morphism 
\[
\mathrm{R\Gamma}(X_{\qsyn}, \Fil^i_{\mathrm{N}}) \xrightarrow{\varphi} \mathrm{R\Gamma}(X_{\qsyn}, \Prism_{-/A} \otimes_{A} I^i)
\]
appeared in Lemma \ref{general property of Nygaard filtration}. 
Derived mod $p^n$ gives the desired functorial commutative diagram.
By \Cref{general property of Nygaard filtration} (3) we know
that $\varphi_i$ is an isomorphism.
\end{proof}

\begin{remark}
In the context of filtered derived infinity categories, a filtration is nothing but an arrow.
Hence one could define two ``quasi-filtrations''\footnote{This terminology is suggested by S.~Mondal.}:
one being the Breuil--Kisin quasi-filtration: $\m^i_n \xrightarrow{\psi} \varphi^* \m ^i _n$;
another being the $i$-th Nygaard quasi-filtration: $\rH^i_{\qsyn} (X, \Fil ^i_{\rm N } \Prism ^{(1)}_n) \to \varphi^* \m ^i _n$.
Then the above is saying that these two quasi-filtrations are canonically identified via $\varphi_i$.
\end{remark}

Let us name the map
\[
\iota^{i,j}_n \colon \rH^i_{\qsyn} (X, \Fil ^j_{\rm N } \Prism ^{(1)}_n) \to \Fil^j_{\BK} \rH ^i_{\qsyn} (X, \Prism^{(1)}_n)
\]
for any pair of natural numbers $(i,j)$ and any $n \in \mathbb{N} \cup \{\infty\}$.
We have the following knowledge of the image of $\iota^{i,j}_n$ when $i \leq j$.

\begin{corollary}\label{cor-BK-Nygaard}
Let $i \leq j$.
Then we have an identification
\[
\mathrm{Im}(\iota^{i,j}_n) \cong \mathrm{Im}(\psi \colon \m^i_n \to \varphi^* \m^i_n) \cdot E^{j-i}.
\]
In particular, define $\widetilde{\FM^i_n} \coloneqq \FM^i_n/[u^\infty]$ and $\widetilde{\varphi^* \FM^i_n} \coloneqq \varphi^*\FM^i_n/[u^\infty]$,
we have an identification
\[
\mathrm{Im}(\widetilde{\iota}^{i,j}_n \colon \rH^i_{\qsyn} (X, \Fil ^i_{\rm N } \Prism ^{(1)}_n) \to \Fil^i_{\BK}\widetilde{\varphi^* \FM^i_n})
\cong \{x \in \widetilde{\varphi ^* \FM} | (1\otimes \varphi) (x) \in E(u)^{j} \widetilde{\FM^i_n}\}.
\]
\end{corollary}
\begin{proof}
The first statement follows from combining \Cref{general property of Nygaard filtration} (2) and \Cref{quasi-filtration identification}.
The second statement follows from the first statement and the fact that $\FM^i_n$ has height $i$.
\end{proof}

Below we make some primitive investigations of what happens without assuming $i \leq j$.

\begin{proposition}
\label{prop-fil-finite-difference}
Let $A= \s$ be the Breuil--Kisin prism.
For any triple $(i,j,n)$, the kernel and cokernel of $\iota^{i,j}_n$ above are finite.
\end{proposition}

\begin{proof}
Note that the kernel and cokernel of $\iota^{i,j}_n$ are finitely generated modules over $\s/(p^n)$.
We have a containment
\[
E(u)^j \cdot \Prism ^{(1)} \subset \Fil^j _N \Prism ^{(1)} \subset \Prism ^{(1)}
\]
of sheaves on $\mathrm{qSyn}_{A/I}$.
This shows that the map $\iota^{i,j}_n$ admits a section up to multiplication by $E(u)^j$, 
therefore the kernel and cokernel of $\iota^{i,j}_n$ are annihilated by $E(u)^j$.
If $n \in \mathbb{N}$, the kernel and cokernel of $\iota^{i,j}_n$ are finitely generated modules over $\s/(p^n, E(u)^j)$, hence finite.

If $n = \infty$, denote the map by $\iota^{i,j}$, we make the following 
\begin{claim}
\label{Nygaard-BK-fil-same-inverting-p}
The map $\iota^{i,j} \colon \rH^i_{\qsyn} (X, \Fil^j _N \Prism ^{(1)})[1/p] \to 
\Fil^j_{\BK} \varphi^* \FM^i[1/p]$ is an isomorphism.
\end{claim}
Granting this claim, the kernel and cokernel of $\iota^{i,j}$ are finitely generated modules over $\s/(E(u)^j)$ annihilated by a power
of $p$, hence finite.
\end{proof}
\begin{proof}[{Proof of \Cref{Nygaard-BK-fil-same-inverting-p}}]
First let us show the injectivity, which is the same as injectivity of
\[
\rH^i_{\qsyn} (X, \Fil^j _N \Prism ^{(1)})[1/p] \to \rH^i_{\qsyn} (X, \Prism ^{(1)})[1/p].
\]
To this end, we use the filtration $\Fil^{i,j}$ discussed in \Cref{subsec-NygaardFil}.
We claim a slightly stronger statement: the maps
\[
\rH^m_{\qsyn} (X, \Fil^{i,j} \Prism ^{(1)})[1/p] \to 
\rH^m_{\qsyn} (X, \Fil^{i,0} \Prism ^{(1)})[1/p]
\]
are injective for all $i \geq 0$.
The case of $i \geq j$ is trivial due to \Cref{general KON filtration properties} (2).
For the rest of $i$, we perform induction on descending $i$.
By five Lemma and \Cref{general KON filtration properties} (3), it suffices to know that the maps
\[
\rH^m (X, \Fil^{j-i}_{\mathrm{H}} \dR_{X/\mathcal{O}_K})[1/p] \to 
\rH^m (X, \dR_{X/\mathcal{O}_K})[1/p]
\]
are injective.
This injectivity is equivalent to the degeneration of the Hodge-to-de Rham spectral sequence for the rigid space $X_K$, which is a result
due to Scholze~\cite[Theorem 1.8]{Sch13}.

Next we show surjectivity by induction on $j$, the case of $j = 0$ being trivial. All we need to show is that the induced map
\[
\mathrm{Coker}\bigg(\rH^i_{\qsyn} (X, \Fil^{j+1} _N \Prism ^{(1)})[1/p]
\to \rH^i_{\qsyn} (X, \Fil^j _N \Prism ^{(1)})[1/p]\bigg)
\xrightarrow{\overline{\varphi}} \frac{E(u)^j \FM^i}{E(u)^{j+1} \FM^i}[1/p]
\]
is injective.
By the injectivity of $\iota^{i,j}[1/p]$ proved in the previous paragraph,
we can rewrite the left hand side as $\rH^i_{\qsyn} (X, \gr^j _N \Prism ^{(1)})[1/p]$.
Recall that $\FM^i[1/p]$ is finite free over $\s[1/p]$ (see~\Cref{lem-etale} (3)),
therefore the right hand side can be rewritten as
$\rH^i_{\qsyn}(X, \overline{\Prism})[1/p]\{j\}$, the $j$-th Breuil--Kisin twist
of the $i$-th Hodge--Tate cohomology of $X_K$.
By \cite[Theorem 15.2]{BS19}, we can identify the left hand side
further as the $j$-th conjugate filtration of the right hand side.
Now it follows from the degeneration of Hodge--Tate spectral sequence~\cite[Theorem 13.3]{BMS1} that $\overline{\varphi}$ is always injective.
\end{proof}

Below we exhibit an example illustrating the necessity of the $i \leq j$ assumption in \Cref{cor-BK-Nygaard}.

\begin{example}[{see~\cite[Section 4]{Li20}}]
Let $K$ be a ramified quadratic extension of $\mathbb{Q}_p$ and let $G$ be a lift of $\alpha_p$ over $\mathcal{O}_K$.
Denote the classifying stack of $G$ by $BG$.
Below we summarize previous study of various cohomologies of $BG$ as documented in \cite[4.6-4.10]{Li20}, following notation thereof.
\begin{enumerate}
\item The Breuil--Kisin prismatic cohomology ring of $BG$ is given by
\[
\rH^*_{\Prism}(BG/\s) \cong \s[\widetilde{u}]/(p \cdot \widetilde{u})
\]
where $\widetilde{u}$ has degree $2$.
\item The Hodge--Tate spectral sequence does not degenerate at $E_2$ page, but does degenerate at $E_3$ page, giving rise
to short exact sequences:
\[
0 \to \rH^{i+1}(BG, \wedge^{i-1} \mathbb{L}_{BG/\mathcal{O}_K}) \simeq \mathbb{F}_p \to \rH^{2i}_{\mathrm{HT}}(BG/\mathcal{O}_K) \simeq \mathcal{O}_K/(p)
\to \rH^{i}(BG, \wedge^{i} \mathbb{L}_{BG/\mathcal{O}_K}) \simeq \mathbb{F}_p \to 0
\]
for all $i > 0$.
\item The Hodge-to-de Rham spectral sequence does not degenerate at $E_1$ page, but does degenerate at $E_2$ page, giving rise
to short exact sequences:
\[
0 \to \rH^{2i-1}(BG, \mathbb{L}_{BG/\mathcal{O}_K}) \simeq \mathbb{F}_p \to \rH^{2i}_{\mathrm{dR}}(BG/\mathcal{O}_K) \simeq \mathcal{O}_K/(p)
\to \rH^{2i}(BG, \mathcal{O}_{BG}) \simeq \mathbb{F}_p \to 0
\]
for all $i > 0$.
\end{enumerate}
By \cite[Theorem 15.2]{BS19}, we have the following commutative diagram:
\[
\xymatrix{
\mathrm{R\Gamma}_{\qsyn}(BG/\s, \Prism^{(1)}) \ar[rr]^-{\varphi} \ar[d] & & \mathrm{R\Gamma}_{\qsyn}(BG/\s, \Prism) \ar[d] \\
\mathrm{R\Gamma_{dR}}(BG/\mathcal{O}_K) \ar[r] & \mathrm{R\Gamma}(BG, \mathcal{O}_{BG}) \ar[r] & \mathrm{R\Gamma_{HT}}(BG/\mathcal{O}_K)
}
\]
where $\varphi$ is the Frobenius on prismatic cohomology, vertical maps are derived modulo $E(u)$ reductions, the two arrows
on the bottom row are natural arrows appearing in Hodge-to-de Rham and Hodge--Tate spectral sequences respectively.
Looking at the degree $2$ cohomology together with (2) and (3) above, we see that $\varphi$ on $\rH^2_{\Prism}(BG/\s)$ 
is given by, up to a unit in $\s/p$, multiplication by $u \in \s/p$.
Since $\varphi$ is a map of $\mathbb{E}_{\infty}$-algebras, using (1) we see that $\varphi$ on $\rH^4_{\Prism}(BG/\s)$ 
is given by, up to a unit in $\s/p$, multiplication by $u^2 = E(u) \in \s/p$.
In particular, we see that $\Fil^1_{\BK} \rH^4_{\qsyn}(BG/\s, \Prism^{(1)})=\rH^4_{\qsyn}(BG/\s, \Prism^{(1)}) $ is the whole cohomology group.

On the other hand, we claim that the map 
\[
\rH^4_{\qsyn}(BG/\s, \Fil^1_{\mathrm{N}}\Prism^{(1)}) \to \rH^4_{\qsyn}(BG/\s, \Prism^{(1)})
\]
is not surjective.
Indeed we have a long exact sequence coming from the exact triangle $\Fil^1_{\mathrm{N}}\Prism^{(1)} \to \Prism^{(1)} \to \mathcal{O}_{BG}$
with the second arrow being the composition of derived modulo $E(u)$ followed by projection modulo first Hodge filtration.
Hence (3) above shows that the cokernel is exactly of length $1$.
This shows that $BG$ is a smooth proper stack counterexample for $(i,j,n) = (4,1,\infty)$.
Since all these cohomology groups are $p$-torsion, we see that this also provides a stacky
counterexample for $(i,j,n) = (3,1,1)$.

Finally let us use an approximation of $BG$ to get a smooth proper scheme counterexample.
By \cite[Subsection 4.3]{Li20} there is a smooth projective fourfold $X$ over $\mathcal{O}_K$ together with a map
$f \colon X \to BG$ such that the induced pullback map of Hodge cohomology is injective when total degree is no larger than $4$.
By functoriality of the formation of Breuil--Kisin filtrations,
we know that 
\[
\mathrm{Im}\bigg(f^* \colon \rH^4_{\qsyn}(BG/\s, \Prism^{(1)}) \to \rH^4_{\qsyn}(X/\s, \Prism^{(1)})\bigg)
\subset \Fil^1_{\BK} \rH^4_{\qsyn}(X/\s, \Prism^{(1)}).
\]
Lastly we claim that $f^*(\widetilde{u}^2) \in \rH^4_{\qsyn}(X/\s, \Prism^{(1)})$ is not in the image of $\rH^4_{\qsyn}(X/\s, \Fil^1_{\mathrm{N}}\Prism^{(1)})$.
To see this it suffices to compare the two exact sequences:
\[
\xymatrix{
\rH^4_{\qsyn}(BG/\s, \Fil^1_{\mathrm{N}}\Prism^{(1)}) \ar[r] \ar[d] & \rH^4_{\qsyn}(BG/\s, \Prism^{(1)}) \ar[r] \ar[d] & \rH^4_{\qsyn}(BG, \mathcal{O}_{BG}) \ar[d]^-{f^*} \\
\rH^4_{\qsyn}(X/\s, \Fil^1_{\mathrm{N}}\Prism^{(1)}) \ar[r] & \rH^4_{\qsyn}(X/\s, \Prism^{(1)}) \ar[r] & \rH^4_{\qsyn}(X, \mathcal{O}_{X})
}
\]
and invoke the fact that $f^*$ is injective by our choice of $X$.
This gives us smooth projective fourfold over $\mathcal{O}_K$ violating the conclusion of \Cref{cor-BK-Nygaard}
for $(i,j,n) = (4,1,\infty)$ or $(i,j,n) = (3,1,1)$.
\end{example}

\subsection{Torsion crystalline cohomology}
Now we are ready to discuss the structure of $\rH^i_{\cris} (X_n / S_n)$ via prismatic cohomology. 
First, we provide an application of the comparison
$\mathrm{R\Gamma}_{\Prism}(\mathcal{X}/\s) \otimes_{\s, \varphi} S
\cong \mathrm{R\Gamma}_{\cris}(\mathcal{X}/S)$,
which concerns the module structure of the cohomology of the latter.
We need some preparations.

\begin{lemma}
\label{coherence lemma}
The rings $S/p^n$ are coherent for all $n \in \mathbb{N}$.
\end{lemma}

We do not know if the ring $S$ itself is coherent.

\begin{proof}
We make an induction on $n$.
The starting case $n=1$: since $S$ is given by $p$-completely
adjoin the divided powers of the Eisenstein polynomial $E(u)$ to $\s$,
we see that $S/p$ is obtained by adjoining divided powers of $E(u) \equiv u^e$
to $\s/p = k[\![u]\!]$.
It is well-known that the result is 
$S/p \cong k[u]/u^{pe} \otimes_k k[u_1, u_2, \ldots]/(u_i^p)$
where $u_i$ is the image of the $p^i$-th divided powers of $E(u)$.
One checks this explicit algebra is coherent by noting
that any finitely generated ideal is generated by polynomials
involving only finitely many variables.

Now we do the induction, which largely relies on \cite[Lemma 3.26]{BMS1}.
Indeed the cited Lemma reduces us to showing the ideal $(p^n)/(p^{n+1})$
in $S/p^{n+1}$, when viewed as an $S/p$-module, is finitely presented.
But in fact $S$ is $p$-torsion free, hence the ideal 
$(p^n)/(p^{n+1})$ is free when viewed as an $S/p$-module with generator $p^n$.
\end{proof}

\begin{lemma}
\label{lem-comparison} 
Suppose that $ C^\bullet$ is a prefect
$\s_n$-complex. 
Then there exists an exact sequence of $S$-modules
	\begin{equation*}
	\xymatrix{
		0 \ar[r] & \rH^i (C^\bullet) \otimes_{\s} S
		\ar[r] & \rH^i(  C ^\bullet \otimes^\BL _\s S )
		\ar[r] & \Tor_1^{\s}(  \rH^{i+1} (C^\bullet), S ) \ar[r] & 0.
	}
	\end{equation*}
In particular $S$ has Tor-amplitude $1$ over $\s$ and the functor $M \mapsto \Tor_1^{\s}(M, S)$ is left exact.
\end{lemma}	

\begin{proof}
For the first claim, see the argument before  the proof of Theorem 5.4 in \cite{Cais-Liu-BK-crys} and replace $A_{\inf}$ (resp.~$A_{\cris}$) there by $\s$ (resp.~$S$). 
The fact that $S$ has Tor-amplitude $1$ over $\s$ follows from the Auslander--Buchsbaum formula
and torsion-freeness of $S$.
\end{proof}

\begin{proposition}
\label{Tor1 is finitely presented prop}
Let $M$ be a finitely generated Kisin module, then 
$\mathrm{Tor}_1^{\s}(M, \varphi_* S)$
is a finitely presented $S$ module.
\end{proposition}

\begin{proof}
Denote $N \coloneqq M[u^\infty]$ which is the maximal finite length $\s$-submodule
inside $M$.

We first show $\mathrm{Tor}_1^{\s}(N, \varphi_* S) \to \mathrm{Tor}_1^{\s}(M, \varphi_* S)$
is an isomorphism.
Since $S$ has Tor-amplitude $1$ over $\s$ by \Cref{lem-comparison},
it suffices to show the vanishing of $\mathrm{Tor}_1^{\s}(M/N, \varphi_* S)$.
Note that $M/N$ is an \'{e}tale Kisin module, we have a sequence
\[
0 \to (M/N)_{\mathrm{tor}} \to M/N \to (M/N)_{\mathrm{tf}} \to 0,
\]
where $(M/N)_{\mathrm{tor}}$ is a successive extension of $k[\![u]\!] = \s/p$
as $M/N$ is \'{e}tale, and $(M/N)_{\mathrm{tf}}$ is torsion-free.
Next observe that both these two structures are preserved under
base change along the Frobenius on $\s$.
Therefore it suffices to show $\mathrm{Tor}_1^{\s}(-, S)=0$
whenever the inputting $\s$-module is $\s/p$
or torsion-free. For the former case, it follows from the fact that $S$
has no $p$-torsion.
For the latter, consider the reflexive hull $M'^{\vee \vee}$ of the inputting module
$M' \subset M'^{\vee \vee}$, which is finite free as $\s$
is a regular Noetherian of dimension $2$.
Finally the desired vanishing of $\mathrm{Tor}_1^{\s}(M', S)$
follows from the left exactness of $\mathrm{Tor}_1$ against $S$ over $\s$:
see \Cref{lem-comparison}.

It suffices to show $\mathrm{Tor}_1^{\s}(N', S)$ is finitely presented
for any finite length $\s$-module,
which is the content of the next lemma.
\end{proof}

\begin{lemma}
Let $N$ be a finite length $\s$-module, then
$N \otimes_{\s} S$ and
$\mathrm{Tor}_1^{\s}(N, S)$ are finitely presented $S$-modules.
\end{lemma}

\begin{proof}
When $N = k \cong \s/(p,u)$, the statement for
$k \otimes_{\s} S = S/(p,u) = (S/p)/u$ and
$\mathrm{Tor}_1^{\s}(k, S) \cong S/p[u]$ follows from the fact that
$S/p$ is coherent (\Cref{coherence lemma}) and \cite[\href{https://stacks.math.columbia.edu/tag/05CW}{Tag 05CW (3)}]{stacks-project}.

Next we make an induction on the length of $N$.
By considering $N \twoheadrightarrow N/(p,u) \simeq k^{\oplus r}$,
we have a short exact sequence
$0 \to N' \to N \to k \to 0$, which induces a long exact sequence:
\[
0 \to \mathrm{Tor}_1^{\s}(N', S) \to \mathrm{Tor}_1^{\s}(N, S) \to
\mathrm{Tor}_1^{\s}(k, S) \to N' \otimes_{\s} S \to N \otimes_{\s} S
\to k \otimes_{\s} S \to 0.
\]
Induction hypotheses imply that all terms except of
$N \otimes_{\s} S$ and
$\mathrm{Tor}_1^{\s}(N, S)$ are finitely presented $S$-modules.
Note that the finite length assumption implies all modules are $S/p^N$-modules
for some sufficiently large $N$.
The coherence of $S/p^N$ (\Cref{coherence lemma}) and \cite[\href{https://stacks.math.columbia.edu/tag/05CW}{Tag 05CW (3)}]{stacks-project}
shows the boundary map $\mathrm{Tor}_1^{\s}(k, S) \to N' \otimes_{\s} S$
has finitely presented kernel and cokernel.
Now we use
\cite[\href{https://stacks.math.columbia.edu/tag/0519}{Tag 0519}]{stacks-project}
to finish the proof.
\end{proof}

\begin{proposition}
\label{fp proposition}
Let $\mathcal{X}$ be a smooth proper $p$-adic formal scheme over 
$\Spf(\mathcal{O}_K)$.
The $S/p^n$-module $\rH^i_{\cris}(\mathcal{X}_n/S_n)$
is finitely presented, for any integer $i$ and
any $n \in \mathbb{N} \cup \{\infty\}$.
\end{proposition}

Let us stress again that this already follows from \cite[Theorem 5.2]{BS19}.

\begin{proof}
The case of finite $n$ follows from \Cref{coherence lemma}:
the prismatic cohomology complex is a perfect complex, hence the comparison
\cite[Theorem 5.2]{BS19} or \Cref{comparing pris and crys}
shows the crystalline cohomology complex is also perfect
over the coherent ring $S/p^n$.
Therefore all of its cohomology modules are finitely presented as $S/p^n$-modules.

Now we turn to the case $n = \infty$.
By \Cref{lem-comparison} there is a short exact sequence:
\[
0 \to
\rH^i_{\Prism}(\mathcal{X}/\s) \otimes_{\s, \varphi} S
\to \rH^i_{\cris}(\mathcal{X}/S) \to 
\mathrm{Tor}_1^{\s}(\rH^{i+1}_{\Prism}(\mathcal{X}/\s), \varphi_* S) \to 0.
\]
Since prismatic cohomology complex is a perfect complex and the ring $\s$
is Noetherian, we know the term $\rH^i_{\Prism}(\mathcal{X}/\s) \otimes_{\s, \varphi} S$
is finitely presented.
Using \cite[\href{https://stacks.math.columbia.edu/tag/0519}{Tag 0519}]{stacks-project}
we are reduced to showing the term 
$\mathrm{Tor}_1^{\s}(\rH^{i+1}_{\Prism}(\mathcal{X}/\s), \varphi_* S)$
is finitely presented.
This in turn follows from \Cref{Tor1 is finitely presented prop} and
the fact that $\rH^{i+1}_{\Prism}(\mathcal{X}/\s)$ is a Kisin module:
see \Cref{cor-pris is Kisin}.
\end{proof}

Now we turn to the main result of our paper,
which concerns the Breuil-module structure of the crystalline cohomology.
Write $\FM^j_n: = \rH^j_\Prism ( X_n/ \s_n)$ and $\calM^j_n: = \rH^j_{\cris} (X_n /S_n)$. 

\begin{lemma} 
\label{lemma-mod I+S}
The following sequence
\[ 
0 \to \m^i _n / u \m ^i _n \to \calM^i _n / I^+S \to  \Tor_1^{\s}(  \m^{i +1}_n, \varphi_* S )/ I^+ S \to 0  
\]
is exact.
\end{lemma}

\begin{proof}
By derived mod $p ^n$ version of \Cref{global comparison}, we have $S \otimes_{\varphi, \s}^\BL \RG_\Prism (X_n/ \s_n  )\simeq \RG _{\cris} (X_n / S_n)$. So Lemma \ref{lem-comparison} yields an exact sequence
\begin{equation}\label{eqn-tensor-tor-2}
\xymatrix{0 \ar[r] &  S \otimes_{\varphi, \s }\m^i_n \ar[r] &  \calM ^i_n \ar[r] & \Tor _1^\s (\m^{i +1}_n, \varphi_* S) \ar[r] & 0}
\end{equation}
as $\varphi$ on $\s$ is finite flat.

We only need to show the above exact sequence remains left exact after modulo $I^+S$. To see this, note that 
$\RG_{\cris} (X_k / W_n (k)) \simeq \RG  _ \Prism (X_n / \s_n) \otimes^\BL _{\s} W(k) \simeq \RG_{\cris} (X_n /S_n) \otimes^\BL_{S} W(k )  $, where in the last identification we use the fact that Frobenius on $W(k)$ is an isomorphism.
Using the exact sequence \eqref{eqn-exact seq for length}, then we have the following commutative diagram 
\[ \xymatrix{  &  S/I ^+S \otimes_{\varphi, \s }\m^i_n \ar[d]^\wr \ar[r] &  \calM ^i_n/ I^+S \ar[d]\ar[r] & \Tor _1^\s (\m^{i +1}_n, \varphi_*S_n)/ I ^+S \ar[r]
\ar[d] & 0 \\   0 \ar[r] & \m^i_n/ u \m ^i _n  \ar[r] &  \rH^i (X_k/ W_n (k)) \ar[r]& \m^{i+1}_n [u]\ar[r] & 0. }\]
Since the left column is an isomorphism, we conclude that the top row is left exact as desired. 
\end{proof}

Recall in \Cref{def-Breuil-modules}, $\rH ^i_{\cris} (X_n / S_n)$ is defined to be a Breuil module if the quadruple
\[
\left (\rH^i_{\cris} (X_n /S_n), \rH ^i _{\cris} (X_n /S_n, \cI^{[i]}_{\cris}), \varphi_i, \nabla \right)
\]
constructed in \S\ref{subsection-stru of crys} is an object of
$\Mod_{S, \tor}^{\varphi, i , \nabla}$.
This condition is equivalent to the triple
\[
\left (\rH^i_{\cris} (X_n /S_n), \rH ^i _{\cris} (X_n /S_n, \cI^{[i]}_{\cris}), \varphi_i \right )
\]
being an object of $\Mod_{S, \tor}^{\varphi, i}$.

\begin{theorem} \label{thm-main-1} 
Let $n \in \mathbb{N}$ and assume $i\leq p-2$. Then $ \rH^j_\Prism(X_n/ \s_n)$ has no $u$-torsion for $j  = i ,i +1$ if and only if 
	$\rH^i _ {\cris}(X_n /S_n)$ is an Breuil module. 
When that happens we have $\underline \calM (\rH^i_\Prism(X_n/ \s_n)) \simeq \rH^i _ {\cris}(X_n /S_n)$ inside $\textnormal{Mod}^{\varphi, i}_{S, \tor}$. 
\end{theorem}


\begin{proof} Write $\FM^j_n: = \rH^j_\Prism ( X_n/ \s_n)$. Suppose that it has no $u$-torsion for $j = i,i  +1$. 
So $\FM ^i _n$  is an  \'etale Kisin module of height $i$ by Proposition \ref{prop-height}.
By the discussion of \S \ref{subsec-Breuilmodules}, we know $\calM^i_n : = \underline \calM (\FM ^i_n)$ is an object of $\text {Mod}^{\varphi, i}_{S,  \tor}$. 
By derived mod $p ^n$ version of \Cref{global comparison}, we have $S \otimes_{\varphi, \s}^\BL \RG_\Prism (X_n/ \s_n  )\simeq \RG _{\cris} (X_n / S_n)$. So Lemma \ref{lem-comparison} yields 
\begin{equation}\label{eqn-tensor-tor}
\xymatrix{0 \ar[r] &  S \otimes_{\varphi, \s }\rH ^i_\Prism ( X_n /\s_n) \ar[r] &  \rH ^i_ {\cris} (X_n / S_n) \ar[r] & \Tor _1^\s (\rH ^{i +1}_\Prism ( X_n /\s_n), \varphi_* S_n) \ar[r] & 0. }
\end{equation}
Our assumption that $ \FM^{i+1}_n $ has no $u$-torsion gives an isomorphism  $\iota: S \otimes_{\varphi, \s}\rH^j_\Prism(X_n/ \s_n) \simeq  \rH ^i_ {\cris} (X_n / S_n) $. 
Now we claim that $\iota$ induces a natural map $\iota^i: \Fil ^i \underline \calM (\FM^i _n) \to \rH ^i_{\cris}  (X_n /S_n, \mathcal I_{\cris}^{[i]})$
and both the source and target are natural submodules inside $\rH ^i_ {\cris} (X_n / S_n)$.
In particular, the $\iota^i$ is an injection.
To see this, we note that $\iota$ is induced by natural map $ \varphi^*\m^i_n \to \calM^i_n$ which we still denote by $\iota$. 
By Theorem \ref{lqsyn H and N filtration}, we have the following commutative diagram 
\[  \xymatrix{ \rH^{i-1}_{\qsyn} (X_n, \Prism^{(1)}_{-/ \s}/ \Fil^i_{\rm N} \Prism^{(1)}_ {-/\s} )\ar[r]^-{\alpha}\ar[d]^\wr & \rH ^i_{\qsyn} (X_n, \Fil^i_{\rm N} \Prism ^{(1)}_{-/ \s}) \ar[r]^-{\beta}\ar[d]& \rH ^i _{\qsyn}(X_n, \Prism ^{(1)}_{-/\s} )\ar[d] ^\iota \ar[r]& \cdots \\  \rH^{i-1}_{\qsyn} (X_n , \dR^\wedge_{-/\s}/ \Fil^i_{\rm H} \dR ^\wedge _{-/ \s} )\ar[r]^-{\alpha'} & \rH ^i_{\qsyn} (X_n, \Fil^i_{\rm H} \dR^\wedge _{-/ \s})\ar[r] ^-{\beta'} & \rH ^i_{\qsyn} (X_n, \dR^\wedge_{-/\s} )\ar[r] &\cdots }\]
with both rows being exact.
By Theorem \ref{lqsyn H and N filtration} (4), the left column is an isomorphism. As $\m ^i_n$ is assumed to have no $u$-torsion, 
\Cref{cor-BK-Nygaard} shows that $\beta$ is an injection. Thus $\alpha$ and hence $\alpha'$ are zero maps. So $\beta'$ is an injection. Therefore, Theorem \ref{Illusie--Bhatt} gives the following commutative diagram 
\[  \xymatrix{\Fil^i_{\BK}\varphi ^* \m^i_n \ar@{^{(}->}[r]\ar[d]   & \varphi^* \m^i_n \ar[d] ^\iota\\ \rH^i_{\cris} (X_n /S_n, \cI ^{[i]}_{\cris}) \ar@{^{(}->}[r]& \rH ^i_{\cris}  (X_n /S_n). }\]
Since $\iota: \underline{\calM} (\m^i_n) \xrightarrow{\simeq} \calM ^i_n = \rH ^i_{\cris}  (X_n /S_n) $ is an isomorphism  and 
$\Fil^i \underline{\calM} (\m ^i _n)$ is the  $S$-submodule of $\calM ^i_n$ generated by the image of $\Fil ^i \varphi^* \m ^i _n$ and $\Fil^i S \cdot \calM ^i _n$,
we see $\Fil^i \underline{\calM} (\m ^i _n) \subset \rH^i_{\cris} (X_n /S_n, \cI ^{[i]}_{\cris})$ via $\iota$. This shows $\iota$ induces an injection
$\iota : \Fil^i \underline{\calM} (\m ^i _n) \inj \rH^i_{\cris} (X_n /S_n, \cI ^{[i]}_{\cris})$.
 
Next we claim that $\iota^i$ is an isomorphism.   
After faithfully flat base changing along $S_n \to A_{\cris, n}: = A_{\cris}/ p ^n$, we are now working with $\mathcal{X} \coloneqq X_{\O_{\C}}$.
Now we need some some facts about the sheaf $\Z_p (h)$ on $\mathcal{X}_{\qsyn}$ defined in \cite[\S 7.4]{BMS2}. 
First  according to \cite[Theorem 10.1]{BMS2}, we have $\Z/ p ^n \Z (h) \simeq \tau^{\leq h } {\rm R}\psi_* (\Z/ p^n \Z (h)) $ where 
$\psi : (\mathcal{X}_{\C})_{\et } \to \mathcal{X}_{\et}$ is the natural map of \'etale sites. 
By \cite[Theorem F]{AMMN20}, when $h \leq p -2$, we have 
\[
\Z_p (h) \simeq \fib (\varphi_h -1: \Fil ^h_{\rm H} \dR^\wedge_{-/\Z_p} \to \dR^\wedge_{-/\Z_p}).
\]
Now \Cref{prop-connection for perfectoid} implies 
\[
\Z_p (h) \simeq \fib (\varphi_h -1: \Fil ^h_{\rm H} \dR^\wedge_{-/A_{\inf}} \to \dR^\wedge_{-/A_{\inf}}).
\]
Since ${\rm fib}$ commutes with derived mod $p^n$, we may apply $\otimes^{\BL}_{\Z} \Z/ p ^n \Z$ to this equation. 
Finally by Theorem \ref{Illusie--Bhatt}, we get the following exact sequence for $i \leq h \leq p-2$: 
\begin{equation}\label{eqn-etale-Fil-2}
\cdots\rH ^{i -1} _{\cris} (\mathcal{X}_n/ A_{\cris, n}) \to  \rH^i_ {\et} (\mathcal{X}_{\C} , \Z/ p^n \Z (h)) \to \rH ^ i _{\cris} (\mathcal{X}_n/ A_{\cris, n}, \cI_{\cris} ^{[h]}) \overset{\varphi_h -1}{\longrightarrow} \rH ^i _{\cris} (\mathcal{X}_n / A_{\cris, n }) 
\end{equation}

By Equation \eqref{eqn-etale-Fil-2} and Proposition \ref{prop-Galois-compatible}, we obtain the following commutative diagram: 
$$\xymatrix{ 0 \ar[r] & T_S (\underline{\calM}(\FM^i_n)) \ar[d]^\alpha   \ar[r]  &  A_{\cris}\otimes_S \Fil^i \underline \calM (\FM^i_n)  \ar[d]^-{1 \otimes \iota^i} \ar[r] & A_{\cris} \otimes_S \underline{\calM}  (\FM^i_n ) \ar[d]^\wr \ar[r] & 0 \\   
& \rH^i_{\et} (\mathcal{X}_{\C}, \Z/ p ^n \Z)(i)  \ar[r] ^- s  & \rH^i _{\cris} (\mathcal{X}_n/ A_{\cris, n}, \cI_{\cris}^{[i]})   \ar[r] & \rH^i _{\cris} (\mathcal{X}_n/ A_{\cris, n})  &  }$$
with both rows being exact.
Since $1 \otimes \iota^i$ is an injection, we see that the map $\alpha$ is also injective. 
Then $\alpha$ must be an isomorphism because $T_S (\underline \calM (\FM ^i_n)) \simeq T_\s (\m^i_n) (i) \simeq \rH ^i _{\et} (\mathcal X_{\C}, \Z/ p ^n \Z)(i)$
due to Proposition \ref{prop-Galois-compatible} and \cite[Theorem 1.8.(4)]{BS19}. 
Therefore $s$ is also injective.  Now by the snake lemma, we see that $\coker (1 \otimes \iota)= 0$ as required. 


Conversely, assume that $\calM ^i_n  : = \rH ^i _{\cris} ( X_n  / S_n )$ is an object in $\textnormal{Mod}^{\varphi, i}_{S, \tor}$ with $\Fil^i \calM^i _n = \rH ^i _{\cris} (X_n  /S_n  , \cI_{\cris}^{[i]})$. 
As before, we consider the base change $\mathcal{X} \coloneqq X_{\O_{\C}}$ and we still have a commutative diagram
$$\xymatrix{ 0 \ar[r] & T_S ({\calM^i_n}) \ar[d]^\alpha  \ar[r]  &  A_{\cris}\otimes_S \Fil^i  \calM^i  _n    \ar[d]^\wr \ar[r] & A_{\cris} \otimes_S {\calM^i _n }   \ar[d]^\wr \ar[r] & 0 \\   & \rH^i_{\et} (\mathcal{X}_{\C}, \Z/ p^n  \Z)(i)  \ar[r] ^-s  & \rH^i _{\cris} (\mathcal{X}_n / A_{\cris, n}, \cI_{\cris}^{[i]})   \ar[r] & \rH^i _{\cris} (\mathcal{X}_n / A_{\cris, n}).  &  }$$
The difference here is that the middle column is now an isomorphism, 
whereas the first column $\alpha$ is not known to be an isomorphism. 

First it is easy to see that $\alpha$ is an injection by chasing the diagram.
Now by Corollary \ref{cor-length-compare}, we have an inequality  
\[
\Len_ {W(k)} (\calM ^ i_n/I^+S) =\Len _{\Z} (T_S (\calM^i _n))\leq \Len _{\Z} (\rH ^i _{\et} (\mathcal{X}_{\C}, \Z/ p ^n \Z)).
\]
On the other hand, by the proof of Lemma \ref{lem-length of both fibers} and Lemma \ref{lemma-mod I+S}, we see that 
\[
\Len _{\Z} (\rH ^i _{\et} (\mathcal{X}_{\C}, \Z/ p ^n \Z)) \leq \Len _{W(k)} (\m ^i_n / u\m^i_n )\leq \Len_ {W(k)} (\calM ^ i_n/I^+S).
\]
Combining the above two inequalities, we see
\[
\Len _{\Z} (\rH ^i _{\et} (\mathcal{X}_{\C}, \Z/ p ^n \Z)) =  \Len _{W(k)} (\m ^i_n / u\m^i_n ) = \Len_ {W(k)} (\calM ^ i_n/I^+S).
\]
Now the proof of Lemma \ref{lem-length of both fibers} implies that $\m ^i _n$ has no $u$-torsion.
By the length equality, the injection
$\m ^i_n/ u \m ^i_n \hookrightarrow \calM^i _n/ I ^+ S$ in \Cref{lemma-mod I+S} is in fact an isomorphism.
and hence $\Tor_1^{\s} (\m^{i+1}_n, \varphi_* S)/I^+S = 0$.
It is easy to see that $\Tor_1^{\s} (\m^{i+1}_n, \varphi_* S)$ is a finitely generated $S$ module, applying Nakayama's lemma
yields 
$\Tor_1^{\s} (\m^{i+1}_n, \varphi_* S) = 0$. Therefore $\m ^{i+1} _n$ has no $u$-torsion by the following claim.  

Claim: If $\m$ is a $p^n$-torsion $\s$-module and $\Tor_1^{\s} (\m ,\varphi_* S)= 0$ then $\m$ has no $u$-torsion. 
To prove this, we first note that 
$\Tor_1^{\s}(-, \varphi_* S)$ is an left exact functor by \Cref{lem-comparison}. 
Secondly note that $\m$ has no $u$-torsion if and only if it has no $(u,p)$-torsion.
Let $\m' \subset \m$ be the submodule of $(u,p)$-torsions in $\m$.
The above discussion implies that $\Tor_1^{\s} (\m', \varphi_* S) = 0$.
Now by definition, we have $\m' \cong \oplus_{\Lambda} k$ as an $\s$-module, where $\Lambda$ is an indexing set.
One computes directly that 
\[
\Tor_1^{\s} (\m', \varphi_* S) = \oplus_{\Lambda} \Tor^1_{\s}(\s/(p,u), \varphi_* S) =
\oplus_{\Lambda} \Tor^1_{\s}(\s/(p,u^p), S) = \oplus_{\Lambda} \ker(S/p \xrightarrow{\cdot u^p} S/p).
\]
Since $\ker(S/p \xrightarrow{\cdot u^p} S/p)$ is nonzero ($u^{pe} = 0$ in $S/p$), the above computation implies
$\Lambda = \emptyset$, as claimed.
\end{proof}


\begin{corollary} \label{cor-small-Breuil}If $ei < p -1$ then $\rH ^{j}_\Prism (X_n / \s_n)$ has no $u$-torsion for $j =i , i+1$, 
and $\rH^i_{\cris} (X_n / S_n)$ is a Breuil module. 
\end{corollary}
\begin{proof} By Lemma \ref{lem-control-torsion} and Proposition \ref{prop-height},
we know that $\rH^{i}_\Prism (X_n / \s_n)$ has no $u$-torsion.
To show that $\rH ^{i +1}_\Prism (X_n / \s_n)$ has no $u$-torsion, we first consider the case that $n =1$. 
The main theorem of \cite{CarusoInvent} shows that $\rH^i_{\cris} (X_1/S_1)$ is a Breuil module when $n=1 $ and $ei < p -1$.
Then \Cref{thm-main-1} shows that $\rH ^{i+1}_\Prism (X_1/ \s_1)$ has no $u$-torsion. 

Let us prove by induction that $\m^{i +1}_n: = \rH^{i +1}_{\Prism} (X_n /\s_n ) $ has no $u$-torsion. 
We use the long exact sequence relating various $\m ^{i+1 }_n : = \rH^{i+1} _{\Prism}(X_n / \s_n)$:
\[ \cdots \to \m^{i}_{n -1}\overset f \longrightarrow \m ^{i+1} _1 \to \m^{i+1}_n \to \m^{i+1}_{n -1}\cdots .\] 
By induction, we may assume that $\m^{i+1}_{n -1}$ has no $u$-torsion. It suffices to prove that 
$\m ^{i+1}_1/ f (\m^{i}_{n -1})$ has no $u$-torsion. 
To that end, write $\mathfrak N \coloneqq f(\m ^{i }_{n-1})$ which has height $i$, 
$\m \coloneqq \m ^{i+1}_1$ which has height $i+1$ and  
$\mathfrak L \coloneqq \m ^{i+1}_1 / \mathfrak N$.
By construction we have the following exact sequence 
\[
0 \to \mathfrak N \overset f \longrightarrow \m \overset{g}{\longrightarrow} \mathfrak{L} \to 0.
\]
Let $\m' = g ^{-1} (\mathfrak{L}[u ^\infty])$. Then we obtain two exact sequences 
\[
0 \to \mathfrak{N } \to \m ' \to \mathfrak{L}[u ^\infty] \to 0 \text{ ~and~ }
0 \to \m' \to \m \to \mathfrak L / \mathfrak{L}[u ^\infty] \to 0 .
\]
The second sequence has all terms being \'etale Kisin modules. Since $\m$ has height $i+1$, 
we conclude that both $\m '$ and   $\mathfrak L / \mathfrak{L}[u ^\infty] $ has height $i+1$. 
Since  both $\mathfrak N$ and $\m'$ are \'etale and hence finite free over $\ku$.
This allows us to choose a basis $e_1, \dots , e_d$ of $\mathfrak N$ and a basis $ e'_1 , \dots , e'_d$ of $\m'$ so that $(e_1 , \dots, e_d ) = (e'_1, \dots , e'_d)\Lambda$, where $\Lambda= \mathrm{diag}(u^{a_1} \dots, u^{a_d})$ a diagonal matrix with  
$u^{a_j}$  on the main diagonal such that $a_1 \leq a_2 \leq \dots \leq a_d$. 
Let $A$ and $A'$ be the matrices of Frobenius for the corresponding basis. We easily see that following relation 
\[ 
\Lambda A = A' \varphi (\Lambda).
\]
Hence the last column of $A'\varphi(\Lambda)$ is divisible by $u ^{pa_d}$. 
Consequently, the last column of $A$ is divisible by $u^{(p-1) a_d}$.
But $\mathfrak N$ has height $i$, which means there exists a matrix $B$ with entries in $\ku$ so that $AB = BA = u ^{ei} I_d$. 
But this is impossible as $ei < p -1$,
unless $a_d =0$. This forces that $\Lambda = I_d$ and hence a posteriori $\mathfrak L$ has no $u$-torsion as desired. 
\end{proof}

\begin{remark}
Let $T$ be the largest integer satisfying $T \cdot e < p-1$, and let $n \in \mathbb{N}$.
It is a result of Min \cite[Lemma 5.1]{Min20} that $\rH^i_{\Prism}(X/\s)$ has no $u$-torsion when
$0 \leq i \leq T+1$.
By a similar argument, one can also show that $\rH^i _\Prism (X_n/\s_n)$ has no $u$-torsion for $0 \leq i \leq T$.
The slight improvement along this direction in \Cref{cor-small-Breuil} is the statement that $\rH^{T+1} _\Prism (X_n/\s_n)$
is also $u$-torsion free.
This would imply Min's result. 
As far as we can tell, Min's strategy does not give $u$-torsion freeness of $\rH^{T+1} _\Prism (X_n/\s_n)$.
\end{remark}


\begin{proposition}\label{prop-connection to etale via TS} 
Let $i \leq p-2$ be an integer.
Suppose that $\rH ^i _{\cris} (X_n /S_n, \cI^{[i]}_{\cris}) \to \rH^i_{\cris} (X_n /S_n) \eqqcolon \calM^i_n$
is injective, and denote its image as $\Fil^i\calM ^i_n$.
Assume furthermore that $\calM^i_n$ together with 
\[
(\Fil^i\calM ^i_n = \rH ^i _{\cris} (X_n /S_n, \cI^{[i]}_{\cris}), \varphi_i, \nabla)
\]
is an object of $\textnormal{Mod}^{\varphi,i,  \nabla}_{S, \tor}$. 
Then $T_S (\calM ^i_n)\simeq \rH^i_{\et} (X_{\bar{\eta}}, \Z/ p ^n \Z)(i)$ as $G_K$-representations.
\end{proposition}  
\begin{proof} Theorem \ref{thm-main-1} together with Proposition \ref{prop-Galois-compatible} have already shown the following isomorphisms
\[  T_S (\rH^i_{\cris} (X_n /S_n)) \overset {\iota_1}{\simeq} T_\s (\rH^i _{\Prism} (X_n / \s_n)) (i) \overset {\iota_2}{\simeq} 
\rH^i_{\et} (X_{\bar{\eta}}, \Z/ p ^n \Z) (i).\]
The main point here is to check these two isomorphisms $\iota_1 , \iota_2 $ here are compatible with $G_K$-actions. 
Let $\mathcal{X} \coloneqq X_{\mathcal{O}_\C}$.

First we have $A_{ \inf} \otimes_{\s} \m ^i_n \simeq \rH^i _{\Prism} (\mathcal{X}_n/ A_{\inf, n})$ which admits natural $G_K$-action. 
Since $A_{\inf}$ is a perfect prism, \cite[Theorem 1.8.(4)]{BS19} proves that $T_\s (\m^i _n)= (\rH^i_\Prism (\mathcal{X}_n/ A_{\inf, n}))^{\varphi=1 } \simeq \rH^i_{\et} (X_{\bar{\eta}} , \Z/ p ^n \Z) $ is compatible with $G_K$-action, as explained in Remark \ref{rem-G-compatible}. 
This concludes that $\iota_2$ is compatible with $G_K$-actions. 

Now Theorem \ref{global comparison} shows that the comparison isomorphism 
\[
\bar \iota: \rH^i_{\Prism} (\mathcal{X}_n/ A_{\inf, n} ) \otimes_{A_{\inf},\varphi } A_{\cris}\simeq \rH^i_{\cris} (\mathcal{X}_n / A_{\cris, n})
\]
is functorial. 
So $\bar \iota$ is compatible with natural  $G_K$-actions on the both sides. 
Also $\bar \iota$ is compatible with the isomorphism 
$\iota: \underline\calM (\m^i_n)\simeq \rH^i_{\cris} (X_n/ S_n)$. 
Applying Remark \ref{rem-compatible} here then concludes that 
\[
\iota_1:  T_S (\rH^i_{\cris} (X_n /S_n)) {\simeq} T_\s (\rH^i _{\Prism} (X_n / \s_n)) (i)
\]
is compatible with $G_K$-actions \emph{if} we define the $G_K$-action on $\rH ^i _ {\cris} (X_n / S_n) \otimes_S A_{\cris}$ via the identification
$\rH ^i _ {\cris} (X_n / S_n) \otimes_S A_{\cris}= \rH ^i_{\cris} (\mathcal{X}_n / A_{\cris, n})$. 
Recall the $G_K$-action on $\rH ^i _ {\cris} (X_n / S_n) \otimes_S A_{\cris} $ is defined via  Formula \eqref{eqn-action},
and we have showed in \S \ref{subsection-GK-action} that these two $G_K$-actions are the same. 
This proves that $\iota_1$ is also compatible with $G_K$-actions. 
\end{proof}

In the end of this subsection, we explain how our results are related to Fontaine--Messing theory in \cite{FontaineMessing} 
(see also \cite{Katovanishingcycles})
for a proper smooth \emph{formal} scheme $X$ over $W(k)$.
For any $n \geq 1$, the scheme $X_n$ is smooth proper over $\Spec(W_n (k))$. 
So when $0 \leq j \leq i \leq p -1$, the triple $ M^i: = (\rH^i_{\cris}(X_n / W_n  (k)), \rH ^i _{\cris}(X_n / W_n (k) , \cI ^{[j]}_{\cris} ), \varphi_j)$ is known to be a Fontaine--Laffaille data. 

Now let $i \leq p-2$ and one wants to show that $T_{\cris} (M^i)\simeq \rH ^i _{\et} (X_{\bar \eta}, \Z/ p ^n \Z)(i)$. 
We recall the construction of $T_{\cris} (M^i)$ in the following: write $\Fil^j M ^i : = \rH_{\cris}^i (X_n / W_n (k), \cI ^{[j]}_{\cris})$ and let \[\Fil ^i (A_{\cris}\otimes_{W(k)} M ^i) = \sum^i_{j = 0} \Fil^{j} A_{\cris} \otimes_{W(k)}\Fil^{i -j} M ^i \subset A_{\cris} \otimes_{W(k)} M ^i. \]Then one can define  
$\varphi_i : \Fil^i (A_{\cris} \otimes _{W(k)} M ^i) \to A_{\cris} \otimes_{W(k)} M ^i$ by $\varphi_i : = \sum\limits_{j = 0}^i \varphi_j|_{\Fil^j A_{\cris}} \otimes \varphi_{i -j}|_{\Fil^{i-j} M^i}$. Now 
$T_{\cris} (M^i): = (\Fil ^i (A_{\cris} \otimes M^i))^{\varphi_i=1} $. 

Let $\calM ^i : = \rH^i _{\cris} (X_n/S_n)$ which is an object of $\Mod^{\varphi, i, \nabla}_{S, \tor}$ by \Cref{cor-small-Breuil}. 
It is clear that the base change map $\iota: S \otimes_ {W(k)} M ^i \to \calM ^i$ is an isomorphism as $W(k) \to S$ is flat. Define 
$\Fil^i (S \otimes_{W(k)}M^i) : = \sum\limits_{j = 0}^i  \Fil^j S \otimes_{W(k)} \Fil^{i -j } M ^i\subset S \otimes_{W(k)}  M ^i$. Since $\Fil^j M^i$ is direct summand of $\Fil^{j -1} M ^i$, 
the natural map  $\Fil^i (S \otimes_{W(k)} M ^i) \to \rH^i _{\cris} (X_n /S_n, \cI^{[i]}_{\cris})$ induced by $\iota$ is injective.  Therefore, we obtain the following commutative diagram
\[\xymatrix{ 0 \ar[r]& T_{\cris} (M^i)\ar[r]\ar@{^{(}->}[d]&  \Fil^i (A_{\cris} \otimes_{W(k)} M^i) \ar[r]^-{\varphi_i-1}\ar@{^{(}->}[d]  &  A_{\cris} \otimes_{W(k)} M ^i \ar[d] ^ \wr  \\ 0 \ar[r] & T_S (\calM^i) \ar[r] &  \Fil^i (A_{\cris} \otimes_{S} \calM^i) \ar[r]^-{\varphi_i -1}  &  A_{\cris} \otimes_{S} \calM ^i.}\]
It is well-known from Fontaine--Laffaille theory that $\Len_{\Z} T_{\cris} (M^i) = \Len_{W(k)} M^i$. 
By Corollary \ref{cor-length-compare}, we know  $\Len_{\Z} (T_S(\calM^i)) = \Len_{W(k)} (\calM^i/ I^+ S) = \Len_{W(k)} M^i$. Therefore, the left column must be bijective.  By Proposition \ref{prop-connection to etale via TS}, it remains to check that the isomorphism $T_{\cris} (M^i) \to T_S(\calM^i)$ is compatible with $G_K$-actions. 
Since the $G_K$-action on $T_S(\calM^i)$ is the $G_K$-action on $A_{\cris} \otimes_S \calM^i$ via \eqref{eqn-action}, 
it suffices to show that $M^i \subset (\calM ^i)^{\nabla = 0}$,
which follows from Proposition \ref{prop-connection for perfectoid} (1). 
\begin{corollary}\label{cor-FM}  Fontaine-Messing theory in \cite{FontaineMessing} and \cite{Katovanishingcycles} accommodates $X$ being proper smooth formal scheme over $W(k)$.  
\end{corollary}
\section{Some calculations on $T_S$}
\subsection{Identification on \eqref{eqn-action-N} and \eqref{eqn-action}}\label{subsec-two-equations}
In this section, we show that \eqref{eqn-action-N} and \eqref{eqn-action} are the same. 
\begin{lemma}If we write $N^n = \sum\limits_{i = 1}^n A_{i, n} u ^i  \nabla^i$ then $A_{i, n+1} = A_{i-1, n} + i A_{i , n}$ and $A_{1, n}= A_ { n , n}=1$
\end{lemma}
\begin{proof} An easy induction on $n$ by $N = u \nabla$.
\end{proof}

Recall that $\gamma_i (t)$ denote the $i$-divided power of $t$.  
\begin{lemma} $\sum_{n \geq i} A_{i, n} \gamma_n (t) = \gamma_i ((e^t -1)).  $
\end{lemma}
\begin{proof}
It suffices to show that Taylor's expansion as functions of $t$ on both sides are equal. This is clear to see that the coefficients of $t^n$, which is first  nonzero term,  coincide  on the both sides. If we write 
$\gamma_i (e^t -1) =\sum_{n \geq i} B_{i, n} \gamma_n (t)$ then it suffices to show that $B_{i, n}$ satisfies the recursive formula: $B_{i, n+ 1} =   B     _{i- 1, n} + i B_ {i , n }$  for $n \geq i$. Note that 
\[\gamma_i (e^t -1)= \frac {1}{i !}\left (\sum_{m = 0}^i {i \choose m} (-1)^{i-m} e^{m t} \right ). \]
Therefore, $B_{i, n}= \frac{1}{i !} \left (\sum \limits_{m = 0}^i {i \choose m} (-1)^{i-m} m ^n \right )$. So to check that $ B _{i-1  , n }+ i B_{i,  n} = B_{i , n+1}$ is equivalent to check that 
\[ \frac{1}{(i -1) !}\left ( \sum_{m = 0} ^{i -1} {i-1 \choose m} (-1) ^{i-1 -m} m ^n +  \sum_{m = 0} ^{i } {i \choose m} (-1) ^{i -m} m ^n \right ) = \frac{1}{i!} \sum_{m = 0} ^{i } {i \choose m} (-1) ^{i-m} m ^{n +1}. \]
This follows that $i \left ( {i \choose m} - {i-1 \choose m} \right ) = {i \choose m } m . $
\end{proof}

Now by the above Lemmas,  $$ \sum _{n  =0}^\infty N ^n (x) \gamma_n  (\log (\underline{\epsilon}(\sigma)) =  \sum _{n  =0}^\infty \nabla ^n (x)  u ^n\gamma_n  ( e^{\log(\underline{\epsilon} (\sigma))}-1 ) = \sum _{n  =0}^\infty \nabla ^n (x)  \gamma_n  (u ( \underline{\epsilon} (\sigma)-1 )) =   \sum _{n  =0}^\infty \nabla ^n (x)  \gamma_n  (\sigma (u)-u ). $$
This proves that  \eqref{eqn-action-N} and \eqref{eqn-action} are the same. 
\subsection{$T_S$ and $T_{\st, \star}$}\label{subsec-two-T}
 In this subsection, we explain our functor $T_S$ and functor $T_{\st , \star}$ used in \cite{CarusoInvent} are the same. 
 For this purpose, we have to review period ring $\widehat A_{\st}$ from \cite{CarusoInvent}. 
 Let $\wh A_{\st}= A_{\cris} \langle X \rangle$ be  the $p$-adic completion of PD algebra of $A_{\cris}$. 
 We extend Frobenius $\varphi$ and filtration of $A_{\cris}$ to $\widehat  A_{\st}$ as follows: Let $\varphi (X) = (1+ X)^p -1$ and 
 \[
 \Fil ^i \wh A_{\st} \coloneqq \left \{\sum_{j = 0} ^\infty a_j \gamma_j (X), a_j \in \Fil^{\max \{i -j, 0\} } A_{\cris}, \lim_{j \to \infty }a_j \to 0  \  \ p-\text{adically}\right \}. 
 \]
 It is easy to see that we can define  $\varphi_r: \Fil ^r \wh A_{\st} \to \wh A_{\st} $ similar to that of $A_{\cris}$.  To extend $G_K$-action to $\wh A_{\st}, $ for any $g \in G_K , $ recall that  $\underline{\varepsilon} (g) = \frac{ g ([\upi])}{[\upi]} \in A_{\inf}$ defined before \eqref{eqn-action-N}. Set $g (X)= \underline{\varepsilon} (g) X+ \underline{\varepsilon} (g)-1$. Finally define an $A_{\cris}$-linear monodromy by set $N(X)= -(1 + X)$. We embed $S $ inside $\wh A_{\st }$ via $u \mapsto [\upi] (1+ X)^{-1}$. At this point, we have two embeddings $S\to \wh A_{\st}$:  $\iota_1: S\inj A_{\cris}  \subset \wh A_{\st } $ via $ u \mapsto [\upi]\in A_{\inf}$ and $\iota_2: S \inj \wh A_{\st} $ via $u \mapsto [\upi] (1+X)^{-1}$. We will use both embeddings in the following. Notice that there is an $A_{\cris}$-linear  projection $q: \wh A_{\st} \to A_{\cris}$ by sending $\gamma_i (X) \mapsto 0 $.  It is easy to check that  $q$ is compatible with filtration, Frobenius, $G_K$-actions, and \emph{both} embeddings $\iota_i : S \inj \wh A_{\st}$. Set $\beta: = \log (1+X)\in \wh A_{\st}$.  
 \begin{remark} Breuil--Caruso's theory set $N(1+X) = 1+X$. Our setting is different by $-1$ sign to fit our setting $\nabla (u) = 1$. There is  no difference for these two different settings up to changing some signs. 
 \end{remark}
 Given a Breuil module $\calM \in \Mod _{S , \tor}^{\varphi, N , h}$, we extends  filtration, $\varphi_h$, monodromy and $G_K$-actions to $\wh A_{\st} \otimes_{\iota_2, S } \calM$ as follows
\[ \Fil ^h  \wh A_{\st} \otimes_{\iota_2, S }  \calM =  \wh A_{\st} \otimes_{\iota_2, S } \Fil ^h \calM+ \Fil ^h \wh A_{\st} \otimes _{\iota_2, S} \calM. \]
For $a\otimes m \in \wh A_{\st} \otimes_{\iota_2, S}  \Fil ^h \calM$, set $\varphi _h (a\otimes m) = \varphi (a) \otimes \varphi _h (m )$, and for $a \otimes m \in  \Fil ^h \wh A_{\st} \otimes _{\iota_2 S} \calM$, set $\varphi _h (a\otimes m) = \varphi_h (a) \otimes \varphi_h (E^h m).$ It is easy to check these $\varphi_h$ are compatible with intersection so that $\varphi_h$ extends to $\wh A_{\st} \otimes_{\iota_2, S} \calM$. 
 We extend  $G_K$-action  from $\wh A_{\st}$ to $\wh A_{\st} \otimes_{\iota _2 , S} \calM  $ by acting  on $\calM$-trivially, and $N (a \otimes m) = N(a) \otimes m + a \otimes N(m), \forall a \in \wh A_{\st}, m \in \calM$. 
 Now set 
 $$T _{\st} (\calM): =  (\Fil ^h (\wh A_{\st} \otimes_{\iota_2, S }\calM))^{\varphi_h =1 , N = 0}. $$
 \begin{proposition} There is an isomorphism $T_ S (\calM)\simeq T_{\st , \star} (\calM)$ as $G_K$-representations. 
 \end{proposition}
\begin{proof} 
For $m \in \calM \subset \wh A_{\st} \otimes_{\iota_2, S} \calM$, set 
$m  ^\nabla \coloneqq \sum\limits_{i = 0} ^\infty N^i (m) \gamma_i(\beta)$ and  $\calM ^{\nabla} = \{ m  ^\nabla | m  \in \calM \}  \subset \wh A_{\st} \otimes_{\iota_2, S} \calM.$ To understand the map $\alpha : \calM  \to \calM^\nabla$,  consider the following diagram induced by $q: \wh A_{\st} \onto A_{\cris}$: 
\[\xymatrix{  & \calM ^\nabla \ar[ldd]^{\alpha'} \ar@{^{(}->}[dr] ^j & \\ \calM \ar@{=}[d] \ar@{^{(}->}[rr] \ar@{->>}[ur] ^\alpha   & &   \wh A_{\st} \otimes_{\iota_2, S} \calM \ar@{->>}[d] ^{q}   \\ \calM \ar@{^{(}->}[rr] &  &   A_{\cris} \otimes_{ S} \calM.}\]
Where $\alpha':  \calM^\nabla \to q(\calM)= \calM$ is induced by $q$. By definition of $\alpha$, it  is easy to show that $\alpha $ and $\alpha '$ is bijective. Also $\alpha $ is an isomorphism of $S$-modules  in the sense of $\alpha (\iota _2 (s) m  ) = \iota _1(s) \alpha (m)  $ for $s\in S $ and $m \in \calM$. Using that $N$ satisfies Griffiths transversality and diagram\eqref{eqn-N-diagram} together with facts that $\beta \in \Fil^1 \wh A_{\st}$ and $\varphi (\beta) = p \beta$,  a similar argument in Lemma \ref{lem-action-from_N} (by replacing $a$ with $\beta$) show that for any $m \in \Fil ^h \calM$ we have  $m ^ \nabla\in \Fil ^h (\wh A_{\st} \otimes_{\iota_2 , S} \calM)$ and 
$\varphi_h (m ^\nabla) = (\varphi_h (m))^\nabla$. In summary,  $\alpha: \calM \to \calM^\nabla$ is an isomorphism in $\Mod_{S, \tor} ^{\varphi, h}$ and injections $\calM \overset \alpha \simeq \calM ^\nabla \subset \wh A_{\st} \otimes_{\iota_2 , S} \calM$ are compatible with with filtrations and $\varphi_h$. 

Now consider the natural map $A_{\cris} \otimes_S \calM ^\nabla \in \wh A_{\st}\otimes_{\iota_2 , S} \calM$ induced by inclusion $j : \calM^\nabla \subset \wh A_{\st} \otimes_{\iota_2 , S} \calM$ which is still denoted by  $\tilde j$. Since $q \circ \tilde j $ is an isomorphism (as $A_{\cris}\otimes_S (q\circ \alpha)$ is an isomorphism), we conclude that 
$ A_{\cris} \otimes_S \calM ^\nabla   $ is an $A_{\cris}$-submodule of $\wh A_{\st} \otimes_{\iota_2 , S} \calM$ which is compatible with filtration and $\varphi _h$.  

Using that $N(\beta) = -1$, we easily see that $\calM^\nabla \subset (\wh A_{\st} \otimes_S \calM) ^{N = 0}$. In particular, we have an injection $\tilde j:  A_{\cris}\otimes_ S M ^ \nabla \to (\wh A_{\st} \otimes_{\iota_2, S} \calM) ^{N = 0}$ compatible with filtration and $\varphi_h$.
Therefore $\tilde j$ induces an injection 
\[ T_S(\calM) = (\Fil ^h (A _{\cris} \otimes_S \calM) )^{\varphi _h =1} \overset \alpha \simeq   ( \Fil ^h (A_{\cris}\otimes_S \calM ^\nabla) ) ^{\varphi _h =1} \subset  (\Fil ^h (\wh A_{\st} \otimes_{\iota_2, S} \calM)  )^{\varphi _h =1 , N = 0}= T_{\st, \star} (\calM).\]
To see this this injection is an isomorphism, we can d\'evissage to the case that $\calM$ is killed by $p$ because both $T_S$ and $T_{\st , \star}$ are exact functors (\cite[Corollary 2.3.10]{CarusoInvent}). In this case, it also well-known  that $\dim_{{\mathbb F}_p } T_{\st, \star}= \rank _{S_1} \calM = \dim_{{\mathbb F}_p } T_S(\calM)$. This establish the isomorphism 
$T_S (\calM) \simeq T_{\st , \star} (\calM)$. Finally, we check this isomorphism is compatible with $G_K$-actions. Note that $T_S(\calM)$ has $G_K$-action via \eqref{eqn-action-N}, while $T_{\st , \star}$ has $G_K$-action from that on $ \wh A_{\st} \otimes_{\iota_2 , S }\calM $ with trivial $G_K$-action on $\calM$. We have to show that $A_{\cris}\otimes_{S}\calM^\nabla$  has $G_K$-action as the subspace of $\wh A_{\st} \otimes_{\iota_2 , S }\calM$ is the same as that defined as \eqref{eqn-action-N}. But this easily follows from the formula that $m  ^\nabla \coloneqq \sum\limits_{i = 0} ^\infty N^i (m) \gamma_i(\beta)$ and $g (\beta) = \log (\underline \varepsilon(g)) + \beta$. 
\end{proof}

\bibliographystyle{amsalpha}
\bibliography{mybib}

\end{document}